\documentclass[a4paper]{article}%
\usepackage{amsmath}
\usepackage{amsfonts}
\usepackage{amssymb}
\usepackage{graphicx}%
\providecommand{\U}[1]{\protect\rule{.1in}{.1in}}
\newtheorem{theorem}{Theorem}

\newtheorem{corollary}{Corollary}

\newtheorem{lemma}{Lemma}

\newtheorem{proposition}{Proposition}
\newtheorem{remark}{Remark}

\newenvironment{proof}[1][Proof]{\noindent\textbf{#1.} }{\ \rule{0.5em}{0.5em}}
\begin{document}

\title{Toda flow with unbounded initial data}
\author{Shinichi Kotani$^{1}$, \ Jiahao Xu$^{2}$, \ Shuo Zhang$^{3}$\ }
\date{}
\maketitle

\begin{abstract}
A Toda flow is constructed starting from a certain class of unbounded initial
conditions including sequences growing with power order of less than $1$.
Unbounded ergodic sequences are allowed, and especially $\beta$-ensembles
matrix models in random matrix theory can be an initial data and they yiled
invariant measures for the flow.

\end{abstract}

\section{Introduction}

In 1967 M. Toda\footnotetext[1]{Dept. of Math.\ Osaka Univ. Toyonaka Japan
\ skotani@outlook.com} \footnotetext[2]{Dept. of Math. Nanjing Univ. Nanjing
China 178881559@qq.com}introduced\footnotetext[3]{Dept. of Math. Nanjing Univ.
Nanjing China shuozhang@smail.nju.edu.cn} an anharmonic system described by an
infinite dimensional system of equations%
\[%
\begin{array}
[c]{l}%
\overset{\cdot}{p}_{n}(t)=e^{-\left(  q_{n}(t)-q_{n-1}(t)\right)
}-e^{-\left(  q_{n+1}(t)-q_{n}(t)\right)  }\\
\overset{\cdot}{q}_{n}(t)=p_{n}(t)
\end{array}
\]
with $n\in\mathbb{Z}$. This equation is called Toda lattice, and it is
rewritten in an equivalent form%
\begin{equation}%
\begin{array}
[c]{l}%
\overset{\cdot}{a}_{n}(t)=a_{n}(t)\left(  b_{n}(t)-b_{n-1}(t)\right) \\
\overset{\cdot}{b}_{n}(t)=2\left(  a_{n+1}(t)^{2}-a_{n}(t)^{2}\right)
\end{array}
\label{1-1}%
\end{equation}
by Flaschka variables%
\[
a_{n}(t)=\dfrac{1}{2}e^{-\left(  q_{n}(t)-q_{n-1}(t)\right)  /2}\text{,
\ \ }b_{n}(t)=-\dfrac{1}{2}p_{n}(t)\text{.}%
\]
Our aim in this article is to solve the Cauchy problem for (\ref{1-1}) with
initial data as general as possible. It is known that this equation has
infinitely many invariants and has a unique solution for any bounded initial
data (refer \cite{t}).

Recently there appeared several papers treating (\ref{1-1}) with unbounded
initial data. E. K. Ifantis-K. N. Vlachou\cite{iv} considered (\ref{1-1}) on a
half axis $\mathbb{Z}_{+}$ and found a very simple form of solutions with
unbounded initial data. Let $H_{q}^{+}$ be the Jacobi operator on
$\mathbb{Z}_{+}$ with coefficient $q=\left\{  a_{n},b_{n}\right\}
_{n\in\mathbb{Z}_{+}}$%
\[
\left\{
\begin{array}
[c]{ll}%
\left(  H_{q}^{+}u\right)  _{n}=a_{n+1}u_{n+1}+a_{n}u_{n-1}+b_{n}u_{n} &
\text{for }n\geq2\\
\left(  H_{q}^{+}u\right)  _{1}=a_{2}u_{2}+b_{1}u_{1} & \text{for }n=1
\end{array}
\right.  \text{, }%
\]
and $\sigma_{+}$ be the spectral measure defined by%
\[
\left(  \left(  H_{q}^{+}-z\right)  ^{-1}\delta_{1},\delta_{1}\right)
=\int_{-\infty}^{\infty}\dfrac{\sigma_{+}\left(  d\lambda\right)  }{\lambda
-z}\text{.}%
\]
Let $q\left(  t\right)  =\left\{  a_{n}\left(  t\right)  ,b_{n}\left(
t\right)  \right\}  _{n\in\mathbb{Z}_{+}}$ be the solution to (\ref{1-1}) with
initial data $q$. Assuming $\int_{-\infty}^{\infty}e^{t\lambda}\sigma
_{+}\left(  d\lambda\right)  <\infty$ for any $t\in\mathbb{R}$, they showed
that the spectral measure for $H_{q\left(  t\right)  }^{+}$ are given by
$e^{2\lambda t}\sigma_{+}\left(  d\lambda\right)  /\int_{-\infty}^{\infty
}e^{2\lambda t}\sigma_{+}\left(  d\lambda\right)  $. Then, the coefficients
$q\left(  t\right)  $ is obtained by solving the inverse spectral problem.

On the other hand, A. Aggarwal\cite{a} considered the problem on the whole
$\mathbb{Z}$ and showed in Proposition 4.7 that if $q=\left\{  a_{n}%
,b_{n}\right\}  _{n\in\mathbb{Z}}$ satisfies%
\[
a_{n}>0\text{, }b_{n}\in\mathbb{R}\text{ and }a_{n}+\left\vert b_{n}%
\right\vert =O\left(  \left\vert n\right\vert ^{\alpha}\right)  \text{ as
}n\rightarrow\pm\infty
\]
for some\textit{\ }$\alpha\in\lbrack0,1)$\textit{, }then there exists a
solution\textit{ }$q\left(  t\right)  =\left\{  a_{n}\left(  t\right)
,b_{n}\left(  t\right)  \right\}  _{n\in\mathbb{Z}}$\textit{ }to\textit{
}(\textit{\ref{1-1}})\textit{ }with $q(0)=q$\textit{.} His approach is to
approximate the unbounded initial data $q$ by $q_{N}=\left.  q\right\vert
_{\left[  -N,N\right]  }$: the restriction of $q$ on a finite set $\left[
-N,N\right]  \subset\mathbb{Z}$.

Our approach employs Sato-Segal-Wilson (SSW) theory \cite{sa}, \cite{sew}
developed by ourselves in \cite{z}. The present authors constructed the Toda
flow (Toda hierarchy) on the space of bounded coefficients $\left\{
a_{n},b_{n}\right\}  _{n\in\mathbb{Z}}$ \cite{z}, and we are interested in
constructing solutions to (\ref{1-1}) for a certain class of unbounded initial
data $\left\{  a_{n},b_{n}\right\}  _{n\in\mathbb{Z}}$ including the case:%
\[
a_{n}\text{, }b_{n}=O\left(  \left\vert n\right\vert ^{\alpha}\right)  \text{
\ as }n\rightarrow\pm\infty\text{ with }\alpha<1\text{.}%
\]

To state the main results we prepare necessary notions and notations. For
$q=\left\{  a_{n},b_{n}\right\}  _{n\in\mathbb{Z}}$ with $a_{n}>0$, $b_{n}%
\in\mathbb{R}$ a Jacobi operator $H_{q}$ on $\mathbb{Z}$ is defined by%
\[
\left(  H_{q}u\right)  _{n}=a_{n+1}u_{n+1}+a_{n}u_{n-1}+b_{n}u_{n}\text{,}%
\]
which yields a self-adjoint operator on $\ell^{2}\left(  \mathbb{Z}\right)  $
under a certain condition on $q$. If $q$ is unbounded, then the spectrum
$\Sigma_{q}$ of $H_{q}$ is an unbounded closed set in $\mathbb{R}$. In this
case the domain $D_{+}\subset\mathbb{C}$ in \cite{z} is chosen so that
$D_{+}\supset\phi^{-1}\left(  \Sigma_{q}\right)  $ with $\phi\left(  z\right)
=z+z^{-1}$. Assuming $\Sigma_{q}=\mathbb{R}$ which is the biggest spectrum, we
take a $D_{+}$:%
\[
D_{+}=\left\{  z\in\mathbb{C}\text{; \ }\left\vert \operatorname{Im}%
\phi\left(  z\right)  \right\vert <\vartheta\right\}  \text{ \ for }%
\vartheta>0\text{.}%
\]
where $\vartheta$ is a positive number satisfying%
\begin{equation}
\vartheta<\left(  \dfrac{\pi}{2}\right)  ^{1/N}\sin\dfrac{\pi}{2N}%
\text{.}\label{1-2}%
\end{equation}
$N$ is a positive integer fixed throughout the paper, which will indicate the
degree of the Toda hierarchy. The domain $D_{+}$ satisfies%
\[
D_{+}\supset\mathbb{R}\cup\left\{  z\in\mathbb{C}\text{; }\left\vert
z\right\vert =1\right\}  \text{, \ \ }D_{+}\ni z\Longrightarrow\overline
{z}\text{, \ }z^{-1}\in D_{+}\text{.}%
\]
For $\vartheta$ in (\ref{1-2}) we have%
\begin{equation}
\cosh\phi\left(  z\right)  ^{N}\neq0\text{ \ on }\overline{D}_{+}%
\backslash\left\{  0\right\}  \text{.}\label{1-3}%
\end{equation}
Let $\omega\left(  x\right)  $ be a positive smooth function on $\mathbb{R}$
defined by the equation:%
\[
\operatorname{Im}\phi\left(  x+i\omega\left(  x\right)  \right)
=\vartheta\text{.}%
\]
The boundary curve $C$ of $D_{+}$ consists of two curves $C_{1}$ and $C_{2}$:%
\[
C_{1}=\left\{  x\pm i\omega\left(  x\right)  \text{; \ }x\in\mathbb{R}%
\right\}  \text{, \ }C_{2}=\left\{  z^{-1}\in\mathbb{C}\text{; \ }z\in
C_{1}\right\}  \text{.}%
\]
The complementary domain of $D_{+}$ is denoted by $D_{-}$, that is,
\[
D_{-}=\mathbb{C}\diagdown\left(  D_{+}\cup C\cup\left\{  0\right\}  \right)
\text{.}%
\]
The curve $C_{1}$ (regarding it as a closed simple curve in $\mathbb{C}%
\cup\left\{  \infty\right\}  $) is oriented anti-clockwise and the orientation
of $C_{2}$ is inherited by the map $z\rightarrow z^{-1}$, that is, $C_{2}$ is
oriented clockwise. Two Hardy spaces on the domains $D_{\pm}$ are defined as
closed subspaces of $L^{2}\left(  C\right)  :$%
\[
H\left(  D_{\pm}\right)  =L^{2}\text{-closure of }\left\{  \text{all rational
functions with no poles in }D_{\pm}\right\}  \text{.}%
\]
Since rational functions $r_{\pm}\in L^{2}\left(  C\right)  $ with no poles in
$D_{\pm}$ has a Cauchy representation:%
\[
r_{\pm}\left(  z\right)  =\pm\dfrac{1}{2\pi i}\int_{C}\dfrac{r_{\pm}\left(
\lambda\right)  }{\lambda-z}d\lambda\text{ \ for }z\in D_{\pm}\text{,}%
\]
one can see that, if $u\in H\left(  D_{\pm}\right)  $,
\[
u\left(  z\right)  =\pm\dfrac{1}{2\pi i}\int_{C}\dfrac{u\left(  \lambda
\right)  }{\lambda-z}d\lambda\text{ \ for }z\in D_{\pm}\text{,}%
\]
which implies%
\[
L^{2}\left(  C\right)  =H\left(  D_{+}\right)  \oplus H\left(  D_{-}\right)
\text{ \ (a direct sum),}%
\]
and the projections to $H\left(  D_{\pm}\right)  $ are given by%
\[
\mathfrak{p}_{\pm}f\left(  z\right)  =\pm\dfrac{1}{2\pi i}\int_{C}%
\dfrac{f\left(  \lambda\right)  }{\lambda-z}d\lambda\text{ \ for }z\in D_{\pm}%
\]
as bounded operators on $L^{2}\left(  C\right)  $ (refer to \cite{v}). Since
we have to multiply functions like $z^{n}e^{cz^{k}}$ to functions in
$L^{2}\left(  C\right)  $, the space $H\left(  D_{+}\right)  $ should be
enlarged:%
\[
H_{N,c}\left(  D_{+}\right)  =\rho_{c}(z)H\left(  D_{+}\right)  \text{ \ for
}c>0\text{ \ with }\rho_{c}(z)=\left(  \cosh\phi\left(  z\right)  ^{N}\right)
^{c}\text{.}%
\]
The norm on $H_{N,c}\left(  D_{+}\right)  $ is defined by%
\[
\left\Vert u\right\Vert _{N,c}=\sqrt{\int_{C}\left\vert u\left(
\lambda\right)  \right\vert ^{2}\left\vert \rho_{c}(\lambda)\right\vert
^{-2}\left\vert d\lambda\right\vert }\text{.}%
\]
The function $\rho_{c}(z)$ is well-defined as an analytic function on a
neighborhood of $\overline{D}_{+}\backslash\left\{  0\right\}  $ due to
(\ref{1-3}) and $\cosh\phi\left(  z\right)  ^{N}>0$ holds on $\mathbb{R}%
\backslash\left\{  0\right\}  $. It should be noted that%
\[
\left\{  z^{n}\right\}  _{n\in\mathbb{Z}}\subset H_{N,c}\left(  D_{+}\right)
\text{.}%
\]
\ Denoting the Euclidean norm in $\mathbb{C}^{2}$ by $\left\Vert
\cdot\right\Vert $, for a positive integer $N$,\ a set of symbols on $C$ is
introduced by%
\begin{equation}
\boldsymbol{A}_{N}\left(  C\right)  =\left\{
\begin{array}
[c]{l}%
\boldsymbol{a}=\left(  a_{1},a_{2}\right)  \text{; }\boldsymbol{a}\text{ is
bounded on }C\text{. For any }c>0\text{ there exists }\\
\text{an analytic bounded vector function }\boldsymbol{f}=\left(  f_{1}%
,f_{2}\right)  \text{ on }D_{+}\text{ such that }\\
\sup_{\lambda\in C}\left\Vert \boldsymbol{b}\left(  \lambda\right)
\right\Vert <\infty\text{ and}\sup_{\lambda\in C_{1}}\left(  \left\Vert
\partial_{\lambda}\boldsymbol{b}\left(  \lambda\right)  \right\Vert
+\left\Vert \partial_{\lambda}\widetilde{\boldsymbol{b}}\left(  \lambda
\right)  \right\Vert \right)  <\infty\text{,}\\
\text{where }\boldsymbol{b}\left(  \lambda\right)  \equiv\rho_{c}\left(
\lambda\right)  \left(  \boldsymbol{a}-\boldsymbol{f}\right)  \left(
\lambda\right)
\end{array}
\right\}  \text{.}\label{1-4}%
\end{equation}
A product $\boldsymbol{a}f$ with function $f$ and the Toeplitz operator with
symbol $\boldsymbol{a}$ are defined by%
\begin{equation}
\left\{
\begin{array}
[c]{l}%
\left(  \boldsymbol{a}f\right)  \left(  \lambda\right)  =a_{1}\left(
\lambda\right)  f\left(  \lambda\right)  +a_{2}\left(  \lambda\right)
Rf\left(  \lambda\right)  \\
\left(  T\left(  \boldsymbol{a}\right)  u\right)  \left(  \lambda\right)
=\left(  \boldsymbol{f}u\right)  \left(  \lambda\right)  +\mathfrak{p}%
_{+}\left(  \left(  \boldsymbol{a}-\boldsymbol{f}\right)  u\right)  \left(
\lambda\right)  \text{ for }u\in H_{N,c}\left(  D_{+}\right)
\end{array}
\text{,}\right.  \label{1-5}%
\end{equation}
where%
\begin{equation}
Rf\left(  \lambda\right)  =\lambda^{-1}f\left(  \lambda^{-1}\right)
\text{.}\label{1-6}%
\end{equation}
It should be noted that $T\left(  \boldsymbol{a}\right)  $ does not depend on
the choice of $\boldsymbol{f}$ and $T\left(  \boldsymbol{a}\right)  $ defines
a bounded operator on $H_{N,c}\left(  D_{+}\right)  $, since $R$ is unitary on
$H_{N,c}\left(  D_{+}\right)  $. To define a group acting on $\boldsymbol{A}%
_{N}\left(  C\right)  $ we introduce the scale $\mathrm{\delta}_{D_{+}}\left(
g\right)  $ of an analytic function $g$ on $D_{+}$ with no zeros:%
\begin{equation}
\mathrm{\delta}_{D_{+}}\left(  g\right)  =\sup\limits_{z\in D_{+}}\left\vert
\dfrac{g^{\prime}\left(  z\right)  }{g\left(  z\right)  }\theta\left(
z\right)  ^{-1}\right\vert \text{, where\ }\theta\left(  z\right)  =N\left(
\left\vert z\right\vert ^{N-1}+\left\vert z\right\vert ^{-N-1}\right)
\text{.}\label{1-7}%
\end{equation}
If the scale $\mathrm{\delta}_{D_{+}}\left(  g\right)  $ satisfies
$\mathrm{\delta}_{D_{+}}\left(  g\right)  <c$, it implies, for some constant
$c_{1}$,%
\[
\left\{
\begin{array}
[c]{c}%
\left\vert g\left(  z\right)  \right\vert \leq c_{1}e^{c\left\vert
z\right\vert ^{N}}\text{ \ \ if }z\in D_{+}\cap\left\{  \left\vert
z\right\vert >1\right\}  \\
\left\vert g\left(  z\right)  \right\vert \leq c_{1}e^{c\left\vert
z\right\vert ^{-N}}\text{ \ if }z\in D_{+}\cap\left\{  \left\vert z\right\vert
<1\right\}
\end{array}
\right.  \text{. \ }%
\]
A group $\Gamma_{N}\left(  D_{+}\right)  $ is defined by%
\begin{equation}
\Gamma_{N}\left(  D_{+}\right)  =\left\{
\begin{array}
[c]{l}%
g\text{; }g\text{ is an analytic function on some }D_{+}^{\prime}\supset
D_{+}\\
\text{with no zeros in }D_{+}^{\prime}\text{ and }\mathrm{\delta}%
_{D_{+}^{\prime}}\left(  g\right)  <\infty
\end{array}
\right\}  \text{,}\label{1-8}%
\end{equation}
where $D_{+}^{\prime}$ is a similar domain defined by another $\vartheta
^{\prime}>\vartheta$. Any rational function $r$ with no zeros nor poles in
$D_{+}\cup C$ is an element of $\Gamma_{N}\left(  D_{+}\right)  $, especially
$z^{n}\in\Gamma_{N}\left(  D_{+}\right)  $ for any $n\in\mathbb{Z}$. Unless
$g$ is bounded, $g\boldsymbol{a}\notin\boldsymbol{A}_{N}\left(  C\right)  $
for $g\in\Gamma_{N}\left(  D_{+}\right)  $ and $\boldsymbol{a}\in
\boldsymbol{A}_{N}\left(  C\right)  $, which bothers us always. We define a
Toeplitz operator for $g\boldsymbol{a}$ by%
\begin{equation}
T\left(  g\boldsymbol{a}\right)  u=g\boldsymbol{f}u+\mathfrak{p}_{+}\left(
g\left(  \boldsymbol{a}-\boldsymbol{f}\right)  u\right)  \in gH_{N,c}\left(
D_{+}\right)  \text{ \ for }u\in H_{N,c}\left(  D_{+}\right)  \text{,}%
\label{1-9}%
\end{equation}
which is possible due to $g\left(  \boldsymbol{a}-\boldsymbol{f}\right)  u\in
L^{2}\left(  C\right)  $ for an appropriate $\boldsymbol{f}$ and does not
depend on the choice of $\boldsymbol{f}$. One can show in Section2 that
$g^{-1}T\left(  g\boldsymbol{a}\right)  -T\left(  \boldsymbol{a}\right)  $ is
a trace class operator on $H_{N,c}\left(  D_{+}\right)  $. This makes it
possible to introduce a tau-function as follows. Set%
\begin{equation}
\boldsymbol{A}_{N}^{inv}\left(  C\right)  =\left\{
\begin{array}
[c]{c}%
\boldsymbol{a}\in\boldsymbol{A}_{N}\left(  C\right)  \text{; \ }T\left(
\boldsymbol{a}\right)  ^{-1}\text{ exists as a bounded}\\
\text{ operator on }H_{N,c}\left(  D_{+}\right)  \text{ for any }c>0
\end{array}
\right\}  \text{,}\label{1-10}%
\end{equation}
Then, we define the tau-function by%
\begin{equation}
\tau_{\boldsymbol{a}}\left(  g\right)  =\det\left(  g^{-1}T\left(
g\boldsymbol{a}\right)  T\left(  \boldsymbol{a}\right)  ^{-1}\right)  \text{
\ for }g\in\Gamma_{N}\left(  D_{+}\right)  \text{.}\label{1-11}%
\end{equation}
This determinant is defined on the Hilbert space $H_{N,c}\left(  D_{+}\right)
$, but it does not depend on $c$ nor $N$ (see Lemma 4.8 of \cite{k2}).
Hereafter, for a given $g\in\Gamma_{N}\left(  D_{+}\right)  $ the positive $c$
is chosen so that $c>\mathrm{\delta}_{D_{+}}\left(  g\right)  $ if we argue
something for the operator $T\left(  g\boldsymbol{a}\right)  $.

A notion of positivity for symbols will be very useful to show the
non-vanishing of $\tau_{\boldsymbol{a}}\left(  g\right)  $. Define%
\[
\overline{f}\left(  z\right)  =\overline{f\left(  \overline{z}\right)  }%
\]
\ for a function $f$ on a set $D$ satisfying $D\ni z\rightarrow\overline{z}\in
D$, and set
\begin{equation}
\left\{
\begin{array}
[c]{l}%
\Gamma_{N}^{\operatorname{real}}\left(  D_{+}\right)  =\left\{  g\in\Gamma
_{N}\left(  D_{+}\right)  \text{; \ }g=\overline{g}\right\}  \text{,}\\
\boldsymbol{A}_{N,+}^{inv}\left(  C\right)  =\left\{  \boldsymbol{a}%
\in\boldsymbol{A}_{N}^{inv}\left(  C\right)  \text{;}%
\begin{array}
[c]{l}%
\boldsymbol{a}=\overline{\boldsymbol{a}}\text{ \ and }\tau_{\boldsymbol{a}%
}\left(  r\right)  \geq0\text{ for any}\\
\text{rational function }r\in\Gamma_{N}^{\operatorname{real}}\left(
D_{+}\right)
\end{array}
\right\}  \text{,}\\
\boldsymbol{A}_{N,++}^{inv}\left(  C\right)  =\left\{  \boldsymbol{a}%
\in\boldsymbol{A}_{N,+}^{inv}\left(  C\right)  \text{; }\tau_{\boldsymbol{a}%
}\left(  z^{n}\right)  >0\text{ \ for any }n\in\mathbb{Z}\right\}  \text{.}%
\end{array}
\right.  \label{1-12}%
\end{equation}
Then, one can show $\tau_{\boldsymbol{a}}\left(  g\right)  >0$ for any
$g\in\Gamma_{N}^{\operatorname{real}}\left(  D_{+}\right)  $ and
$\boldsymbol{a}\in\boldsymbol{A}_{N,++}^{inv}\left(  C\right)  $ in
Proposition \ref{p1}. The property $\tau_{\boldsymbol{a}}\left(  g\right)  >0$
guarantees the invertibility of $g^{-1}T\left(  g\boldsymbol{a}\right)  $ on
$H_{N,c}\left(  D_{+}\right)  $ in spite of $g\boldsymbol{a}\notin
\boldsymbol{A}_{N}\left(  C\right)  $, which allows us to define
$\tau_{g\boldsymbol{a}}\left(  g_{1}\right)  $ with the property
$\tau_{g\boldsymbol{a}}\left(  g_{1}\right)  >0$ by%
\[
\tau_{g\boldsymbol{a}}\left(  g_{1}\right)  =\det\left(  \left(
g_{1}g\right)  ^{-1}T\left(  g_{1}g\boldsymbol{a}\right)  \left(
g^{-1}T\left(  g\boldsymbol{a}\right)  \right)  ^{-1}\right)
\]
for any $g_{1}\in$ $\Gamma_{N}^{\operatorname{real}}\left(  D_{+}\right)  $.
With this in mind, the coefficients $a_{n}\left(  g\boldsymbol{a}\right)  $,
$b_{n}\left(  g\boldsymbol{a}\right)  $ are defined by:%
\begin{equation}
\left\{
\begin{array}
[c]{l}%
a_{n}\left(  g\boldsymbol{a}\right)  =\sqrt{\dfrac{\tau_{g\boldsymbol{a}%
}\left(  z^{n}\right)  \tau_{g\boldsymbol{a}}\left(  z^{n-2}\right)  }%
{\tau_{g\boldsymbol{a}}\left(  z^{n-1}\right)  ^{2}}}\\
b_{n}\left(  g\boldsymbol{a}\right)  =\dfrac{\left.  \partial_{\varepsilon
}\tau_{g\boldsymbol{a}}\left(  z^{n}q_{\varepsilon^{-1}}\right)  \right\vert
_{\varepsilon=0}}{\tau_{g\boldsymbol{a}}\left(  z^{n}\right)  }-\dfrac{\left.
\partial_{\varepsilon}\tau_{g\boldsymbol{a}}\left(  z^{n-1}q_{\varepsilon
^{-1}}\right)  \right\vert _{\varepsilon=0}}{\tau_{g\boldsymbol{a}}\left(
z^{n-1}\right)  }%
\end{array}
\right.  \label{1-13}%
\end{equation}
with $q_{\zeta}\left(  z\right)  =\left(  1-\zeta^{-1}z\right)  ^{-1}$.

Now we introduce a space of initial data. For a Jacobi coefficient $q=\left\{
a_{n},b_{n}\right\}  _{n\in\mathbb{Z}}$ let $m_{\pm}$ be the Weyl functions
for $H_{q}$ and $\sigma_{\pm}$ be their spectral measures (refer to Appendix
for their definitions). Set%
\begin{equation}
Q_{N}=\left\{  q=\left\{  a_{n},b_{n}\right\}  _{n\in\mathbb{Z}}\text{;
\ }\int_{-\infty}^{\infty}e^{c\left\vert \lambda\right\vert ^{N}}\sigma_{\pm
}\left(  d\lambda\right)  <\infty\text{ for any }c>0\right\}  \text{.}
\label{1-14}%
\end{equation}
The integrability condition for $\sigma_{\pm}$ is equivalent to%
\[
\left(  \exp c\left(  H_{q}\right)  ^{N}\delta_{0},\delta_{0}\right)
<\infty\text{ \ for any }c>0\text{,}%
\]
where $\delta_{0}$ is an element of $\ell^{2}\left(  \mathbb{Z}\right)  $
satisfying $\delta_{0}\left(  n\right)  =0$ for $n\neq0$, $\delta_{0}\left(
0\right)  =1$. For such $\sigma_{\pm}$ the moment problem is unique and
$H_{q}$ is uniquely extendable as a self-adjoint operator on $\ell^{2}\left(
\mathbb{Z}\right)  $. Define a symbol $\boldsymbol{m}$ by%
\begin{equation}
\left\{
\begin{array}
[c]{l}%
\boldsymbol{m}\left(  z\right)  =\left(  \dfrac{zm(z)-1}{z^{2}-1},z^{2}%
\dfrac{z-m(z)}{z^{2}-1}\right)  \text{ \ \ with}\\
m(z)=\left\{
\begin{array}
[c]{ll}%
z+z^{-1}+a_{1}^{2}m_{+}\left(  z+z^{-1}\right)  & \text{if }\left\vert
z\right\vert >1\\
-a_{0}^{2}m_{-}\left(  z+z^{-1}\right)  +b_{0} & \text{if }\left\vert
z\right\vert <1
\end{array}
\right.  \text{.}%
\end{array}
\right.  \label{1-15}%
\end{equation}
Then, the property $\boldsymbol{m}\in\boldsymbol{A}_{N,++}^{inv}\left(
C\right)  $ is verified for any domain $D_{+}$. For a group%
\begin{equation}
\Gamma_{N}^{\operatorname{real}}=\bigcup_{D_{+}}\Gamma_{N}%
^{\operatorname{real}}\left(  D_{+}\right)  \label{1-16}%
\end{equation}
define $\mathrm{Toda}\left(  g\right)  $ on $Q_{N}$ by%
\begin{equation}
\mathrm{Toda}\left(  g\right)  q=\left\{  a_{n}\left(  g\boldsymbol{m}\right)
,b_{n}\left(  g\boldsymbol{m}\right)  \right\}  _{n\in\mathbb{Z}}\in
Q_{N}\text{.} \label{1-17}%
\end{equation}
Here we have to choose a suitable $\vartheta$ in (\ref{1-2}) for each
$g\in\Gamma_{N}^{\operatorname{real}}$ so that $g$ has no poles nor zeros in
$D_{+}$, and $\tau_{g\boldsymbol{m}}\left(  z^{n}\right)  $ does not depend on
the choice of $D_{+}$. The metrics on $Q_{N}$ and $\Gamma_{N}%
^{\operatorname{real}}$ are defined%
\[%
\begin{array}
[c]{l}%
\mathrm{d}\left(  q_{1},q_{2}\right)  =\sum_{k\in\mathbb{Z}}2^{-\left\vert
k\right\vert }\left(
\begin{array}
[c]{c}%
\left\vert a_{k}\left(  q_{1}\right)  -a_{k}\left(  q_{2}\right)  \right\vert
+\left\vert b_{k}\left(  q_{1}\right)  -b_{k}\left(  q_{2}\right)  \right\vert
\\
+\left\vert \left(  \left(  \cosh\left(  kH_{q_{1}}^{N}\right)  -\cosh\left(
kH_{q_{2}}^{N}\right)  \right)  \delta_{0},\delta_{0}\right)  \right\vert
\end{array}
\right)  \wedge1\text{,}\\
\mathrm{d}_{c}\left(  g_{1},g_{2}\right)  =\sqrt{%
%TCIMACRO{\dint _{C^{2}}}%
%BeginExpansion
{\displaystyle\int_{C^{2}}}
%EndExpansion
\left\vert \dfrac{g_{1}\left(  z\right)  ^{-1}g_{1}\left(  \lambda\right)
-g_{2}\left(  z\right)  ^{-1}g_{2}\left(  \lambda\right)  }{2\pi i\left(
\lambda-z\right)  }\right\vert ^{2}\dfrac{\left\vert d\lambda\right\vert
\left\vert dz\right\vert }{\left\vert \rho_{c}\left(  z\right)  \right\vert
^{2}\left\vert \rho_{c}\left(  \lambda\right)  \right\vert ^{2}}}\text{.}%
\end{array}
\]

A sequence $g_{n}\in\Gamma_{N}^{\operatorname{real}}$ is said to converge to
$g\in\Gamma_{N}^{\operatorname{real}}$ in the metric $\mathrm{d}_{c}\left(
g_{1},g_{2}\right)  $, if $\mathrm{\delta}_{D_{+}}\left(  g\right)  <c$,
$\sup_{n\geq1}\mathrm{\delta}_{D_{+}}\left(  g_{n}\right)  <c$ and
$\mathrm{d}_{c}\left(  g_{n},g\right)  \rightarrow0$ hold. Our main theorem is

\begin{theorem}
\label{t1}$\left\{  \mathrm{Toda}\left(  g\right)  \right\}  _{g\in\Gamma
_{N}^{\operatorname{real}}}$ defines a continuous flow on $Q_{N}$. Moreover,
for a real polynomial $p$ of $\mathrm{deg}$\textrm{\thinspace}$p\leq N$
$\mathrm{Toda}\left(  e^{-2t\widehat{p}}\right)  q$ solves the Cauchy problem
of Toda hierarchy%
\begin{equation}
\partial_{t}H_{q_{t}}=\left[  H_{q_{t}},p\left(  H_{q_{t}}\right)
_{a}\right]  \text{ \ with }q_{0}=q\text{,} \label{1-18}%
\end{equation}
where $\widehat{p}$ is the polynomial part of $p\left(  z+z^{-1}\right)  $ and
$\left(  \cdot\right)  _{a}$ denotes the anti-symmetrization of an operator
$\cdot$. Especially, if $p(z)=z$, $\mathrm{Toda}\left(  e^{-2tz}\right)  q$
gives a solution to the Toda lattice.
\end{theorem}

The condition in $Q_{N}$ is given indirectly in terms of the spectral measures
$\sigma_{\pm}$. A sufficient condition for this is

\begin{theorem}
\label{t2}For $q=\left\{  a_{n},b_{n}\right\}  _{n\in\mathbb{Z}}$ assume%
\begin{equation}
a_{n}\text{, }b_{n}=o\left(  \left\vert n\right\vert ^{1/N}\right)  \text{
\ as }n\rightarrow\pm\infty\text{.} \label{1-19}%
\end{equation}
Then, $q\in Q_{N}$ is valid, hence $\mathrm{Toda}\left(  g\right)  q$ is
well-defined for $g\in\Gamma_{N}^{\operatorname{real}}$.
\end{theorem}

Theorem \ref{t2} generalizes the result obtained by A. Aggarwal\cite{a}, in
Proposition 4.7. Although it is expected that $\mathrm{Toda}\left(  g\right)
q$ satisfies (\ref{1-19}) if so does $q\in Q_{N}$, its validity is open. The
problem to be solved here is to obtain a necessary condition for $q$ to
satisfy%
\[
\int_{-\infty}^{\infty}e^{c\left\vert \lambda\right\vert ^{N}}\sigma
_{+}\left(  d\lambda\right)  <\infty\text{ for any }c>0\text{.}%
\]

Theorem \ref{t2} has a corollary.

\begin{corollary}
\label{c1}Let $q_{\omega}=\left\{  a\left(  \theta^{n}\omega\right)  ,b\left(
\theta^{n}\omega\right)  \right\}  _{n\in\mathbb{Z}}$ be an ergodic Jacobi
sequence and assume%
\begin{equation}
\mathbb{E}\left(  a\left(  \omega\right)  ^{N}+\left\vert b\left(
\omega\right)  \right\vert ^{N}\right)  <\infty\text{.} \label{1-20}%
\end{equation}
Then, $q_{\omega}$ satisfies the condition in Theorem\ref{t2} for a.s.
$\omega$. Especially, if $q_{\omega}$ are independent random sequences with
distributions%
\[
\left\{
\begin{array}
[c]{l}%
P\left(  a\left(  \theta^{n}\omega\right)  <x\right)  =\dfrac{2}{\sigma
\Gamma\left(  \nu/2\right)  }%
%TCIMACRO{\dint _{0}^{x}}%
%BeginExpansion
{\displaystyle\int_{0}^{x}}
%EndExpansion
e^{-y^{2}/\sigma^{2}}y^{\nu-1}dy\text{ \ }(a\left(  \theta^{n}\omega\right)
^{2}\text{ }\sim\chi^{2}\text{-distribution})\\
P\left(  b\left(  \theta^{n}\omega\right)  <x\right)  =\dfrac{1}{\sqrt
{2\pi\sigma^{2}}}%
%TCIMACRO{\dint _{-\infty}^{x}}%
%BeginExpansion
{\displaystyle\int_{-\infty}^{x}}
%EndExpansion
e^{-\left(  y-m\right)  ^{2}/2\sigma^{2}}dy\text{ \ }\left(  \sim
\mathcal{N}\left(  m,\sigma^{2}\right)  \right)
\end{array}
\right.
\]
for $\sigma>0$, $\nu>0$, $m\in\mathbb{R}$, then $\mathrm{Toda}\left(
g\right)  q_{\omega}$ is well-defined for $g\in\Gamma_{N}^{\operatorname{real}%
}$ for any $N\geq1$. In this case the probability measure induced by this
$q_{\omega}$ satisfies $\left(  \mathrm{Toda}\left(  g\right)  \right)
^{\ast}\mu=\mu$ for any $g\in\Gamma_{N}^{\operatorname{real}}$.
\end{corollary}

Since $\mathrm{Toda}\left(  g\right)  q_{\omega}=\left\{  a_{n}^{\omega
}\left(  g\right)  ,b_{n}^{\omega}\left(  g\right)  \right\}  _{n\in
\mathbb{Z}}$ has the same distribution as that of $q_{\omega}$, $\left\{
a_{n}^{\omega}\left(  g\right)  ,b_{n}^{\omega}\left(  g\right)  \right\}
_{n\in\mathbb{Z}}$ are also independent and have the $\chi^{2}$-distribution
and the normal distribution, hence $\mathrm{Toda}\left(  g\right)  q_{\omega}$
satisfies (\ref{1-19}).

The proof of Theorem \ref{t1} can be explained by the Darboux transformation,
which has been an effective tool in solving integrable systems having Lax
pairs. If the spectrum $\Sigma_{q}$ of a Jacobi operator $H_{q}$ satisfies
$\sup\Sigma_{q}<\infty$ and let $\zeta\in\left(  \sup\Sigma_{q},\infty\right)
$, then one can regard $q_{\zeta}\left(  z\right)  =\left(  1-\zeta
^{-1}z\right)  ^{-1}\in$ $\Gamma_{N}^{\operatorname{real}}$. The coefficients
$\widehat{q}=\left\{  \widehat{a}_{n},\widehat{b}_{n}\right\}  $ associated
with $\mathrm{Toda}\left(  q_{\zeta}\right)  q$ are%
\[
\widehat{a}_{n}=a_{n-1}\left(  a_{n}\dfrac{f_{n}}{f_{n-1}}\right)
^{1/2}\left(  a_{n-1}\dfrac{f_{n-1}}{f_{n-2}}\right)  ^{-1/2}\text{,
\ }\widehat{b}_{n}=a_{n+1}\dfrac{f_{n+1}}{f_{n}}-a_{n}\dfrac{f_{n}}{f_{n-1}%
}+b_{n}%
\]
with $q=\left\{  a_{n},b_{n}\right\}  $ and $f_{n}=\zeta^{-n}\tau
_{z^{n}\boldsymbol{m}}\left(  q_{\zeta}\right)  \tau_{z^{n}\boldsymbol{m}%
}\left(  z^{-1}\right)  ^{-1/2}$ ($\boldsymbol{m}$ is in (\ref{1-15}) for
$q$), where $f_{n}$ satisfies%
\[
\left(  H_{q}f\right)  _{n}=a_{n+1}f_{n+1}+a_{n}f_{n-1}+b_{n}f_{n}=\left(
\zeta+\zeta^{-1}\right)  f_{n}\text{ and }f_{n}\neq0\text{.}%
\]
This defines a map $q\rightarrow\widehat{q}$ on $Q_{N}$, which is called
Darboux transformation for Jacobi operators. The commutativity of the Toda
flow implies that for any Toda lattice $q$ we have a new Toda lattice
$\widehat{q}$. This procedure to obtain a new solution for Toda lattice was
extensively employed by V.B. Matveev- M.A.Salle\cite{m1}. The novelty in the
present article is that one can solve Cauchy problem for Toda lattice by
iterating Darboux transformation infinitely many times, since%
\[
e^{tz}=\lim_{n\rightarrow\infty}q_{nt^{-1}}\left(  z\right)  ^{n}%
\Longrightarrow\mathrm{Toda}\left(  e^{tz}\right)  q=\lim_{n\rightarrow\infty
}\mathrm{Toda}\left(  q_{nt^{-1}}^{n}\right)  q\text{. }%
\]
If no semi-infinite gap exists, we take $\left(  q_{n\zeta}q_{n\overline
{\zeta}}\right)  ^{n}$ for $\zeta\in\mathbb{C}\backslash\mathbb{R}$ s.t.
$2\operatorname{Re}\zeta=t^{-1}$ as a sequence approximating $e^{tz}$.

The generality of the conditions in the Theorems can be examined by the works
of E. K. Ifantis-K. N. Vlachou. In \cite{iv} they proved%
\[
a_{n}=\sqrt{n}\text{, \ }b_{n}=\gamma n\text{ }\Longrightarrow\text{ }%
a_{n}\left(  t\right)  =\sqrt{n}e^{\gamma t}\text{, \ }b_{n}\left(  t\right)
=n\gamma-\gamma^{-1}\left(  1-e^{2\gamma t}\right)  \text{,}%
\]
which indicates that (\ref{1-19}) does not seem to be necessary for the
existence of solutions for $N=1$. But the condition $\int_{-\infty}^{\infty
}e^{t\lambda}\sigma_{+}\left(  d\lambda\right)  <\infty$ is satisfied in this
case. On the other hand, in \cite{iv2} they provided examples of Toda lattice
which are exploding at finite time. Although they considered the case where
$a_{n}\left(  t\right)  $, $b_{n}\left(  t\right)  $ take a form of
$a_{n}\left(  t\right)  =f(t)a_{n}$, $b_{n}\left(  t\right)  =g(t)b_{n}+h(t)$
for $n\geq1$ and $t\geq0$, this process can be applied also to Toda lattice on
the whole $\mathbb{Z}$, and among other solutions one has
\[
\left\{
\begin{array}
[c]{l}%
a_{n}\left(  t\right)  =c\dfrac{\sqrt{n\left(  n-1\right)  +\alpha n+\beta}%
}{1-2ct}\\
b_{n}\left(  t\right)  =c\dfrac{2n+\alpha}{1-2ct}%
\end{array}
\right.
\]
with constants $c>0$ and $\alpha$, $\beta$ satisfying $\beta>\left(
\alpha-1\right)  ^{2}/4$. This solution explodes at time $t=\left(  2c\right)
^{-1}$, which suggests that if the initial data $a_{n}$ violates the condition
(\ref{1-19}), then we may have explosion at finite time.

One of the motivations to study the Toda lattice with unbounded initial data
is to apply it to Toda lattice with random initial data, which are often
unbounded. The invariant measures appearing in Corollary\ref{c1} is the
$\beta$-ensembles introduced by I. Dumitriu-A. Edelman\cite{d} in relation to
random matrix theory. The $\beta$-ensembles models are known by H.
Spohn\cite{sp} to give invariant measures (Gibbs measures) for Toda lattice,
and its rigorous treatment was implemented by A. Aggarwal\cite{a}. In this
context D.A. Croydon-M. Sasada-S. Tsujimoto\cite{c} obtained an invariant
measure for the time discrete Toda lattice by applying Pitman's transformation
in probability theory, and R. Killip-J. Murphy-M. Visan\cite{kmv} constructed
a solution to the KdV equation starting from Gaussian white noise and proved
the induced measure is invariant under the equation.

In the KdV case SSW theory was extended in \cite{k} and \cite{k2} to construct
the KdV flow in a general space of initial conditions. In order to establish
Theorem \ref{t1} we follow the arguments in \cite{z} faithfully with \cite{k}
and \cite{k2} as our guide, but most of the parts proceed in a self-contained
way except the equation (\ref{1-18}). In the unbounded case symbols
$g\boldsymbol{a}$ are not generally bounded even if so is $\boldsymbol{a}$,
which makes arguments difficult. This obstacle is overcome by subtracting
compensators from symbols similarly as \cite{k}. In this case also the
invertibility of Toeplitz operators is crucial to prove the absence of
singulalities of the flow, and is achieved through the tau-functions.

\section{Tau-function}

In this paper the tau-functions are employed to show the invertibility of
Toeplitz operators and to express the Toda flow as in \cite{z}. The difference
between the present tau-functions and the previous ones is in the fact that
the operators $g^{-1}T\left(  g\boldsymbol{a}\right)  T\left(  \boldsymbol{a}%
\right)  ^{-1}-I$ in question are not immediately of trace class.

\subsection{A subclass of $\boldsymbol{A}_{N}^{inv}\left(  C\right)  $}

The invertibility of $T\left(  \boldsymbol{a}\right)  $ and more generally of
$T\left(  g\boldsymbol{a}\right)  $ is a key factor for the Toda flow to have
no singuralities. Therefore, we start with giving a concrete subclass of
$\boldsymbol{A}_{N}^{inv}\left(  C\right)  $. The argument proceeds similarly
to that of \cite{z}.

We first show that $H\left(  D_{+}\right)  $ is dense in $H_{N,c}\left(
D_{+}\right)  \left(  =\rho_{c}H\left(  D_{+}\right)  \right)  $.

\begin{lemma}
\label{l2-1}$H\left(  D_{+}\right)  $ is dense in $H_{N,c}\left(
D_{+}\right)  $.
\end{lemma}

\begin{proof}
To show the statement we approximate $\rho_{c}$ by bounded analytic functions
on $D_{+}$. Define%
\[
R_{n}\left(  \zeta\right)  =\left(  1+\dfrac{\zeta^{N}/n}{1+\zeta^{2N}%
/n}\right)  ^{n}=\exp\left(  \int_{0}^{1}\dfrac{\zeta^{N}dt}{1+\left(
\zeta^{2N}+\zeta^{N}t\right)  /n}\right)  \equiv\exp I_{n}\left(
\zeta\right)  \text{.}%
\]
Clearly, $\lim_{n\rightarrow\infty}R_{n}\left(  \zeta\right)  =$ $e^{\zeta
^{N}}$ holds. For $c_{1}$, $c_{2}>0$ set%
\[
U=\left\{  \zeta\in\mathbb{C}\text{; }\left\vert \operatorname{Im}%
\zeta\right\vert <c_{1}\right\}  \text{, \ }E=\left\{  \zeta\in U\text{;
}\left\vert \operatorname{Re}\zeta\right\vert >c_{2}\right\}  \text{.}%
\]
The real part and imaginary part of the exponent $I_{n}\left(  \zeta\right)  $
are%
\begin{equation}
\left\{
\begin{array}
[c]{l}%
\operatorname{Re}I_{n}\left(  \zeta\right)  =%
%TCIMACRO{\dint _{0}^{1}}%
%BeginExpansion
{\displaystyle\int_{0}^{1}}
%EndExpansion
\dfrac{\operatorname{Re}\zeta^{N}+\left\vert \zeta^{N}\right\vert ^{2}\left(
\left\vert \operatorname{Re}\zeta^{N}\right\vert +t\right)  /n}{\left\vert
1+\left(  \zeta^{2N}+\zeta^{N}t\right)  /n\right\vert ^{2}}dt\\
\operatorname{Im}I_{n}\left(  \zeta\right)  =\left(  \operatorname{Im}%
\zeta^{N}\right)  \left(  1-\left\vert \zeta^{N}\right\vert ^{2}/n\right)
%TCIMACRO{\dint _{0}^{1}}%
%BeginExpansion
{\displaystyle\int_{0}^{1}}
%EndExpansion
\dfrac{dt}{\left\vert 1+\left(  \zeta^{2N}+\zeta^{N}t\right)  /n\right\vert
^{2}}%
\end{array}
\right.  \text{.} \label{2-1}%
\end{equation}
If $c_{2}$ is sufficiently large, one has $\operatorname{Re}\left(  \zeta
^{2N}+\zeta^{N}t\right)  >0$ for $\zeta\in E$ and $t\in\left[  0,1\right]  $,
hence%
\[
\left\vert R_{n}\left(  \zeta\right)  ^{\pm1}\right\vert \leq\exp\left(
\left\vert \operatorname{Re}\zeta^{N}\right\vert +\int_{0}^{1}\dfrac
{\left\vert \zeta^{N}\right\vert ^{2}\left(  \left\vert \operatorname{Re}%
\zeta^{N}\right\vert +t\right)  }{\left\vert \operatorname{Re}\left(
\zeta^{2N}+\zeta^{N}t\right)  \right\vert ^{2}}dt\right)  \text{.}%
\]
Since%
\[
c_{3}\equiv\sup_{\zeta\in E,\text{ }t\in\left[  0,1\right]  }\dfrac{\left\vert
\zeta^{N}\right\vert ^{2}\left(  \left\vert \operatorname{Re}\zeta
^{N}\right\vert +t\right)  }{\left\vert \operatorname{Re}\left(  \zeta
^{2N}+\zeta^{N}t\right)  \right\vert ^{2}}<\infty
\]
holds, one has $\left\vert R_{n}\left(  \zeta\right)  \right\vert \leq
e^{c_{3}}e^{\left\vert \operatorname{Re}\zeta^{N}\right\vert }$ \ on $E$. On
the other hand, the boundedness of $U\backslash E$ implies%
\[
\sup_{n\geq n_{1}\text{, }\zeta\in U\backslash E}\left\vert R_{n}\left(
\zeta\right)  ^{\pm1}\right\vert <\infty
\]
for sufficiently large $n_{1}\geq1$. Therefore, for any $n\geq n_{1}$ one has
$\left\vert R_{n}\left(  \zeta\right)  ^{\pm1}\right\vert \leq c_{4}%
e^{\left\vert \operatorname{Re}\zeta^{N}\right\vert }$ \ on $U$ with a
constant $c_{4}$. Define rational functions $r_{n}$ by $r_{n}\left(  z\right)
=R_{n}\left(  \phi\left(  z\right)  \right)  $. Recalling $\phi\left(
D_{+}\right)  =U$ with $c_{1}=\vartheta$, we obtain the estimates
\begin{equation}
\left\vert r_{n}\left(  z\right)  ^{\pm1}\right\vert \leq c_{4}\left\vert
\cosh\phi\left(  z\right)  ^{N}\right\vert =c_{4}\left\vert \rho_{1}\left(
z\right)  \right\vert \text{ \ on }D_{+}\text{.} \label{2-2}%
\end{equation}
The property $r_{n}\left(  z\right)  +r_{n}\left(  z\right)  ^{-1}\neq0$ on
$D_{+}$ should be examined. If $c_{2}$ is sufficiently large,
$\operatorname{Re}\zeta^{N}>0$ or $\operatorname{Re}\zeta^{N}<-1$ holds for
$\zeta\in E$, which implies $\operatorname{Re}I_{n}\neq0$ by (\ref{2-1}).
Observing $R_{n}\left(  \zeta\right)  +R_{n}\left(  \zeta\right)  ^{-1}=2\cosh
I_{n}\left(  \zeta\right)  $ and the zeros of $\cosh z$ are $\left\{  \left(
k+1/2\right)  \pi i\right\}  _{k\in\mathbb{Z}}\subset i\mathbb{R}$, one has
$\cosh I_{n}\left(  \zeta\right)  \neq0$ for $\zeta\in E$. The analytic
functions $\cosh I_{n}\left(  \zeta\right)  $ converges $\cosh\zeta^{N}$
uniformly on a compact set $\overline{U\diagdown E}$, and $\cosh\zeta^{N}%
\neq0$ holds on $\overline{U}$ by the definition of $U$ ($c_{1}=\vartheta$).
Thus, there exists a sufficiently large $n_{1}$ such that $\cosh I_{n}\left(
\zeta\right)  \neq0$ on $\overline{U\diagdown E}$ for any $n\geq n_{1}$.
Consequently, one has $r_{n}\left(  z\right)  +r_{n}\left(  z\right)
^{-1}\neq0$ on $D_{+}$ for any $n\geq n_{1}$. For $u\in H\left(  D_{+}\right)
$ set
\[
u_{n}\left(  z\right)  =\left(  \dfrac{r_{n}\left(  z\right)  +r_{n}\left(
z\right)  ^{-1}}{2}\right)  ^{c}\rho_{c}\left(  z\right)  ^{-1}u\left(
z\right)  \text{.}%
\]
Since $r_{n}\left(  z\right)  $ are bounded on $U$ for each fixed $n\geq
n_{1}$, one sees $u_{n}\in H\left(  D_{+}\right)  $. And, in the identity%
\[
\left\Vert u_{n}-u\right\Vert _{H_{N.c}}^{2}=\int_{C}\left\vert \left(
\dfrac{r_{n}\left(  \lambda\right)  +r_{n}\left(  \lambda\right)  ^{-1}}%
{2}\right)  ^{c}\rho_{c}\left(  \lambda\right)  ^{-1}-1\right\vert
^{2}\left\vert u\left(  \lambda\right)  \right\vert ^{2}\dfrac{\left\vert
d\lambda\right\vert }{\left\vert \rho_{c}\left(  \lambda\right)  \right\vert
^{2}}%
\]
applying the dominated convergence theorem due to (\ref{2-2}), we obtain
$\lim_{n\rightarrow\infty}\left\Vert u_{n}-u\right\Vert _{H_{N.c}}=0$.\bigskip
\end{proof}

This Lemma ensures the eligibility of the definition (\ref{1-5}) of $T\left(
\boldsymbol{a}\right)  $. Originally $T\left(  \boldsymbol{a}\right)  $ is
defined on $H\left(  D_{+}\right)  $ by%
\[
T\left(  \boldsymbol{a}\right)  u=\mathfrak{p}_{+}\left(  \boldsymbol{a}%
u\right)  \text{,}%
\]
which is equal to%
\begin{equation}
T\left(  \boldsymbol{a}\right)  u=\boldsymbol{f}u+\mathfrak{p}_{+}\left(
\left(  \boldsymbol{a}-\boldsymbol{f}\right)  u\right)  \label{2-3}%
\end{equation}
due to $\boldsymbol{f}u\in H\left(  D_{+}\right)  $. This formula enables us
to extend the definition of $T\left(  \boldsymbol{a}\right)  $ to
$H_{N,c}\left(  D_{+}\right)  $ thanks to Lemma \ref{l2-1}. This also shows
that the operator $T\left(  \boldsymbol{a}\right)  $ is independent of the
choice of the compensators $\boldsymbol{f}$.

Let $\boldsymbol{M}_{N}\left(  C\right)  $ be the set of all symbols
$\boldsymbol{m}=\left(  m_{1},m_{2}\right)  $ satisfying%

\begin{equation}%
\begin{tabular}
[c]{ll}%
(i) & $m_{j}$, $m_{j}^{\prime}$ are bounded analytic on $D_{-}$, and
$m_{j}=\overline{m}_{j}$ for $j=1$, $2$,\\
(ii) & $\boldsymbol{m}\in\boldsymbol{A}_{N}\left(  C\right)  $,\\
(iii) & $m_{1}-1$, $\widetilde{m}_{1}-1\in H\left(  D_{-}\right)  $ and
$m_{2}$, $\ \widetilde{m}_{2}\in H\left(  D_{-}\right)  $,\\
(iv) & $\inf_{z\in D_{-}}\left\vert m_{1}(z)\widetilde{m}_{1}(z)-m_{2}%
(z)\widetilde{m}_{2}(z)\right\vert >0$,
\end{tabular}
\ \label{2-4}%
\end{equation}
where $\widetilde{f}\left(  z\right)  =f\left(  z^{-1}\right)  $. For
$\boldsymbol{m}=\left(  m_{1},m_{2}\right)  $, $\boldsymbol{n}=\left(
n_{1},n_{2}\right)  \in\boldsymbol{M}_{N}\left(  C\right)  $ define a product%
\[
\boldsymbol{m}\cdot\boldsymbol{n}=\left(  m_{1}n_{1}+m_{2}\widetilde{n}%
_{2},m_{1}n_{2}+m_{2}\widetilde{n}_{1}\right)  \text{.}%
\]
Since (i), (ii), (iii) are valid for $\boldsymbol{m}\cdot\boldsymbol{n}$ and
\begin{align*}
&  \left(  \boldsymbol{m}\cdot\boldsymbol{n}\right)  _{1}(z)\widetilde{\left(
\boldsymbol{m}\cdot\boldsymbol{n}\right)  }_{1}(z)-\left(  \boldsymbol{m}%
\cdot\boldsymbol{n}\right)  _{2}(z)\widetilde{\left(  \boldsymbol{m}%
\cdot\boldsymbol{n}\right)  }_{2}(z)\\
&  =\left(  m_{1}(z)\widetilde{m}_{1}(z)-m_{2}(z)\widetilde{m}_{2}(z)\right)
\left(  n_{1}(z)\widetilde{n}_{1}(z)-n_{2}(z)\widetilde{n}_{2}(z)\right)
\end{align*}
holds, one has $\boldsymbol{m}\cdot\boldsymbol{n}\in$ $\boldsymbol{M}%
_{N}\left(  C\right)  $. Clearly, the inclusion $\boldsymbol{M}_{N}\left(
C\right)  \subset\boldsymbol{A}_{N}\left(  C\right)  $ is valid.

\begin{lemma}
\label{l2-2}(i) For $\boldsymbol{m}$, $\boldsymbol{n}\in$ $\boldsymbol{M}%
_{N}\left(  C\right)  $ it holds that%
\[
T\left(  \boldsymbol{m}\cdot\boldsymbol{n}\right)  =T\left(  \boldsymbol{m}%
\right)  T\left(  \boldsymbol{n}\right)  \text{.}%
\]
(ii) There exists a bounded analytic function $f$ on $D_{+}$ such that the
compensators $f_{j}$ for $\boldsymbol{m}\in$ $\boldsymbol{M}_{N}\left(
C\right)  $ satisfy%
\begin{equation}
\left\{
\begin{array}
[c]{l}%
\sup_{\lambda\in C}\left\vert \rho_{c}\left(  M-f\right)  \left(
\lambda\right)  \right\vert <\infty\\
\sup_{\lambda\in C_{1}}\left(  \left\vert \partial_{\lambda}\left(  \rho
_{c}\left(  M-f\right)  \left(  \lambda\right)  \right)  \right\vert
+\left\vert \partial_{\lambda}\left(  \rho_{c}\left(  M-\widetilde{f}\right)
\left(  \lambda\right)  \right)  \right\vert \right)  <\infty\\
\inf_{z\in D_{+}}\left\vert f\left(  z\right)  \right\vert >0
\end{array}
\right.  \text{,} \label{2-5}%
\end{equation}
where $M=m_{1}\widetilde{m}_{1}-m_{2}\widetilde{m}_{2}$. Then, the symbol
$\boldsymbol{m}^{-1}=\left(  \widetilde{m}_{1}/M,-m_{2}/M\text{ }\right)  $
satisfies $\boldsymbol{m}^{-1}\in$ $\boldsymbol{M}_{N}\left(  C\right)  $ and
$T\left(  \boldsymbol{m}\right)  ^{-1}=T\left(  \boldsymbol{m}^{-1}\right)  $.
\end{lemma}

\begin{proof}
The identity in (i) is verified as follows. Let $\boldsymbol{m}=\left(
m_{1},m_{2}\right)  $, $\boldsymbol{n}=\left(  n_{1},n_{2}\right)
\in\boldsymbol{M}_{N}\left(  C\right)  $. Since $m_{j}$ is bounded analytic on
$D_{-}$, $m_{j}H\left(  D_{-}\right)  \subset H\left(  D_{-}\right)  $ for
$j=1$, $2$ are valid. Hence, for $u\in H\left(  D_{+}\right)  $ one has%
\[
\mathfrak{p}_{+}\left(  m_{i}n_{j}u\right)  =\mathfrak{p}_{+}\left(
m_{i}\mathfrak{p}_{+}n_{j}u\right)  +\mathfrak{p}_{+}\left(  m_{i}%
\mathfrak{p}_{-}n_{j}u\right)  =\mathfrak{p}_{+}\left(  m_{i}\mathfrak{p}%
_{+}n_{j}u\right)  \text{.}%
\]
Therefore, it holds that%
\begin{align*}
T\left(  \boldsymbol{m}\right)  T\left(  \boldsymbol{n}\right)  u  &
=\mathfrak{p}_{+}\left(  m_{1}\mathfrak{p}_{+}\left(  n_{1}u+n_{2}Ru\right)
+m_{2}R\mathfrak{p}_{+}\left(  n_{1}u+n_{2}Ru\right)  \right) \\
&  =\mathfrak{p}_{+}\left(  m_{1}\left(  n_{1}u+n_{2}Ru\right)  +m_{2}%
\mathfrak{p}_{+}\left(  \widetilde{n}_{1}Ru+\widetilde{n}_{2}u\right)  \right)
\\
&  =\mathfrak{p}_{+}\left(  m_{1}\left(  n_{1}u+n_{2}Ru\right)  +m_{2}\left(
\widetilde{n}_{1}Ru+\widetilde{n}_{2}u\right)  \right) \\
&  =\mathfrak{p}_{+}\left(  \left(  \boldsymbol{m}\cdot\boldsymbol{n}\right)
u\right)  =T\left(  \boldsymbol{m}\cdot\boldsymbol{n}\right)  u\text{,}%
\end{align*}
where where we have used the property $R\mathfrak{p}_{+}=\mathfrak{p}_{+}R$.
Then, Lemma \ref{l2-1} and the boundedness in $H_{N,c}\left(  D_{+}\right)  $
of $T\left(  \boldsymbol{m}\right)  $, $T\left(  \boldsymbol{n}\right)  $
imply (i).

The bounded analytic vector function $\left(  \widetilde{f}_{1}/f,-f_{2}%
/f\right)  $ is a candidate of a compensator for $\boldsymbol{m}^{-1}$. Under
the condition (\ref{2-5}), it is not difficult to verify that it works really
as a compensator, which leads us to (ii). The conditions (iii) and (iv) of
(\ref{2-4}) for $\boldsymbol{m}^{-1}$ are easily verified to hold, hence
$\boldsymbol{m}^{-1}\in\boldsymbol{M}_{N}\left(  C\right)  $. This together
with $\boldsymbol{m}^{-1}\cdot\boldsymbol{m}=\boldsymbol{m}\cdot
\boldsymbol{m}^{-1}=\left(  1,0\right)  \equiv\boldsymbol{1}$ yields (ii) of
the lemma.\bigskip
\end{proof}

Later we will show $\boldsymbol{m}$ in (\ref{1-15}) is an element of
$\boldsymbol{M}_{N}$ and satisfies (\ref{2-5}).

\subsection{Basic properties of tau-functions}

In this section we show that $T\left(  g\boldsymbol{a}\right)  $,
$\tau_{\boldsymbol{a}}\left(  g\right)  $ are well-defined and they are
continuous under suitable norms.

For $g\in\Gamma_{N}\left(  D_{+}\right)  $ and $\boldsymbol{a}\in
\boldsymbol{A}_{N}\left(  C\right)  $ the operator $T\left(  g\boldsymbol{a}%
\right)  $ is defined in (\ref{1-9})%
\[
T\left(  g\boldsymbol{a}\right)  u=g\boldsymbol{f}u+\mathfrak{p}_{+}\left(
g\left(  \boldsymbol{a}-\boldsymbol{f}\right)  u\right)  \text{ \ for }u\in
H_{N,c}\left(  D_{+}\right)  \text{,}%
\]
where the compensator $\boldsymbol{f}$ is so chosen that $\rho_{c+c^{\prime}%
}\left(  \boldsymbol{a}-\boldsymbol{f}\right)  =O\left(  1\right)  $ for
$c^{\prime}=\mathrm{\delta}_{D_{+}}\left(  g\right)  $. This definition of
$T\left(  g\boldsymbol{a}\right)  $ does not depend on the choice of
compensators $\boldsymbol{f}$. Indeed, the right hand side turns to \
\[
g\boldsymbol{f}u+\mathfrak{p}_{+}\left(  g\left(  \boldsymbol{a}%
-\boldsymbol{f}\right)  u\right)  =g\rho_{c^{\prime}}^{-1}\boldsymbol{f}%
\rho_{c^{\prime}}u+\mathfrak{p}_{+}\left(  g\rho_{c^{\prime}}^{-1}\left(
\boldsymbol{a}-\boldsymbol{f}\right)  \rho_{c^{\prime}}u\right)  =T\left(
g\rho_{c^{\prime}}^{-1}\boldsymbol{a}\right)  \rho_{c^{\prime}}u\text{,}%
\]
where $g\rho_{c^{\prime}}^{-1}\boldsymbol{a}\in\boldsymbol{A}_{N}\left(
C\right)  $ since $g\rho_{c^{\prime}}^{-1}$ is bounded. Here, the operator
$T\left(  g\rho_{c^{\prime}}^{-1}\boldsymbol{a}\right)  $ maps
$H_{N,c+c^{\prime}}\left(  D_{+}\right)  $ to $H_{N,c+c^{\prime}}\left(
D_{+}\right)  $, thus (\ref{2-3}) implies that the left hand side is
independent of the choice of $\boldsymbol{f}$.

For the definition of the tau-functions $\tau_{\boldsymbol{a}}\left(
g\right)  $ it is necessary to have a different expression of $T\left(
g\boldsymbol{a}\right)  $:
\begin{align*}
&  T\left(  g\boldsymbol{a}\right)  u\left(  z\right) \\
&  =g\boldsymbol{f}u\left(  z\right)  +g\mathfrak{p}_{+}\left(  \left(
\boldsymbol{a}-\boldsymbol{f}\right)  u\right)  \left(  z\right)  +\int
_{C}\dfrac{g\left(  \lambda\right)  -g\left(  z\right)  }{2\pi i\left(
\lambda-z\right)  }\left(  \boldsymbol{a}-\boldsymbol{f}\right)  u\left(
\lambda\right)  d\lambda\\
&  =gT\left(  \boldsymbol{a}\right)  u\left(  z\right)  +\int_{C}%
\dfrac{g\left(  \lambda\right)  -g\left(  z\right)  }{2\pi i\left(
\lambda-z\right)  }\left(  \boldsymbol{a}-\boldsymbol{f}\right)  u\left(
\lambda\right)  d\lambda\text{.}%
\end{align*}
If $\mathrm{\delta}_{D_{+}}\left(  g\right)  <c$ and a compensator
$\boldsymbol{f}$ is chosen such that $\rho_{c}\left(  \boldsymbol{a}%
-\boldsymbol{f}\right)  u\in L^{2}\left(  C\right)  $, one has
\[
\int_{C}\dfrac{g\left(  \lambda\right)  -g\left(  z\right)  }{2\pi i\left(
\lambda-z\right)  }\rho_{c}\left(  \lambda\right)  ^{-1}\mathfrak{p}%
_{+}\left(  \rho_{c}\left(  \boldsymbol{a}-\boldsymbol{f}\right)  u\right)
\left(  \lambda\right)  d\lambda=0\text{ \ for }z\in D_{+}\text{,}%
\]
since, for fixed $z\in D_{+}$, the function of $\lambda$:%
\[
\dfrac{g\left(  \lambda\right)  -g\left(  z\right)  }{2\pi i\left(
\lambda-z\right)  }\rho_{c}\left(  \lambda\right)  ^{-1}\mathfrak{p}%
_{+}\left(  \rho_{c}\left(  \boldsymbol{a}-\boldsymbol{f}\right)  u\right)
\left(  \lambda\right)
\]
is analytic on $D_{+}$. Thus, we have%
\begin{align}
&  \int_{C}\dfrac{g\left(  \lambda\right)  -g\left(  z\right)  }{2\pi i\left(
\lambda-z\right)  }\left(  \boldsymbol{a}-\boldsymbol{f}\right)  u\left(
\lambda\right)  d\lambda\nonumber\\
&  =\int_{C}\dfrac{g\left(  \lambda\right)  -g\left(  z\right)  }{2\pi
i\left(  \lambda-z\right)  }\rho_{c}\left(  \lambda\right)  ^{-1}%
\mathfrak{p}_{-}\left(  \rho_{c}\left(  \boldsymbol{a}-\boldsymbol{f}\right)
u\right)  \left(  \lambda\right)  d\lambda\text{.} \label{2-6}%
\end{align}
Now, we introduce two operators:
\begin{equation}
\left\{
\begin{array}
[c]{ll}%
S_{\boldsymbol{a}}u=\mathfrak{p}_{-}\left(  \rho_{c}\left(  \boldsymbol{a}%
-\boldsymbol{f}\right)  u\right)  & \text{: \ }H_{N,c}\left(  D_{+}\right)
\rightarrow H\left(  D_{-}\right) \\
\left(  H_{g}v\right)  \left(  z\right)  =\int_{C}\dfrac{\left(  g\left(
\lambda\right)  -g\left(  z\right)  \right)  \rho_{c}\left(  \lambda\right)
^{-1}}{2\pi i\left(  \lambda-z\right)  }v\left(  \lambda\right)  d\lambda &
\text{: \ }H\left(  D_{-}\right)  \rightarrow H_{N,c}\left(  D_{+}\right)
\end{array}
\right.  \text{.} \label{2-7}%
\end{equation}
Then, (\ref{2-6}) implies%
\begin{equation}
T\left(  g\boldsymbol{a}\right)  =gT\left(  \boldsymbol{a}\right)
+H_{g}S_{\boldsymbol{a}}\text{.} \label{2-8}%
\end{equation}

We estimate the Hilbert-Schmidt (HS in short) norms of $S_{\boldsymbol{a}}$
and $g^{-1}H_{g}$. Noting%
\[
\rho_{c}\left(  \lambda\right)  \left(  \boldsymbol{a}-\boldsymbol{f}\right)
\left(  \lambda\right)  u\left(  \lambda\right)  =\rho_{2c}\left(
\lambda\right)  \left(  \boldsymbol{a}-\boldsymbol{f}\right)  \left(
\lambda\right)  \rho_{c}\left(  \lambda\right)  ^{-1}u\left(  \lambda\right)
\]
due to $\rho_{c}\left(  \lambda\right)  =\rho_{c}\left(  \lambda^{-1}\right)
$, we can modify $S_{\boldsymbol{a}}$:%
\begin{align*}
\left(  S_{\boldsymbol{a}}u\right)  \left(  z\right)   &  =\int_{C}\dfrac
{\rho_{2c}\left(  \lambda\right)  \left(  \boldsymbol{a}-\boldsymbol{f}%
\right)  \left(  \lambda\right)  \rho_{c}\left(  \lambda\right)  ^{-1}u\left(
\lambda\right)  }{2\pi i\left(  \lambda-z\right)  }d\lambda\\
&  =\int_{C}\dfrac{\left(  \rho_{2c}\left(  \lambda\right)  \left(
\boldsymbol{a}-\boldsymbol{f}\right)  \left(  \lambda\right)  -\rho
_{2c}\left(  z\right)  \left(  \boldsymbol{a}-\boldsymbol{f}\right)  \left(
z\right)  \right)  \rho_{c}\left(  \lambda\right)  ^{-1}u\left(
\lambda\right)  }{2\pi i\left(  \lambda-z\right)  }d\lambda\text{,}%
\end{align*}
where we have used%
\[
\int_{C}\dfrac{\rho_{c}\left(  \lambda\right)  ^{-1}u\left(  \lambda\right)
}{2\pi i\left(  \lambda-z\right)  }d\lambda=-\mathfrak{p}_{-}\left(  \rho
_{c}^{-1}u\right)  \left(  z\right)  =0\text{ \ for }z\in D_{-}\text{,}%
\]
since $\rho_{c}^{-1}u\in H\left(  D_{+}\right)  $, which leads to
\[
\left\Vert S_{\boldsymbol{a}}\right\Vert _{HS}^{2}\leq\int_{C^{2}}\left\Vert
U_{\boldsymbol{a}}\left(  z,\lambda\right)  \right\Vert ^{2}\left\vert
\rho_{2c}\left(  \lambda\right)  \right\vert ^{-2}\left\vert d\lambda
\right\vert \left\vert dz\right\vert
\]
with%
\[
U_{\boldsymbol{a}}\left(  z,\lambda\right)  =\dfrac{\rho_{2c}\left(
\lambda\right)  \left(  \boldsymbol{a}-\boldsymbol{f}\right)  \left(
\lambda\right)  -\rho_{2c}\left(  z\right)  \left(  \boldsymbol{a}%
-\boldsymbol{f}\right)  \left(  z\right)  }{2\pi i\left(  \lambda-z\right)
}\text{.}%
\]
Here, $\left\Vert S_{\boldsymbol{a}}\right\Vert _{HS}$ denotes the HS norm of
the operator $S_{\boldsymbol{a}}$ from $H_{N,c}\left(  D_{+}\right)  $ to
$H\left(  D_{-}\right)  $. In order to estimate $\left\Vert S_{\boldsymbol{a}%
}\right\Vert _{HS}$, precisely speaking, first we have to consider the
operator $S_{\boldsymbol{a}}$ from $L_{N}^{2}\left(  C\right)  $ to
$L^{2}\left(  C\right)  $ and then restrict it to $H_{N,c}\left(
D_{+}\right)  $. For two $\boldsymbol{a}_{1}$, $\boldsymbol{a}_{2}%
\in\boldsymbol{A}_{N}\left(  C\right)  $, we have%
\begin{equation}
\left\Vert S_{\boldsymbol{a}_{1}}-S_{\boldsymbol{a}_{2}}\right\Vert _{HS}%
^{2}\leq\int_{C^{2}}\left\Vert U_{\boldsymbol{a}_{1}}\left(  z,\lambda\right)
-U_{\boldsymbol{a}_{2}}\left(  z,\lambda\right)  \right\Vert ^{2}\left\vert
\rho_{2c}\left(  \lambda\right)  \right\vert ^{-2}\left\vert d\lambda
\right\vert \left\vert dz\right\vert \text{.} \label{2-9}%
\end{equation}
Similarly, for $g\in\Gamma_{N}\left(  D_{+}\right)  $, the HS norm of the
operator $g^{-1}H_{g}$ from $H_{-}$ to $H_{N,c}\left(  D_{+}\right)  $ is
estimated as%
\[
\left\Vert g^{-1}H_{g}\right\Vert _{HS}^{2}\leq\int_{C^{2}}\left\vert
K_{g}\left(  z,\lambda\right)  \right\vert ^{2}\left\vert \rho_{c}\left(
z\right)  \right\vert ^{-2}\left\vert \rho_{c}\left(  \lambda\right)
\right\vert ^{-2}\left\vert d\lambda\right\vert \left\vert dz\right\vert
\]
with%
\[
K_{g}\left(  z,\lambda\right)  =\dfrac{g\left(  z\right)  ^{-1}g\left(
\lambda\right)  -1}{2\pi i\left(  \lambda-z\right)  }\text{,}%
\]
and, for $g_{1}$, $g_{2}\in\Gamma_{N}\left(  D_{+}\right)  $, we have
\begin{equation}
\left\Vert g_{1}^{-1}H_{g_{1}}-g_{2}^{-1}H_{g_{2}}\right\Vert _{HS}^{2}%
\leq\int_{C^{2}}\left\vert K_{g_{1}}\left(  z,\lambda\right)  -K_{g_{2}%
}\left(  z,\lambda\right)  \right\vert ^{2}\dfrac{\left\vert d\lambda
\right\vert \left\vert dz\right\vert }{\left\vert \rho_{c}\left(  z\right)
\right\vert ^{2}\left\vert \rho_{c}\left(  \lambda\right)  \right\vert ^{2}%
}\text{.} \label{2-10}%
\end{equation}
Set%
\begin{equation}
\left\{
\begin{array}
[c]{l}%
\mathrm{d}_{c}\left(  g_{1},g_{2}\right)  =\sqrt{%
%TCIMACRO{\dint _{C^{2}}}%
%BeginExpansion
{\displaystyle\int_{C^{2}}}
%EndExpansion
\left\vert K_{g_{1}}\left(  z,\lambda\right)  -K_{g_{2}}\left(  z,\lambda
\right)  \right\vert ^{2}\dfrac{\left\vert d\lambda\right\vert \left\vert
dz\right\vert }{\left\vert \rho_{c}\left(  z\right)  \right\vert
^{2}\left\vert \rho_{c}\left(  \lambda\right)  \right\vert ^{2}}}\\
\left\Vert \boldsymbol{a}\right\Vert _{c,1}=\inf\limits_{\boldsymbol{f}%
}\left(  \sup_{\lambda\in C}\left(  \left\Vert \boldsymbol{f}\left(
\lambda\right)  \right\Vert +\left\Vert \rho_{c}\left(  \boldsymbol{a}%
-\boldsymbol{f}\right)  \left(  \lambda\right)  \right\Vert \right)  \right)
\text{ }\\
\left\Vert \boldsymbol{a}\right\Vert _{c,2}=\inf\limits_{\boldsymbol{f}%
}\left(  \sqrt{%
%TCIMACRO{\dint _{C^{2}}}%
%BeginExpansion
{\displaystyle\int_{C^{2}}}
%EndExpansion
\left\Vert \dfrac{\rho_{c}\left(  \lambda\right)  \left(  \boldsymbol{a}%
-\boldsymbol{f}\right)  \left(  \lambda\right)  -\rho_{c}\left(  z\right)
\left(  \boldsymbol{a}-\boldsymbol{f}\right)  \left(  z\right)  }{2\pi
i\left(  \lambda-z\right)  }\right\Vert ^{2}\dfrac{\left\vert d\lambda
\right\vert \left\vert dz\right\vert }{\left\vert \rho_{c}\left(
\lambda\right)  \right\vert ^{2}}}\right)  \text{ }\\
\left\Vert \boldsymbol{a}\right\Vert _{c,3}=\inf\limits_{\boldsymbol{f}%
}\left(  \sup_{\lambda\in C_{1}}\left(  \left\Vert \partial_{\lambda}\rho
_{c}\left(  \boldsymbol{a}-\boldsymbol{f}\right)  \left(  \lambda\right)
\right\Vert +\left\Vert \partial_{\lambda}\widetilde{\rho_{c}\left(
\boldsymbol{a}-\boldsymbol{f}\right)  }\left(  \lambda\right)  \right\Vert
\right)  \right)
\end{array}
\right.  \text{.} \label{2-11}%
\end{equation}
For the metric $\mathrm{d}_{c}\left(  g_{1},g_{2}\right)  $, $\mathrm{d}%
_{c}\left(  g_{1},g_{2}\right)  =0$ holds if and only if $g_{1}%
=\mathrm{const.}g_{2}$. The quantities $\left\Vert \boldsymbol{a}\right\Vert
_{c,j}$ define semi-norms on $\boldsymbol{A}_{N}\left(  C\right)  $. The
continuity of the relevant operators with respect to $g$ and $\boldsymbol{a}$
is proved by the following:

\begin{lemma}
\label{l2-3}There exists a constant $c_{1}$ such that for $\boldsymbol{a}_{1}%
$, $\boldsymbol{a}_{2}$, $\boldsymbol{a}\in\boldsymbol{A}_{N}\left(  C\right)
$ and $g_{1}$, $g_{2}\in\Gamma_{N}\left(  D_{+}\right)  $
\begin{equation}
\left\{
\begin{array}
[c]{l}%
\left\Vert T\left(  \boldsymbol{a}\right)  \right\Vert \leq c_{1}\left\Vert
\boldsymbol{a}\right\Vert _{c,1}\\
\left\Vert S_{\boldsymbol{a}_{1}}-S_{\boldsymbol{a}_{2}}\right\Vert _{HS}%
\leq\left\Vert \boldsymbol{a}_{1}-\boldsymbol{a}_{2}\right\Vert _{c,2}\\
\left\Vert g_{1}^{-1}H_{g_{1}}-g_{2}^{-1}H_{g_{2}}\right\Vert _{HS}%
\leq\mathrm{d}_{c}\left(  g_{1},g_{2}\right)
\end{array}
\right.  \text{,} \label{2-12}%
\end{equation}
and for $\boldsymbol{a}\in\boldsymbol{A}_{N}\left(  C\right)  $
\begin{equation}
\left\Vert \boldsymbol{a}\right\Vert _{c,2}\leq c_{1}\left\Vert \boldsymbol{a}%
\right\Vert _{c,3}\text{.} \label{2-13}%
\end{equation}

\end{lemma}

\begin{proof}
The first estimate in (\ref{2-12}) is derived from (\ref{1-5}):%
\begin{align*}
\left\Vert T\left(  \boldsymbol{a}\right)  u\right\Vert _{N,c}  &
\leq\left\Vert \boldsymbol{f}u\right\Vert _{N,c}+\left\Vert \mathfrak{p}%
_{+}\left(  \left(  \boldsymbol{a}-\boldsymbol{f}\right)  u\right)
\right\Vert _{L^{2}\left(  C\right)  }\\
&  \leq\left\Vert \boldsymbol{f}\right\Vert _{\infty}\left\Vert u\right\Vert
_{N,c}+\left\Vert \mathfrak{p}_{+}\right\Vert \left\Vert \left(
\boldsymbol{a}-\boldsymbol{f}\right)  u\right\Vert _{L^{2}\left(  C\right)
}\text{\ }\\
&  \leq\left\Vert \boldsymbol{f}\right\Vert _{\infty}\left\Vert u\right\Vert
_{N,c}+\left\Vert \mathfrak{p}_{+}\right\Vert \left\Vert \rho_{c}\left(
\boldsymbol{a}-\boldsymbol{f}\right)  \right\Vert _{\infty}\left\Vert
u\right\Vert _{N,c}\leq c_{1}\left\Vert \boldsymbol{a}\right\Vert
_{c,1}\left\Vert u\right\Vert _{N,c}\text{,}%
\end{align*}
where $\left\Vert \cdot\right\Vert _{N,c}$ denotes the norm in $H_{N,c}\left(
D_{+}\right)  $ and $\left\Vert \boldsymbol{f}\right\Vert _{\infty}%
=\sup_{\lambda\in C}\left\Vert \boldsymbol{f}\left(  \lambda\right)
\right\Vert $. The rest two estimates of (\ref{2-12}) follow from (\ref{2-9})
and (\ref{2-10}).

For $\boldsymbol{a}\in\boldsymbol{A}_{N}\left(  C\right)  $ let
$\boldsymbol{b}\equiv\rho_{c}\left(  \boldsymbol{a}-\boldsymbol{f}\right)  $.
We divide the integral in $\left\Vert \boldsymbol{a}\right\Vert _{c,2}$ into
the 4 parts:%
\[
\int_{C^{2}}\left\Vert \dfrac{\boldsymbol{b}\left(  \lambda\right)
-\boldsymbol{b}\left(  z\right)  }{2\pi i\left(  \lambda-z\right)
}\right\Vert ^{2}\left\vert \rho_{c}\left(  \lambda\right)  \right\vert
^{-2}\left\vert d\lambda\right\vert \left\vert dz\right\vert =I_{11}%
+I_{12}+I_{21}+I_{22}\text{,}%
\]
where%
\[
I_{ij}=\int_{C_{i}\times C_{j}}\left\Vert \dfrac{\boldsymbol{b}\left(
\lambda\right)  -\boldsymbol{b}\left(  z\right)  }{2\pi i\left(
\lambda-z\right)  }\right\Vert ^{2}\left\vert \rho_{c}\left(  \lambda\right)
\right\vert ^{-2}\left\vert d\lambda\right\vert \left\vert dz\right\vert
\text{.}%
\]
The terms $I_{ij}$ for $i\neq j$ are estimated as
\[
I_{ij}\leq\pi^{-2}\left\Vert \boldsymbol{b}\right\Vert _{\infty}^{2}%
%TCIMACRO{\dint _{C_{i}}}%
%BeginExpansion
{\displaystyle\int_{C_{i}}}
%EndExpansion
\dfrac{\left\vert d\lambda\right\vert }{\left\vert \rho_{c}\left(
\lambda\right)  \right\vert ^{2}}\int_{C_{j}}\dfrac{\left\vert dz\right\vert
}{\left\vert \lambda-z\right\vert ^{2}}\text{,}%
\]
which are finite, since $\left\vert \lambda-z\right\vert \geq\mathrm{dist}%
\left(  C_{1},C_{2}\right)  >0$. The curve $C_{1}$ consists of two components
$C_{1}^{\pm}=C_{1}\cap\mathbb{C}_{\pm}$, and the integral $I_{11}$ is a sum of
the four integrals on the domains $C_{1}^{+}\times C_{1}^{+}$, $C_{1}%
^{+}\times C_{1}^{-}$, $C_{1}^{-}\times C_{1}^{+}$ and $C_{1}^{-}\times
C_{1}^{-}$. The second and third integrals can be estimated by $\mathrm{const}%
.\left\Vert \boldsymbol{b}\right\Vert _{\infty}^{2}$ similarly to $I_{21}$.
For the first integral, the region where $\lambda$ and $z$ are very close has
to be treated carefully. The integral on $C_{1}^{+}$ is divided into the two
parts:
\[
\int_{C_{1}^{+}}\left\Vert \dfrac{\boldsymbol{b}\left(  \lambda\right)
-\boldsymbol{b}\left(  z\right)  }{2\pi i\left(  \lambda-z\right)
}\right\Vert ^{2}\left\vert dz\right\vert =I_{1}+I_{2}%
\]
with%
\[
I_{1}=\int_{\left\vert \lambda-z\right\vert \leq1}\left\Vert \dfrac
{\boldsymbol{b}\left(  \lambda\right)  -\boldsymbol{b}\left(  z\right)  }{2\pi
i\left(  \lambda-z\right)  }\right\Vert ^{2}\left\vert dz\right\vert \text{,
\ }I_{2}=\int_{\left\vert \lambda-z\right\vert >1}\left\Vert \dfrac
{\boldsymbol{b}\left(  \lambda\right)  -\boldsymbol{b}\left(  z\right)  }{2\pi
i\left(  \lambda-z\right)  }\right\Vert ^{2}\left\vert dz\right\vert \text{.}%
\]
Observing that the curve $C_{1}^{+}$ is almost parallel to the real line as
$\lambda\rightarrow\pm\infty$, we treat the integrals as if they were on the
real line. Then, the term $I_{1}$ is estimated from above by $c_{1}%
\sup_{\lambda\in C_{1}^{+}}\left\Vert \partial_{\lambda}\boldsymbol{b}\left(
\lambda\right)  \right\Vert ^{2}$ with a constant $c_{1}$, and $I_{2}$ is
estimated by $c_{2}\sup_{\lambda\in C_{1}^{+}}\left\Vert \boldsymbol{b}\left(
\lambda\right)  \right\Vert ^{2}$ with another constant $c_{2}$. This leads us
to%
\[
\int_{C_{1}^{+}\times C_{1}^{+}}\left\Vert \dfrac{\boldsymbol{b}\left(
\lambda\right)  -\boldsymbol{b}\left(  z\right)  }{2\pi i\left(
\lambda-z\right)  }\right\Vert ^{2}\dfrac{\left\vert d\lambda\right\vert
\left\vert dz\right\vert }{\left\vert \rho_{c}\left(  \lambda\right)
\right\vert ^{2}}\leq c_{3}\sup_{\lambda\in C_{1}^{+}}\left(  \left\Vert
\partial_{\lambda}\boldsymbol{b}\left(  \lambda\right)  \right\Vert
^{2}+\left\Vert \boldsymbol{b}\left(  \lambda\right)  \right\Vert ^{2}\right)
\text{,}%
\]
where $c_{3}=\max\left(  c_{1},c_{2}\right)  \int_{C_{1}^{+}}\left\vert
\rho_{c}\left(  \lambda\right)  \right\vert ^{-2}\left\vert d\lambda
\right\vert $. The fourth integral on $C_{1}^{-}\times C_{1}^{-}$ is equal to
the first integral on $C_{1}^{+}\times C_{1}^{+}$. The integral $I_{22}$ is
equal to $I_{11}$ replacing $\boldsymbol{b}$ by $\widetilde{\boldsymbol{b}}$.
Therefore, we obtain (\ref{2-13}).\bigskip
\end{proof}

To verify convergences by the metric $\mathrm{d}_{c}\left(  g_{1}%
,g_{2}\right)  $ it is convenient to give a sufficient condition. Set%
\[
G_{c}\left(  z,\lambda\right)  =\left(  2c\int_{z}^{\lambda}\theta\left(
\xi\right)  \left\vert d\xi\right\vert \right)  \exp\left(  c\int_{z}%
^{\lambda}\theta\left(  s\right)  \left\vert ds\right\vert \right)  \text{,}%
\]
where the integral $\int_{z}^{\lambda}\cdot\left\vert ds\right\vert $ denotes
the integral along the shortest path connecting $z$ and $\lambda$ in $D_{+}$.

\begin{lemma}
\label{l2-4}For $g_{1}$, $g_{2}\in\Gamma_{N}\left(  D_{+}\right)  $ and any
$c>c^{\prime}\equiv\max_{j}\left\{  \mathrm{\delta}_{D_{+}}\left(
g_{j}\right)  \right\}  $,
\[
\mathrm{d}_{c}\left(  g_{1},g_{2}\right)  \leq\sqrt{\int_{C^{2}}\left\vert
\dfrac{G_{c^{\prime}}\left(  z,\lambda\right)  }{\lambda-z}\right\vert
^{2}\left\vert \rho_{c}\left(  \lambda\right)  \right\vert ^{-2}\left\vert
\rho_{c}\left(  z\right)  \right\vert ^{-2}\left\vert d\lambda\right\vert
\left\vert dz\right\vert }<\infty
\]
is valid. Especially, $\mathrm{d}_{c}\left(  g,1\right)  <\infty$ holds for
$g\in\Gamma_{N}\left(  D_{+}\right)  $, if $c>\mathrm{\delta}_{D_{+}}\left(
g\right)  $. Moreover, for a sequence $g_{n}\in\Gamma_{N}\left(  D_{+}\right)
$, suppose $\sup_{n\geq1}\mathrm{\delta}_{D_{+}}\left(  g_{n}\right)  <c$ and
$g_{n}\left(  \lambda\right)  \rightarrow g\left(  \lambda\right)  $ for a.e.
$\lambda\in C$. Then, $\mathrm{d}_{c}\left(  g_{n},g\right)  \rightarrow0$ as
$n\rightarrow\infty$ holds.\newline
\end{lemma}

\begin{proof}
We use an identity: for $z$, $\lambda\in D_{+}$
\begin{align*}
&  g_{1}\left(  z\right)  ^{-1}g_{1}\left(  \lambda\right)  -g_{2}\left(
z\right)  ^{-1}g_{2}\left(  \lambda\right) \\
&  =\int_{z}^{\lambda}\left(  \dfrac{g_{1}^{\prime}\left(  \xi\right)  }%
{g_{1}\left(  \xi\right)  }-\dfrac{g_{2}^{\prime}\left(  \xi\right)  }%
{g_{2}\left(  \xi\right)  }\right)  \exp\left(  \int_{z}^{\xi}\dfrac
{g_{1}^{\prime}\left(  s\right)  }{g_{1}\left(  s\right)  }ds+\int_{\xi
}^{\lambda}\dfrac{g_{2}^{\prime}\left(  s\right)  }{g_{2}\left(  s\right)
}ds\right)  d\xi\text{,}%
\end{align*}
which leads to%
\begin{align*}
&  \left\vert g_{1}\left(  z\right)  ^{-1}g_{1}\left(  \lambda\right)
-g_{2}\left(  z\right)  ^{-1}g_{2}\left(  \lambda\right)  \right\vert \\
&  \leq\int_{z}^{\lambda}\left(  \left\vert \dfrac{g_{1}^{\prime}\left(
\xi\right)  }{g_{1}\left(  \xi\right)  }-\dfrac{g_{2}^{\prime}\left(
\xi\right)  }{g_{2}\left(  \xi\right)  }\right\vert \right)  \exp\left(
\int_{z}^{\xi}\left\vert \dfrac{g_{1}^{\prime}\left(  s\right)  }{g_{1}\left(
s\right)  }\right\vert \left\vert ds\right\vert +\int_{\xi}^{\lambda
}\left\vert \dfrac{g_{2}^{\prime}\left(  s\right)  }{g_{2}\left(  s\right)
}\right\vert \left\vert ds\right\vert \right)  \left\vert d\xi\right\vert \\
&  \leq\left(  2c^{\prime}\int_{z}^{\lambda}\theta\left(  \xi\right)
\left\vert d\xi\right\vert \right)  \exp\left(  c^{\prime}\int_{z}^{\lambda
}\theta\left(  s\right)  \left\vert ds\right\vert \right)  =G_{c^{\prime}%
}\left(  z,\lambda\right)  \text{.}%
\end{align*}
If we can show%
\[
I\equiv\int_{C^{2}}\left\vert \dfrac{G_{c^{\prime}}\left(  z,\lambda\right)
}{\lambda-z}\right\vert ^{2}\left\vert \rho_{c}\left(  \lambda\right)
\right\vert ^{-2}\left\vert \rho_{c}\left(  z\right)  \right\vert
^{-2}\left\vert d\lambda\right\vert \left\vert dz\right\vert <\infty\text{,}%
\]
the dominated convergence theorem implies the second statement of the lemma.
The integral $I$ is divided into the four parts $I_{ij}$, where $I_{ij}$
denote the integrals on $C_{i}\times C_{j}$ for $i$, $j=1$, $2$. Since the
distance $\mathrm{dist}\left(  C_{1},C_{2}\right)  >0$, the integrals $I_{ij}$
for $i\neq j$ are easily shown to be finite. Observe%
\[
\lim_{\lambda\rightarrow\infty,\text{ }\lambda\in C_{1}\text{ }}\left\vert
\operatorname{Re}\lambda\right\vert ^{-N}\int_{\lambda_{0}}^{\lambda}%
\theta\left(  s\right)  \left\vert ds\right\vert =1
\]
for a fixed $\lambda_{0}\in C_{1}$, which implies that the exponential part of
$G\left(  z,\lambda\right)  $ grows at most of order $\rho_{c^{\prime\prime}%
}\left(  z\right)  $, $\rho_{c^{\prime\prime}}\left(  \lambda\right)  $. To
have the finiteness of $I_{11}$, we have to consider the affect by the
denominator $\lambda-z$ of the integrant. This singularity is compensated by
the term $\int_{z}^{\lambda}\theta\left(  \xi\right)  \left\vert
d\xi\right\vert $, which is dominated by%
\[
\dfrac{N}{\left\vert \lambda-z\right\vert }\int_{z}^{\lambda}\left\vert
\xi\right\vert ^{N-1}\left\vert d\xi\right\vert \leq c_{3}\left\vert
\lambda\right\vert ^{N-1}\text{, \ if }\left\vert \lambda\right\vert
/2<\left\vert z\right\vert <2\left\vert \lambda\right\vert \text{.}%
\]
If $\left\vert z\right\vert \leq\left\vert \lambda\right\vert /2$ or
$\left\vert z\right\vert \geq2\left\vert \lambda\right\vert $ holds, we have
$\left\vert \lambda-z\right\vert \geq\left\vert \lambda\right\vert -\left\vert
z\right\vert \geq\left\vert \lambda\right\vert /2$ or $\left\vert
\lambda-z\right\vert \geq\left\vert z\right\vert -\left\vert \lambda
\right\vert \geq\left\vert \lambda\right\vert $ respectively, which ensures
the finiteness of $I_{11}$. The identity $G\left(  z,\lambda\right)  =G\left(
z^{-1},\lambda^{-1}\right)  $ yields $I_{22}=$ $I_{11}$, which completes the
proof of the lemma.\bigskip
\end{proof}

Letting $\boldsymbol{a}_{0}=\boldsymbol{0}$ and $g_{2}=1$ in Lemma \ref{l2-3},
we see that $S_{\boldsymbol{a}}$ and $H_{g}$ are of Hilbert-Schmidt class.
Then, (\ref{2-8}) implies that $g^{-1}T\left(  g\boldsymbol{a}\right)
-T\left(  \boldsymbol{a}\right)  $ is of trace class on $H_{N,c}\left(
D_{+}\right)  $ for any $c>\mathrm{\delta}_{D_{+}}\left(  g\right)  $. Hence,
one can define $\tau_{\boldsymbol{a}}\left(  g\right)  $ by (\ref{1-11}) for
$\boldsymbol{a}\in\boldsymbol{A}_{N}^{inv}\left(  C\right)  $, $g\in\Gamma
_{N}\left(  D_{+}\right)  $, namely%
\[
\tau_{\boldsymbol{a}}\left(  g\right)  =\det\left(  g^{-1}T\left(
g\boldsymbol{a}\right)  T\left(  \boldsymbol{a}\right)  ^{-1}\right)  \text{.}%
\]
Lemma \ref{l2-1} implies that $\tau_{\boldsymbol{a}}\left(  g\right)  $ does
not depend on $c$ nor $N$. It is generally known that, for a trace class
operator $K$, $\det\left(  I+K\right)  \neq0$ holds if and only if $I+K$ is
invertible as a bounded operator. Therefore, if $\tau_{\boldsymbol{a}}\left(
g\right)  $ $\neq0$, the operator $g^{-1}T\left(  g\boldsymbol{a}\right)  $ is
invertible on $H_{N,c}\left(  D_{+}\right)  $ for any $c>\mathrm{\delta
}_{D_{+}}\left(  g\right)  $.

The tau-function $\tau_{g\boldsymbol{a}}$ for an extended symbol
$g\boldsymbol{a}$ is defined by:%
\[
\tau_{g\boldsymbol{a}}\left(  g_{1}\right)  =\det\left(  g_{1}^{-1}%
g^{-1}T\left(  g_{1}g\boldsymbol{a}\right)  \left(  g^{-1}T\left(
g\boldsymbol{a}\right)  \right)  ^{-1}\right)
\]
for $g$, $g_{1}\in\Gamma_{N}\left(  D_{+}\right)  $ and $\boldsymbol{a}%
\in\boldsymbol{A}_{N}^{inv}\left(  C\right)  $, assuming $g^{-1}T\left(
g\boldsymbol{a}\right)  $ is invertible (equivalently, $\tau_{\boldsymbol{a}%
}\left(  g\right)  \neq0$). The basic properties of the tau-functions are as
follows. Recall the notations%
\[
\widetilde{f}\left(  z\right)  =f\left(  z^{-1}\right)  \text{, \ \ }%
Rf(z)=z^{-1}f\left(  z^{-1}\right)  \text{, \ \ }\widetilde{\boldsymbol{a}%
}=\left(  \widetilde{a}_{1},\widetilde{a}_{2}\right)  \text{.}%
\]
The operator $R$ is unitary in $H_{N,c}\left(  D_{+}\right)  $, and the
identity $R\mathfrak{p}_{+}=\mathfrak{p}_{+}R$ yields%
\begin{equation}
RT\left(  g\boldsymbol{a}\right)  =T\left(  \widetilde{g\boldsymbol{a}%
}\right)  R\text{.} \label{2-14}%
\end{equation}

\begin{lemma}
\label{l2-5}(i) For $\boldsymbol{a}\in\boldsymbol{A}_{N}^{inv}\left(
C\right)  $ and $g\in\Gamma_{N}\left(  D_{+}\right)  $ with $\tau
_{\boldsymbol{a}}\left(  g\right)  \neq0$ it holds that $\widetilde
{\boldsymbol{a}}\in\boldsymbol{A}_{N}^{inv}\left(  C\right)  $ and
$\tau_{g\boldsymbol{a}}\left(  g_{1}\right)  =\tau_{\widetilde{g\boldsymbol{a}%
}}\left(  \widetilde{g}_{1}\right)  $ for any $g_{1}\in\Gamma_{N}\left(
D_{+}\right)  $.\newline(ii) (Cocycle property) \ For $\boldsymbol{a}%
\in\boldsymbol{A}_{N}^{inv}\left(  C\right)  $ and $g$, $g_{1}$, $g_{2}%
\in\Gamma_{N}\left(  D_{+}\right)  $ with $\tau_{\boldsymbol{a}}\left(
g\right)  $, $\tau_{\boldsymbol{a}}\left(  gg_{1}\right)  \neq0$, one has%
\[
\tau_{g\boldsymbol{a}}\left(  g_{1}g_{2}\right)  =\tau_{g\boldsymbol{a}%
}\left(  g_{1}\right)  \tau_{g_{1}g\boldsymbol{a}}\left(  g_{2}\right)
\text{.}%
\]
Especially, if $g_{1}(z)=g_{1}(z^{-1})$ holds, $\tau_{g_{1}g\boldsymbol{a}%
}\left(  g_{2}\right)  =\tau_{g\boldsymbol{a}}\left(  g_{2}\right)  $ and
$\tau_{g\boldsymbol{a}}\left(  g_{1}g_{2}\right)  =\tau_{g\boldsymbol{a}%
}\left(  g_{1}\right)  \tau_{g\boldsymbol{a}}\left(  g_{2}\right)  $ are
valid. In this case one has $\tau_{g\boldsymbol{a}}\left(  g_{1}\right)
\neq0$.\newline
\end{lemma}

\begin{proof}
The invariance
\[%
%TCIMACRO{\dint _{C^{2}}}%
%BeginExpansion
{\displaystyle\int_{C^{2}}}
%EndExpansion
\left\vert \dfrac{\widetilde{f}\left(  \lambda\right)  -\widetilde{f}\left(
z\right)  }{2\pi i\left(  \lambda-z\right)  }\right\vert ^{2}\dfrac{\left\vert
d\lambda\right\vert \left\vert dz\right\vert }{\left\vert \rho_{c}\left(
\lambda\right)  \right\vert ^{2}}=%
%TCIMACRO{\dint _{C^{2}}}%
%BeginExpansion
{\displaystyle\int_{C^{2}}}
%EndExpansion
\left\vert \dfrac{f\left(  \lambda\right)  -f\left(  z\right)  }{2\pi i\left(
\lambda-z\right)  }\right\vert ^{2}\dfrac{\left\vert d\lambda\right\vert
\left\vert dz\right\vert }{\left\vert \rho_{c}\left(  \lambda\right)
\right\vert ^{2}}%
\]
shows $\widetilde{\boldsymbol{a}}\in\boldsymbol{A}_{N}\left(  C\right)  $ for
$\boldsymbol{a}\in\boldsymbol{A}_{N}\left(  C\right)  $, and the identity
$RT\left(  \boldsymbol{a}\right)  =T\left(  \widetilde{\boldsymbol{a}}\right)
R$ of (\ref{2-14}) for $g=1$ implies $\widetilde{\boldsymbol{a}}%
\in\boldsymbol{A}_{N}^{inv}\left(  C\right)  $ if $\boldsymbol{a}%
\in\boldsymbol{A}_{N}^{inv}\left(  C\right)  $. The identity $\tau
_{g\boldsymbol{a}}\left(  g_{1}\right)  =\tau_{\widetilde{g\boldsymbol{a}}%
}\left(  \widetilde{g}_{1}\right)  $ follows easily from (\ref{2-14}), which
proves (i).

Observe%
\[
\tau_{g\boldsymbol{a}}\left(  g_{1}g_{2}\right)  =\det\left(  \left(
g_{1}g_{2}g\right)  ^{-1}T\left(  g_{1}g_{2}g\boldsymbol{a}\right)  \left(
g^{-1}T\left(  g\boldsymbol{a}\right)  \right)  ^{-1}\right)  =\det\left(
AB\right)
\]
with%
\[
\left\{
\begin{array}
[c]{l}%
A=\left(  g_{1}g_{2}g\right)  ^{-1}T\left(  g_{1}g_{2}g\boldsymbol{a}\right)
\left(  \left(  g_{1}g\right)  ^{-1}T\left(  g_{1}g\boldsymbol{a}\right)
\right)  ^{-1}\\
B=\left(  g_{1}g\right)  ^{-1}T\left(  g_{1}g\boldsymbol{a}\right)  \left(
g^{-1}T\left(  g\boldsymbol{a}\right)  \right)  ^{-1}%
\end{array}
\right.  \text{.}%
\]
The identity $\tau_{g\boldsymbol{a}}\left(  g_{1}g_{2}\right)  =\tau
_{g\boldsymbol{a}}\left(  g_{1}\right)  \tau_{g_{1}g\boldsymbol{a}}\left(
g_{2}\right)  $ follows from the property:%
\[
\det\left(  I+X\right)  \left(  I+Y\right)  =\det\left(  I+X\right)
\det\left(  I+Y\right)
\]
for trace class operators $X$, $Y$. The second identity is a consequence of
$g_{1}g\boldsymbol{a}=g\boldsymbol{a}g_{1}$ and $T\left(  g_{1}g\boldsymbol{a}%
\right)  =T\left(  g\boldsymbol{a}\right)  g_{1}$. The property $\tau
_{g\boldsymbol{a}}\left(  g_{1}\right)  \neq0$ follows by setting $g_{2}%
=g_{1}^{-1}$, which shows (ii).\bigskip\ 
\end{proof}

The continuity of tau-functions can be obtained by Lemma \ref{l2-3}.

\begin{lemma}
\label{l2-6}For $g_{1}$, $g_{2}$, $g\in\Gamma_{N}\left(  D_{+}\right)  $ and
$\boldsymbol{a}_{1}$, $\boldsymbol{a}_{2}$, $\boldsymbol{a}\in\boldsymbol{A}%
_{N}^{inv}\left(  C\right)  $ assume for an $M>0$
\[
\left\Vert \boldsymbol{a}\right\Vert _{c,2}\text{, }\left\Vert \boldsymbol{a}%
_{j}\right\Vert _{c,2}\text{, }\left\Vert T\left(  \boldsymbol{a}\right)
^{-1}\right\Vert \text{, }\left\Vert T\left(  \boldsymbol{a}_{j}\right)
^{-1}\right\Vert \leq M\text{.}%
\]
Then, there exist constants $c_{j}$ depending on $M$, $c$ such that the
followings are valid. \newline(i) \ $\left\vert \tau_{\boldsymbol{a}}\left(
g_{1}\right)  -\tau_{\boldsymbol{a}}\left(  g_{2}\right)  \right\vert \leq
c_{1}\mathrm{d}_{c}\left(  g_{1},g_{2}\right)  $ if $\mathrm{\delta}_{D_{+}%
}\left(  g_{j}\right)  <c$.\newline(ii) $\left\vert \tau_{\boldsymbol{a}_{1}%
}\left(  g\right)  -\tau_{\boldsymbol{a}_{2}}\left(  g\right)  \right\vert
\leq c_{2}\left(  \left\Vert \boldsymbol{a}_{1}-\boldsymbol{a}_{2}\right\Vert
_{c,1}+\left\Vert \boldsymbol{a}_{1}-\boldsymbol{a}_{2}\right\Vert
_{c,2}\right)  $ if $\mathrm{\delta}_{D_{+}}\left(  g\right)  <c$.
$\ $\newline The cocycle property $\tau_{g\boldsymbol{a}}\left(  g_{1}\right)
=\tau_{\boldsymbol{a}}\left(  gg_{1}\right)  /\tau_{\boldsymbol{a}}\left(
g\right)  $ reduces the continuity of general $\tau_{g\boldsymbol{a}}\left(
g_{1}\right)  $ to that of $\tau_{\boldsymbol{a}}\left(  g_{1}\right)  $.
\end{lemma}

\begin{proof}
Set $A_{j}=g^{-1}T\left(  g\boldsymbol{a}_{j}\right)  T\left(  \boldsymbol{a}%
_{j}\right)  ^{-1}-I$. We have from (\ref{2-8})%
\[
\left\{
\begin{array}
[c]{l}%
\left\Vert A_{1}-A_{2}\right\Vert _{tr}=\left\Vert \left(  g_{1}^{-1}H_{g_{1}%
}-g_{2}^{-1}H_{g_{2}}\right)  S_{\boldsymbol{a}}T\left(  \boldsymbol{a}%
\right)  ^{-1}\right\Vert _{tr}\\
\text{ \ \ \ \ \ \ \ \ \ \ \ \ \ \ \ }\leq\left\Vert g_{1}^{-1}H_{g_{1}}%
-g_{2}^{-1}H_{g_{2}}\right\Vert _{HS}\left\Vert S_{\boldsymbol{a}}\right\Vert
_{HS}\left\Vert T\left(  \boldsymbol{a}\right)  ^{-1}\right\Vert \text{,}\\
\left\Vert A_{j}\right\Vert _{tr}=\left\Vert g_{j}^{-1}H_{g_{j}}%
S_{\boldsymbol{a}}T\left(  \boldsymbol{a}\right)  ^{-1}\right\Vert _{tr}%
\leq\left\Vert g_{j}^{-1}H_{g_{j}}\right\Vert _{HS}\left\Vert
S_{\boldsymbol{a}}\right\Vert _{HS}\left\Vert T\left(  \boldsymbol{a}\right)
^{-1}\right\Vert \text{.}%
\end{array}
\right.
\]
Lemmas \ref{l2-3}, \ref{l2-4} and a general inequality (refer \cite{si})%
\[
\left\vert \det\left(  I+A_{1}\right)  -\det\left(  I+A_{2}\right)
\right\vert \leq\left\Vert A_{1}-A_{2}\right\Vert _{tr}\exp\left(  \left\Vert
A_{1}\right\Vert _{tr}+\left\Vert A_{2}\right\Vert _{tr}+1\right)
\]
for trace class operators $A_{j}$ imply (i).

(\ref{2-8}) yields%
\[
g^{-1}T\left(  g\boldsymbol{a}_{1}\right)  T\left(  \boldsymbol{a}_{1}\right)
^{-1}-g^{-1}T\left(  g\boldsymbol{a}_{2}\right)  T\left(  \boldsymbol{a}%
_{2}\right)  ^{-1}=g^{-1}H_{g}\left(  S_{\boldsymbol{a}_{1}}T\left(
\boldsymbol{a}_{1}\right)  ^{-1}-S_{\boldsymbol{a}_{2}}T\left(  \boldsymbol{a}%
_{2}\right)  ^{-1}\right)  \text{,}%
\]
and the second term is estimated as%
\begin{align*}
&  \left\Vert S_{\boldsymbol{a}_{1}}T\left(  \boldsymbol{a}_{1}\right)
^{-1}-S_{\boldsymbol{a}_{2}}T\left(  \boldsymbol{a}_{2}\right)  ^{-1}%
\right\Vert _{HS}\\
&  =\left\Vert \left(  S_{\boldsymbol{a}_{1}}-S_{\boldsymbol{a}_{2}}\right)
T\left(  \boldsymbol{a}_{1}\right)  ^{-1}+S_{\boldsymbol{a}_{2}}\left(
T\left(  \boldsymbol{a}_{1}\right)  ^{-1}-T\left(  \boldsymbol{a}_{2}\right)
^{-1}\right)  \right\Vert _{HS}\\
&  \leq\left\Vert S_{\boldsymbol{a}_{1}}-S_{\boldsymbol{a}_{2}}\right\Vert
_{HS}\left\Vert T\left(  \boldsymbol{a}_{1}\right)  ^{-1}\right\Vert
+\left\Vert S_{\boldsymbol{a}_{2}}\right\Vert _{HS}\left\Vert T\left(
\boldsymbol{a}_{1}\right)  -T\left(  \boldsymbol{a}_{2}\right)  \right\Vert
\left\Vert T\left(  \boldsymbol{a}_{1}\right)  ^{-1}\right\Vert \left\Vert
T\left(  \boldsymbol{a}_{2}\right)  ^{-1}\right\Vert \text{,}%
\end{align*}
which leads to (ii) by applying Lemma \ref{l2-3}.\bigskip
\end{proof}

We show the non-vanishing of the tau-function $\tau_{\boldsymbol{a}}\left(
g\right)  $ by approximating $g$ by rational functions. Therefore, this
approximation is crucial in this article. It is convenient for us to decompose
the group $\Gamma_{N}\left(  D_{+}\right)  $ into 2 parts. For this purpose we
introduce 4 domains:%
\[
\left\{
\begin{array}
[c]{ll}%
D_{-}^{1}=D_{-}\cap\left\{  \left\vert z\right\vert >1\right\}  \text{,} &
D_{-}^{2}=D_{-}\cap\left\{  \left\vert z\right\vert <1\right\} \\
D_{+}^{1}=\mathbb{C}\diagdown\overline{D}_{-}^{1}\text{,} & D_{+}^{2}=\left(
\mathbb{C}\diagdown\overline{D}_{-}^{2}\right)  \mathbb{\cup}\left\{
\infty\right\}
\end{array}
\right.  \text{.}%
\]
Then, $D_{-}=D_{-}^{1}\cup D_{-}^{2}$, $D_{+}=D_{+}^{1}\cap D_{+}^{2}$ hold.
Define%
\[
\text{\ }\mathrm{\delta}_{D_{+}^{1}}\left(  g\right)  =\sup_{z\in D_{+}%
^{1},\left\vert z\right\vert \geq1}\left\vert \dfrac{g^{\prime}(z)}%
{g(z)}\dfrac{1}{Nz^{N-1}}\right\vert \text{, \ }\mathrm{\delta}_{D_{+}^{2}%
}\left(  g\right)  =\sup_{z\in D_{+}^{2},\left\vert z\right\vert \leq
1}\left\vert \dfrac{g^{\prime}(z)}{g(z)}\dfrac{1}{Nz^{-N-1}}\right\vert
\text{,}%
\]
and%
\[
\Gamma_{N}\left(  D_{+}^{j}\right)  =\left\{
\begin{array}
[c]{c}%
g\text{; }g\text{ is analytic on }\left(  D_{+}^{\prime}\right)  ^{j}\text{,
}g\left(  z\right)  \neq0\text{ on }\left(  D_{+}^{\prime}\right)  ^{j}\text{
}\\
\text{and }\mathrm{\delta}_{\left(  D_{+}^{\prime}\right)  ^{j}}\left(
g\right)  <\infty\text{ for some }D_{+}^{\prime}\supset D_{+}%
\end{array}
\right\}  \text{,}%
\]
where $D_{+}^{\prime}$ is a domain similar to $D_{+}$. Clearly, it holds that
$\widetilde{g}\in\Gamma_{N}\left(  D_{+}^{2}\right)  $ if and only if
$g\in\Gamma_{N}\left(  D_{+}^{1}\right)  $, and $\Gamma_{N}\left(  D_{+}%
^{j}\right)  \subset\Gamma_{N}\left(  D_{+}\right)  $ for $j=1$, $2$. Although
$q_{\zeta}\in\Gamma_{N}\left(  D_{+}^{1}\right)  $ for $\zeta\in D_{-}^{1}$,
one has $q_{\zeta}\notin\Gamma_{N}\left(  D_{+}^{2}\right)  $ for $\zeta\in
D_{-}^{2}$, since $q_{\zeta}\left(  \infty\right)  =0$. Moreover, $z^{n}%
\notin$ $\Gamma_{N}\left(  D_{+}^{1}\right)  \cup\Gamma_{N}\left(  D_{+}%
^{2}\right)  $ for any $n\in\mathbb{Z}\backslash\left\{  0\right\}  $.

\begin{lemma}
\label{l2-7}For $g\in\Gamma_{N}\left(  D_{+}\right)  $ there exist $g_{j}%
\in\Gamma_{N}\left(  D_{+}^{j}\right)  $ such that $g=z^{n}g_{1}g_{2}$ for
$n\in\mathbb{Z}$.
\end{lemma}

\begin{proof}
To eliminate the singularities of an analytic function $g^{\prime}/g$ in
$D_{+}^{\prime}$ at $\infty$ and $0$ of forms $z^{N-1}$ and $z^{-N-1}$ we
prepare polynomials $p(z)$, $q(z)$ of degree $N$, $N+1$ having no zeros in
$\overline{D}_{+}^{^{\prime}}$. Observe an identity on $D_{+}^{\prime}$:%
\[
\dfrac{g^{\prime}\left(  z\right)  }{g\left(  z\right)  }=p\left(  z\right)
\widetilde{q}\left(  z\right)  \dfrac{1}{2\pi i}\int_{C^{\prime}}\dfrac
{1}{\lambda-z}\dfrac{g^{\prime}\left(  \lambda\right)  }{g\left(
\lambda\right)  }\dfrac{1}{p\left(  \lambda\right)  \widetilde{q}\left(
\lambda\right)  }d\lambda\text{.}%
\]
Setting%
\[
\left\{
\begin{array}
[c]{l}%
h_{1}^{\prime}(z)=\left(  \widetilde{q}\left(  z\right)  -q\left(  0\right)
\right)  \left(  pf_{1}(z)-\sum\limits_{0\leq k\leq N}\dfrac{\left(
pf_{1}\right)  ^{\left(  k\right)  }(0)}{k!}z^{k}\right)  +q\left(  0\right)
pf_{1}(z)\\
h_{2}^{\prime}(z)=p\left(  z\right)  \left(  \left(  q\widetilde{f}%
_{2}\right)  (z^{-1})-\sum\limits_{1\leq k\leq N+1}\dfrac{\left(
q\widetilde{f}_{2}\right)  ^{\left(  k\right)  }(0)}{k!}z^{-k}\right) \\
r(z)=\left(  \widetilde{q}\left(  z\right)  -q\left(  0\right)  \right)
\sum\limits_{0\leq k\leq N}\dfrac{\left(  pf_{1}\right)  ^{\left(  k\right)
}(0)}{k!}z^{k}+p\left(  z\right)  \sum\limits_{1\leq k\leq N+1}\dfrac{\left(
q\widetilde{f}_{2}\right)  ^{\left(  k\right)  }(0)}{k!}z^{-k}%
\end{array}
\right.  \text{,}%
\]
where%
\[
f_{j}(z)=\dfrac{1}{2\pi i}\int_{C_{j}^{\prime}}\dfrac{1}{\lambda-z}%
\dfrac{g^{\prime}\left(  \lambda\right)  }{g\left(  \lambda\right)  }\dfrac
{1}{p\left(  \lambda\right)  q\left(  \lambda^{-1}\right)  }d\lambda\text{.}%
\]
One sees that $h_{j}^{\prime}(z)$ are analytic on $\left(  D_{+}^{^{\prime}%
}\right)  ^{j}$ for $j=1$, $2$, and for $z\in D_{+}^{\prime}$ they satisfy%
\begin{equation}
h_{1}^{\prime}(z)+h_{2}^{\prime}(z)+r(z)=\dfrac{g^{\prime}\left(  z\right)
}{g\left(  z\right)  }\text{.} \label{2-15}%
\end{equation}
$h_{2}^{\prime}(z)$ has the asymptotic behavior:%
\[
h_{2}^{\prime}(z)=O\left(  z^{-2}\right)  \text{, }r(z)=O\left(
z^{N-1}\right)  \text{ \ as }z\rightarrow\infty\text{ in }D_{+}^{^{\prime}%
}\text{,}%
\]
and $g^{\prime}(z)/g(z)=O\left(  z^{N-1}\right)  $ as $z\rightarrow\infty$ in
$D_{+}^{^{\prime}}$. Therefore, (\ref{2-15}) implies $h_{1}^{\prime
}(z)=O\left(  z^{N-1}\right)  $ as $z\rightarrow\infty$ in $D_{+}^{^{\prime}}%
$. Whereas, $h_{1}^{\prime}(z)$ is analytic at $0$, hence, the property
$h_{2}^{\prime}(z)=O\left(  z^{-N-1}\right)  $ as $z\rightarrow0$ in
$D_{+}^{^{\prime}}$ follows due to $g^{\prime}(z)/g(z)$, $r(z)=O\left(
z^{-N-1}\right)  $ as $z\rightarrow0$ in $D_{+}^{^{\prime}}$. Let%
\[
r(z)=\sum_{0\leq k\leq N-1}\alpha_{k}z^{k}+\sum_{2\leq k\leq N+1}\beta
_{k}z^{-k}+\gamma z^{-1}\equiv r_{1}(z)+r_{2}(z)+\gamma z^{-1}\text{.}%
\]
Defining%
\[
g_{j}\left(  z\right)  =g\left(  z_{0}\right)  ^{1/2}\exp\left(  \int_{z_{0}%
}^{z}\left(  h_{j}^{\prime}(\zeta)+r_{j}(\zeta)\right)  d\zeta\right)  \text{
for a }z_{0}\in D_{+}\text{,}%
\]
we see $g_{j}\in\Gamma_{N}\left(  D_{+}^{j}\right)  $ and $g\left(  z\right)
=\left(  z/z_{0}\right)  ^{\gamma}g_{1}\left(  z\right)  g_{2}\left(
z\right)  $. Since the domain $\left(  D_{+}^{^{\prime}}\right)  ^{1}$ is
simply connected, the function $g_{1}(z)$ can be determined as a single valued
analytic function on $\left(  D_{+}^{^{\prime}}\right)  ^{1}$, which together
with $h_{1}^{\prime}(\zeta)+r_{1}(\zeta)=O(\zeta^{N-1})$ as $\zeta
\rightarrow\infty$ implies $g_{1}\in\Gamma_{N}\left(  D_{+}^{1}\right)  $.
Since $h_{2}^{\prime}(\zeta)+r_{2}(\zeta)=O(\zeta^{-2})$ as $\zeta
\rightarrow\infty$, the integral $\int_{z_{0}}^{z}\left(  h_{2}^{\prime}%
(\zeta)+r_{2}(\zeta)\right)  d\zeta$ defines an analytic function at
$z=\infty$, hence on $\left(  D_{+}^{\prime}\right)  ^{2}$, which with
$h_{2}^{\prime}(\zeta)+r_{2}(\zeta)=O(\zeta^{-N-1})$ as $\zeta\rightarrow0$
shows $g_{2}\in\Gamma_{N}\left(  D_{+}^{2}\right)  $. The number\ $\gamma$
should be an integer because $g\left(  z\right)  $ is single-valued on
$D_{+}^{\prime}$.\bigskip
\end{proof}

The approximation of $g\in\Gamma_{N}\left(  D_{+}\right)  $ by rational
functions is crucial in the following arguments.

\begin{lemma}
\label{l2-8}For $g\in\Gamma_{N}\left(  D_{+}^{j}\right)  $ there exists
rational functions $\left\{  r_{n}\right\}  _{n\geq1}\subset\Gamma_{N}\left(
D_{+}^{j}\right)  $ such that for sufficiently large $c$%
\[
\lim_{n\rightarrow\infty}\mathrm{d}_{c}\left(  r_{n},g\right)  =0\text{.}%
\]

\end{lemma}

\begin{proof}
Let $g\in\Gamma_{N}\left(  D_{+}^{1}\right)  $ and assume $g(0)=1$ without
loss of generality. The proof consists of 3 steps.

\textbf{Step} \textbf{1}. Set%
\[
h\left(  z\right)  =%
%TCIMACRO{\dint _{0}^{z}}%
%BeginExpansion
{\displaystyle\int_{0}^{z}}
%EndExpansion
\dfrac{g^{\prime}\left(  \zeta\right)  }{g\left(  \zeta\right)  }d\zeta\text{,
\ }g_{n}\left(  z\right)  =\exp\left(  \int_{0}^{z}\dfrac{h^{\prime}\left(
\lambda\right)  }{1+\lambda^{2}/n}d\lambda\right)  \in\Gamma_{N-1}\left(
D_{+}^{1}\right)  \text{.}%
\]
Then, $g_{n}\left(  z\right)  \rightarrow$ $g\left(  z\right)  $ on $\left(
D_{+}^{\prime}\right)  ^{1}$ and%
\[
\left\vert \dfrac{g_{n}^{\prime}\left(  z\right)  }{g_{n}\left(  z\right)
}\right\vert =\dfrac{\left\vert h^{\prime}\left(  z\right)  \right\vert
}{\left\vert 1+z^{2}/n\right\vert }\leq c_{1}\left\vert z\right\vert
^{N-1}\text{ \ on }\left(  D_{+}^{\prime}\right)  ^{1}\text{ if }\left\vert
z\right\vert \geq1\text{,}%
\]
which implies $\sup_{n\geq1}\mathrm{\delta}_{D_{+}^{1}}\left(  g_{n}\right)
<\infty$, and $\lim_{n\rightarrow\infty}\mathrm{d}_{c}\left(  g_{n},g\right)
=0$ for sufficiently large $c$ due to Lemma \ref{l2-4}.

\textbf{Step} \textbf{2}. Assume $g\in\Gamma_{N-1}\left(  D_{+}^{1}\right)  $
and set $h\left(  z\right)  =\int_{0}^{z}g^{\prime}\left(  \zeta\right)
/g\left(  \zeta\right)  d\zeta$ again. We approximate $h(z)$ by rational
functions. Let $f$ be an integrable function on the curve $C_{1}^{\prime}$.
Then, for $M>0$, $u\left(  z\right)  \equiv\dfrac{1}{2\pi i}\int_{C^{\prime}%
}\dfrac{1}{\lambda-z}f\left(  \lambda\right)  d\lambda$ has a decomposition:%
\begin{equation}
u\left(  z\right)  =\dfrac{1}{2\pi i}\int_{C_{1}^{\prime}:\left\vert
\lambda\right\vert \leq M}\dfrac{1}{\lambda-z}f\left(  \lambda\right)
d\lambda+\dfrac{1}{2\pi i}\int_{C_{1}^{\prime}:\left\vert \lambda\right\vert
>M}\dfrac{1}{\lambda-z}f\left(  \lambda\right)  d\lambda\text{.} \label{2-16}%
\end{equation}
Noting $\mathrm{dist}\left(  C_{1},C_{1}^{\prime}\right)  >0$, one has%
\[
\sup_{z\in D_{+}^{1}}\left\vert \dfrac{1}{2\pi i}\int_{C_{1}^{\prime
}:\left\vert \lambda\right\vert >M}\dfrac{1}{\lambda-z}f\left(  \lambda
\right)  d\lambda\right\vert \leq\dfrac{1}{2\pi\mathrm{dist}\left(
C_{1},C_{1}^{\prime}\right)  }\int_{C_{1}^{\prime}:\left\vert \lambda
\right\vert >M}\left\vert f\left(  \lambda\right)  \right\vert \left\vert
d\lambda\right\vert \text{,}%
\]
which can be made small as much as possible by choosing a large $M$. On the
other hand, the first term of (\ref{2-16}) is an integral on a compact set. If
a compact set $K$ of $\mathbb{C}$ has diameter $\varepsilon$, choosing $b\in
K$, the expansion:%
\[
\dfrac{1}{2\pi i}\int_{K}\dfrac{1}{\lambda-z}f\left(  \lambda\right)
d\lambda=\left(  b-z\right)  ^{-1}\sum_{j\geq0}\left(  b-z\right)  ^{-j}%
\dfrac{1}{2\pi i}\int_{K}\left(  b-\lambda\right)  ^{j}f\left(  \lambda
\right)  d\lambda
\]
converges uniformly on the region $\left\vert b-z\right\vert \geq2\varepsilon
$. A Cauchy integral on a general compact set $K$ can be approximated by
rational functions uniformly on the $2\varepsilon$-neighborhood of $K$ by
covering $K$ by a finite number of $\varepsilon$-disks. Hence, $u(z)$ can be
approximated by rational functions with no poles in $D_{+}^{1}$ uniformly on
$D_{+}^{1}$. Clearly, the derivative $u^{\prime}(z)$ also can be approximated
by the derivatives of the rational functions uniformly. Applying this argument
to $f(\lambda)=\lambda^{-1}h\left(  \lambda\right)  \left(  \lambda-b\right)
^{-N}\in L^{1}\left(  C_{1}^{\prime}\right)  $ with $b\in\mathbb{C}%
\diagdown\overline{D_{+}^{1}}$, one sees that $f\left(  z\right)  $ can be
approximated uniformly by rational functions $s_{k}(z)$ with no poles on
$\overline{D_{+}^{1}}$. Note $s_{k}(z)=O\left(  z^{-1}\right)  $ as
$z\rightarrow\infty$. In this case, since $h^{\prime}\left(  z\right)
=O\left(  z^{N-1}\right)  $, one has $s_{k}^{\prime}(z)=O\left(
z^{-2}\right)  $ as $z\rightarrow\infty$ uniformly in $k\geq1$ as well.

\textbf{Step 3}. Set
\[
g_{n}\left(  z\right)  =\left(  1+\dfrac{z^{2N}-h\left(  z\right)  }%
{n}\right)  ^{-n}\left(  1+\dfrac{z^{2N}}{n}\right)  ^{n}\text{,}%
\]
and define rational functions:%
\[
r_{n}\left(  z\right)  =\left(  1+\dfrac{z^{2N}-h_{n}\left(  z\right)  }%
{n}\right)  ^{-n}\left(  1+\dfrac{z^{2N}}{n}\right)  ^{n}\text{ with }%
h_{n}\left(  z\right)  =z\left(  z-b\right)  ^{N}s_{n}\left(  z\right)
\text{.}%
\]
It is clear that $\lim_{n\rightarrow\infty}r_{n}\left(  z\right)  =e^{h\left(
z\right)  }=g(z)$ uniformly on each compact set of $\overline{D}_{+}^{1}$. We
show $\sup_{n\geq n_{1}}\mathrm{\delta}_{D_{+}^{1}}\left(  r_{n}\right)
<\infty$ for some $n_{1}\geq1$. Observe%
\[
\dfrac{r_{n}^{\prime}\left(  z\right)  }{r_{n}\left(  z\right)  }=\dfrac
{h_{n}^{\prime}(z)}{1+\left(  z^{2N}-h_{n}\left(  z\right)  \right)
/n}-\dfrac{2Nz^{2N}/n}{1+z^{2N}/n}\dfrac{z^{-1}h_{n}\left(  z\right)
}{1+\left(  z^{2N}-h_{n}\left(  z\right)  \right)  /n}\text{.}%
\]
For $M>0$ set%
\[
\left\{
\begin{array}
[c]{l}%
c_{1}=\inf_{z\in D_{+}^{1},\left\vert \operatorname{Re}z\right\vert
>M}\operatorname{Re}\left(  z^{2N}-h_{n}\left(  z\right)  \right)  \left(
\operatorname{Re}z\right)  ^{-2N}\\
c_{2}=\sup_{z\in D_{+}^{1},\left\vert \operatorname{Re}z\right\vert \leq
M}\left\vert z^{2N}-h_{n}\left(  z\right)  \right\vert
\end{array}
\right.  \text{.}%
\]
We have $c_{1}>0$ for sufficiently large $M$, since $h_{n}\left(  z\right)
=O\left(  z^{N}\right)  $ as $z\rightarrow\infty$ uniformly on $n$ and
$\sup_{z\in D_{+}^{1}}\left\vert \operatorname{Im}z\right\vert <\infty$. Thus,
we have%
\[
\left\vert 1+\dfrac{z^{2N}-h_{n}\left(  z\right)  }{n}\right\vert \geq\left\{
\begin{array}
[c]{ll}%
1+c_{1}\left(  \operatorname{Re}z\right)  ^{2N}/n & \text{if \ }\left\vert
\operatorname{Re}z\right\vert >M\text{, }z\in D_{+}^{1}\\
1-c_{2}/n & \text{if \ }\left\vert \operatorname{Re}z\right\vert \leq M\text{,
}z\in D_{+}^{1}%
\end{array}
\right.  \text{.}%
\]
Similarly, one has%
\[
\left\vert \dfrac{z^{2N}/n}{1+z^{2N}/n}\right\vert \leq c_{3}\dfrac{\left(
\operatorname{Re}z\right)  ^{2N}/n}{1+\left(  \operatorname{Re}z\right)
^{2N}/n}\leq c_{3}%
\]
with some constant $c_{3}$. Consequently,%
\[
\left\vert \dfrac{r_{n}^{\prime}\left(  z\right)  }{r_{n}\left(  z\right)
}\right\vert \leq c_{4}\left(  \left\vert h_{n}^{\prime}(z)\right\vert
+\left\vert z\right\vert ^{-1}\left\vert h_{n}\left(  z\right)  \right\vert
\right)
\]
for another constant $c_{4}$, which implies $\sup_{n\geq n_{1}}\mathrm{\delta
}_{D_{+}^{1}}\left(  r_{n}\right)  <\infty$ (note $h_{n}(0)=0$ and
$h_{n}\left(  z\right)  =O\left(  z^{N}\right)  $, $h_{n}^{\prime}\left(
z\right)  =O\left(  z^{N-1}\right)  $ uniformly in $n$) for $n_{1}>c_{2}$, and
Lemma \ref{l2-4} yields $\mathrm{d}_{c}\left(  r_{n},g\right)  \rightarrow0$
for a sufficiently large $c>0$. The case $g\in\Gamma_{N}\left(  D_{+}%
^{2}\right)  $ can be treated by applying the above argument to $\widetilde
{g}\in\Gamma_{N}\left(  D_{+}^{1}\right)  $.\bigskip
\end{proof}

\subsection{$m$-function}

The $m$-function is defined for any symbol $\boldsymbol{a}\in\boldsymbol{A}%
_{N}^{inv}\left(  C\right)  $ and $g\in\Gamma_{N}\left(  D_{+}\right)  $
satisfying $\tau_{\boldsymbol{a}}\left(  g\right)  \neq0$, and is identified
with the Weyl functions for $\boldsymbol{a}\in\boldsymbol{A}_{N,++}%
^{inv}\left(  C\right)  $. It is used to show the non-vanishing of
$\tau_{\boldsymbol{a}}\left(  g\right)  $ and to express the Toda flow itself.
The flow is determined uniquely not by the symbol $\boldsymbol{a}$ but by the
$m$-function, and an equivalent symbol can be defined by the $m$-function. In
such a way the $m$-function is a crucial quantity in this paper. In this
section we provide basic properties for the $m$-function. Almost every part of
this section can be obtained by following the arguments in \cite{z} faithfully.

Since $g^{-1}z^{n}\in H_{N,c}\left(  D_{+}\right)  $ for any $n\in\mathbb{Z}$
and $g\in\Gamma_{N}\left(  D_{+}\right)  $ such that $\mathrm{\delta}_{D_{+}%
}\left(  g\right)  <c$, one can define
\begin{equation}
\varphi_{g\boldsymbol{a}}^{\left(  n\right)  }\equiv\mathfrak{p}_{-}\left(
g\left(  \boldsymbol{a}-\boldsymbol{f}\right)  \left(  g^{-1}T\left(
g\boldsymbol{a}\right)  \right)  ^{-1}g^{-1}z^{n}\right)  \in H\left(
D_{-}\right)  \label{2-18}%
\end{equation}
for $\boldsymbol{a}\in\boldsymbol{A}_{N}^{inv}\left(  C\right)  $ satisfying
$\tau_{\boldsymbol{a}}\left(  g\right)  \neq0$. Then, we have%
\begin{equation}
g\boldsymbol{a}\left(  g^{-1}T\left(  g\boldsymbol{a}\right)  \right)
^{-1}g^{-1}z^{n}=z^{n}+\varphi_{g\boldsymbol{a}}^{\left(  n\right)  }\text{.}
\label{2-19}%
\end{equation}
If the curve $C$ is bounded as in \cite{z}, $\mathfrak{p}_{-}g\boldsymbol{f}%
\left(  g^{-1}T\left(  g\boldsymbol{a}\right)  \right)  ^{-1}g^{-1}z^{n}=0$ is
valid and $\varphi_{g\boldsymbol{a}}^{\left(  n\right)  }$ turns to%
\[
\varphi_{g\boldsymbol{a}}^{\left(  n\right)  }=\mathfrak{p}_{-}\left(
g\boldsymbol{a}T\left(  g\boldsymbol{a}\right)  ^{-1}z^{n}\right)  \text{.}%
\]
$\left\{  \varphi_{g\boldsymbol{a}}^{\left(  n\right)  }\right\}
_{n\in\mathbb{Z}}$ are fundamental in SSW theory by which any important
quantity can be described. The definition of $\varphi_{g\boldsymbol{a}%
}^{\left(  n\right)  }$ yields%
\[
\varphi_{g\boldsymbol{a}}^{\left(  n\right)  }\left(  z\right)  =\int
_{C}\dfrac{g\left(  \lambda\right)  \left(  \boldsymbol{a}-\boldsymbol{f}%
\right)  \left(  \lambda\right)  \left(  \left(  g^{-1}T\left(
g\boldsymbol{a}\right)  \right)  ^{-1}g^{-1}z^{n}\right)  \left(
\lambda\right)  }{2\pi i\left(  z-\lambda\right)  }d\lambda\text{.}%
\]

We know $\left(  g^{-1}T\left(  g\boldsymbol{a}\right)  \right)  ^{-1}%
g^{-1}z^{n}\in H_{N,c}\left(  D_{+}\right)  $, and for any $c^{\prime}>0$ we
can choose $\boldsymbol{f}$ such that $\rho_{c^{\prime}}\left(  \boldsymbol{a}%
-\boldsymbol{f}\right)  =O\left(  1\right)  $. Hence, one can assume
\[
g\left(  \boldsymbol{a}-\boldsymbol{f}\right)  \left(  \left(  g^{-1}T\left(
g\boldsymbol{a}\right)  \right)  ^{-1}g^{-1}z^{n}\right)
\]
decays exponentially fast on $C$ as $z\rightarrow\infty$ or $z\rightarrow0$,
which implies%
\begin{equation}
\varphi_{g\boldsymbol{a}}^{\left(  n\right)  }\left(  z\right)  =\left\{
\begin{array}
[c]{l}%
\sum_{1\leq k\leq K}\alpha_{k}z^{-k}+z^{-K}\psi_{K,1}\left(  z\right) \\
\sum_{0\leq k\leq K-1}\beta_{k}z^{k}+z^{K}\text{ }\psi_{K,2}\left(  z\right)
\end{array}
\right.  \label{2-20}%
\end{equation}
for any $K\geq1$ with $\psi_{K,j}\in H\left(  D_{-}\right)  $ for $j=1$, $2$
and
\[
\left\{
\begin{array}
[c]{l}%
\alpha_{k}=\dfrac{1}{2\pi i}\int_{C}\lambda^{k-1}g\left(  \lambda\right)
\left(  \boldsymbol{a}-\boldsymbol{f}\right)  \left(  \lambda\right)  \left(
\left(  g^{-1}T\left(  g\boldsymbol{a}\right)  \right)  ^{-1}g^{-1}%
z^{n}\right)  \left(  \lambda\right)  d\lambda\\
\beta_{k}=-\dfrac{1}{2\pi i}\int_{C}\lambda^{-k-1}g\left(  \lambda\right)
\left(  \boldsymbol{a}-\boldsymbol{f}\right)  \left(  \lambda\right)  \left(
\left(  g^{-1}T\left(  g\boldsymbol{a}\right)  \right)  ^{-1}g^{-1}%
z^{n}\right)  \left(  \lambda\right)  d\lambda
\end{array}
\right.  \text{.}%
\]
We denote%
\[
\varphi_{g\boldsymbol{a}}^{\left(  n\right)  }\left(  0\right)  =\lim
_{z\rightarrow0}\varphi_{g\boldsymbol{a}}^{\left(  n\right)  }\left(
z\right)  =\beta_{0}\text{.}%
\]

Here, we remark the following fact. Let $f$ be a function on $C$ such that%
\[
\int_{C}\left\vert f\left(  \lambda\right)  \right\vert ^{2}\left\vert
d\lambda\right\vert <\infty\text{, \ }\int_{C}\left\vert \lambda^{\pm
1}f\left(  \lambda\right)  \right\vert ^{2}\left\vert d\lambda\right\vert
<\infty\text{.}%
\]
Under this condition one has%
\begin{align*}
\int_{C}\left\vert f\left(  \lambda\right)  \right\vert \left\vert
d\lambda\right\vert  &  =\int_{C_{1}}\left\vert f\left(  \lambda\right)
\right\vert \left\vert d\lambda\right\vert +\int_{C_{2}}\left\vert f\left(
\lambda\right)  \right\vert \left\vert d\lambda\right\vert \\
&  \leq\left\Vert \lambda^{-1}\right\Vert _{C_{1}}\left\Vert \lambda
f\right\Vert _{C_{1}}+\left\Vert 1\right\Vert _{C_{2}}\left\Vert f\right\Vert
_{C_{2}}<\infty\text{,}%
\end{align*}
similarly $\int_{C}\left\vert \lambda^{-1}f\left(  \lambda\right)  \right\vert
\left\vert d\lambda\right\vert <\infty$ also holds. Therefore, we have%
\begin{align}
\left(  \mathfrak{p}_{+}\left(  \lambda f\right)  \right)  \left(  z\right)
&  =\dfrac{1}{2\pi i}\int_{C}\dfrac{\lambda f\left(  \lambda\right)  }%
{\lambda-z}d\lambda\nonumber\\
&  =z\dfrac{1}{2\pi i}\int_{C}\dfrac{f\left(  \lambda\right)  }{\lambda
-z}d\lambda+\dfrac{1}{2\pi i}\int_{C}f\left(  \lambda\right)  d\lambda
\nonumber\\
&  =z\left(  \mathfrak{p}_{+}f\right)  \left(  z\right)  +\lim_{z\rightarrow
\infty}z\left(  \mathfrak{p}_{-}f\right)  \left(  z\right)  \text{,}
\label{2-21}%
\end{align}
and similarly%
\begin{equation}
\left(  \mathfrak{p}_{+}\left(  \lambda^{-1}f\right)  \right)  \left(
z\right)  =\dfrac{1}{2\pi i}\int_{C}\dfrac{\lambda^{-1}f\left(  \lambda
\right)  }{\lambda-z}d\lambda=z^{-1}\left(  \mathfrak{p}_{+}f\right)  \left(
z\right)  +z^{-1}\left(  \mathfrak{p}_{-}f\right)  \left(  0\right)
\label{2-22}%
\end{equation}

The Lemma below states that $\varphi_{g\boldsymbol{a}}^{\left(  n\right)  }$
can be expressed by $\left\{  \varphi_{g\boldsymbol{a}}^{\left(  0\right)
},\varphi_{g\boldsymbol{a}}^{\left(  -1\right)  }\right\}  $, which follows
from the invariance of $D_{-}$ under $z\rightarrow z^{-1}$ and $\boldsymbol{a}%
\phi=\phi\boldsymbol{a}$.

\begin{lemma}
\label{l2-9}Assume $\tau_{\boldsymbol{a}}\left(  g\right)  \neq0$ for
$\boldsymbol{a}\in\boldsymbol{A}_{N}^{inv}\left(  C\right)  $, $g\in\Gamma
_{N}\left(  D_{+}\right)  $. Then, $\varphi_{g\boldsymbol{a}}^{\left(
n\right)  }\left(  z\right)  $ satisfies\newline(i) \ $\varphi_{\widetilde
{g\boldsymbol{a}}}^{\left(  n\right)  }\left(  z\right)  =z^{-1}%
\varphi_{g\boldsymbol{a}}^{\left(  -n-1\right)  }\left(  z^{-1}\right)
$\newline(ii) $\left(  z+z^{-1}\right)  \left(  z^{n}+\varphi_{g\boldsymbol{a}%
}^{\left(  n\right)  }\right)  =z^{n+1}+\varphi_{g\boldsymbol{a}}^{\left(
n+1\right)  }+z^{n-1}+\varphi_{g\boldsymbol{a}}^{\left(  n-1\right)  }$

$\ \ \ \ \ \ \ \ \ \ \ \ \ \ \ \ \ \ \ \ \ \ \ \ \ +\varphi_{\widetilde
{g\boldsymbol{a}}}^{\left(  -n-1\right)  }\left(  0\right)  \left(
1+\varphi_{g\boldsymbol{a}}^{\left(  0\right)  }\right)  +\varphi
_{g\boldsymbol{a}}^{\left(  n\right)  }\left(  0\right)  \left(
z^{-1}+\varphi_{g\boldsymbol{a}}^{\left(  -1\right)  }\right)
\ \ \ \ \ \ \ \ \ \ \ \ \ \ \ \ \ \ \ \ \ \ \ \ \ \ \ \ \ \ \ \ \ \ \ \ \ \ \ \ \ $%

\end{lemma}

\begin{proof}
(i) is immediate from its definition and (\ref{2-14}). Recall $\phi\left(
z\right)  =z+z^{-1}$. Set%
\[
u=\left(  g^{-1}T\left(  g\boldsymbol{a}\right)  \right)  ^{-1}g^{-1}z^{n}\in
H_{N,c}\left(  D_{+}\right)  \text{.}%
\]
Then, applying (\ref{2-21}), (\ref{2-22}) to $f=g\left(  \boldsymbol{a}%
-\boldsymbol{f}\right)  u$ one has
\begin{align*}
&  \mathfrak{p}_{+}\left(  \phi g\left(  \boldsymbol{a}-\boldsymbol{f}\right)
u\right) \\
&  =\phi\mathfrak{p}_{+}\left(  g\left(  \boldsymbol{a}-\boldsymbol{f}\right)
u\right)  +\lim_{z\rightarrow\infty}z\left(  \mathfrak{p}_{-}g\left(
\boldsymbol{a}-\boldsymbol{f}\right)  u\right)  \left(  z\right)
+z^{-1}\left(  \mathfrak{p}_{-}g\left(  \boldsymbol{a}-\boldsymbol{f}\right)
u\right)  \left(  0\right) \\
&  =\phi\mathfrak{p}_{+}\left(  g\left(  \boldsymbol{a}-\boldsymbol{f}\right)
u\right)  +\lim_{z\rightarrow\infty}z\varphi_{g\boldsymbol{a}}^{\left(
n\right)  }\left(  z\right)  +z^{-1}\varphi_{g\boldsymbol{a}}^{\left(
n\right)  }\left(  0\right) \\
&  =\phi\mathfrak{p}_{+}\left(  g\left(  \boldsymbol{a}-\boldsymbol{f}\right)
u\right)  +\varphi_{\widetilde{g\boldsymbol{a}}}^{\left(  -n-1\right)
}\left(  0\right)  +z^{-1}\varphi_{g\boldsymbol{a}}^{\left(  n\right)
}\left(  0\right)  \text{ \ (due to (i)),}%
\end{align*}
hence, $\boldsymbol{a}\phi u=\phi\boldsymbol{a}u$ implies%
\begin{align*}
T\left(  g\boldsymbol{a}\right)  \phi u  &  =\phi g\boldsymbol{f}%
u+\mathfrak{p}_{+}\left(  \phi g\left(  \boldsymbol{a}-\boldsymbol{f}\right)
u\right) \\
&  =\phi T\left(  g\boldsymbol{a}\right)  u+\varphi_{\widetilde
{g\boldsymbol{a}}}^{\left(  -n-1\right)  }\left(  0\right)  +z^{-1}%
\varphi_{g\boldsymbol{a}}^{\left(  n\right)  }\left(  0\right) \\
&  =z^{n+1}+z^{n-1}+\varphi_{\widetilde{g\boldsymbol{a}}}^{\left(
-n-1\right)  }\left(  0\right)  +z^{-1}\varphi_{g\boldsymbol{a}}^{\left(
n\right)  }\left(  0\right)  \text{.}%
\end{align*}
Applying $g\boldsymbol{a}T\left(  g\boldsymbol{a}\right)  ^{-1}$ to the above
identity one has from (\ref{2-19})
\begin{align*}
\phi\left(  z^{n}+\varphi_{g\boldsymbol{a}}^{\left(  n\right)  }\right)   &
=z^{n+1}+\varphi_{g\boldsymbol{a}}^{\left(  n+1\right)  }+z^{n-1}%
+\varphi_{g\boldsymbol{a}}^{\left(  n-1\right)  }\\
&  +\varphi_{\widetilde{g\boldsymbol{a}}}^{\left(  -n-1\right)  }\left(
0\right)  \left(  1+\varphi_{g\boldsymbol{a}}^{\left(  0\right)  }\right)
+\varphi_{g\boldsymbol{a}}^{\left(  n\right)  }\left(  0\right)  \left(
z^{-1}+\varphi_{g\boldsymbol{a}}^{\left(  -1\right)  }\right)  \text{,}%
\end{align*}
which is (ii).\bigskip
\end{proof}

For rational $r\in\Gamma_{N}\left(  D_{+}\right)  $ the relevant operators
become of finite rank.

\begin{lemma}
\label{l2-10}Let $r=p/q\in\Gamma_{N}\left(  D_{+}\right)  $ with polynomials
$p$, $q$. For $\boldsymbol{a}\in\boldsymbol{A}_{N}\left(  C\right)  $,
$g\in\Gamma_{N}\left(  D_{+}\right)  $, the operator $r^{-1}g^{-1}T\left(
rg\boldsymbol{a}\right)  -g^{-1}T\left(  g\boldsymbol{a}\right)  $ is of
finite rank, whose image is a subspace spanned by
\[
\left\{  g^{-1}r^{-1}z^{k}\right\}  _{0\leq k\leq\deg p-1}\cup\left\{
g^{-1}z^{k}\right\}  _{0\leq k\leq\deg q-1}\left(  \subset H_{N,c}\left(
D_{+}\right)  \right)  \text{.}%
\]

\end{lemma}

\begin{proof}
For $u\in H_{N,c}\left(  D_{+}\right)  $ one has
\begin{align*}
&  \left(  g^{-1}r^{-1}T\left(  rg\boldsymbol{a}\right)  -g^{-1}T\left(
g\boldsymbol{a}\right)  \right)  u\left(  z\right) \\
&  =\left(  g\left(  z\right)  r\left(  z\right)  \right)  ^{-1}\int_{C}%
\dfrac{r\left(  \lambda\right)  -r\left(  z\right)  }{2\pi i\left(
\lambda-z\right)  }g\left(  \lambda\right)  \left(  \boldsymbol{a}%
-\boldsymbol{f}\right)  u\left(  \lambda\right)  d\lambda\text{.}%
\end{align*}
Noting%
\[
\dfrac{r\left(  \lambda\right)  -r\left(  z\right)  }{\lambda-z}=q\left(
\lambda\right)  ^{-1}\dfrac{p\left(  \lambda\right)  -p\left(  z\right)
}{\lambda-z}-r\left(  z\right)  q\left(  \lambda\right)  ^{-1}\dfrac{q\left(
\lambda\right)  -q\left(  z\right)  }{\lambda-z}%
\]
and%
\[
\dfrac{\lambda^{n}-z^{n}}{\lambda-z}=\sum_{0\leq j\leq n-1}\lambda^{^{j}%
}z^{n-j-1}\text{ \ for }n\geq1\text{,}%
\]
one has the conclusion.\bigskip
\end{proof}

For any rational function $r\in\Gamma_{N}\left(  D_{+}\right)  $ the
tau-function $\tau_{g\boldsymbol{a}}\left(  r\right)  $ can be expressed by
$\left\{  \varphi_{g\boldsymbol{a}}^{\left(  0\right)  },\varphi
_{g\boldsymbol{a}}^{\left(  -1\right)  }\right\}  $. Here we compute it for
$r=q_{\zeta}$, $q_{\zeta}q_{\eta}$, $q_{\zeta}\widetilde{q}_{\eta}$, and
remark the vanishing of $\tau_{g\boldsymbol{a}}\left(  r\right)  $ for a
certain $r$ under a degenerate condition for later purpose.

\begin{lemma}
\label{l2-11}Let $g\in\Gamma_{N}\left(  D_{+}\right)  $, $\boldsymbol{a}%
\in\boldsymbol{A}_{N}^{inv}\left(  C\right)  $ with $\tau_{\boldsymbol{a}%
}\left(  g\right)  \neq0$. Then, it holds that for $\zeta$, $\eta\in D_{-}%
$\newline(i) \ $\tau_{g\boldsymbol{a}}\left(  q_{\zeta}\right)  =1+\varphi
_{g\boldsymbol{a}}^{\left(  0\right)  }\left(  \zeta\right)  $,\newline(ii)
$\tau_{g\boldsymbol{a}}\left(  q_{\zeta}q_{\eta}\right)  =\dfrac{1}{\zeta
-\eta}\det\left(
\begin{array}
[c]{cc}%
1+\varphi_{g\boldsymbol{a}}^{\left(  0\right)  }\left(  \eta\right)  &
1+\varphi_{g\boldsymbol{a}}^{\left(  0\right)  }\left(  \zeta\right) \\
\eta+\varphi_{g\boldsymbol{a}}^{\left(  1\right)  }\left(  \eta\right)  &
\zeta+\varphi_{g\boldsymbol{a}}^{\left(  1\right)  }\left(  \zeta\right)
\end{array}
\right)  $,\newline(iii) $\tau_{g\boldsymbol{a}}\left(  q_{\zeta}\widetilde
{q}_{\eta}\right)  =\dfrac{1}{\eta-\zeta^{-1}}\det\left(
\begin{array}
[c]{cc}%
\eta+\varphi_{g\boldsymbol{a}}^{\left(  -1\right)  }\left(  \eta^{-1}\right)
& \zeta^{-1}+\varphi_{g\boldsymbol{a}}^{\left(  -1\right)  }\left(
\zeta\right) \\
1+\varphi_{g\boldsymbol{a}}^{\left(  0\right)  }\left(  \eta^{-1}\right)  &
1+\varphi_{g\boldsymbol{a}}^{\left(  0\right)  }\left(  \zeta\right)
\end{array}
\right)  $,\newline(iv) If $\tau_{g\boldsymbol{a}}\left(  q_{\zeta}\right)
=0$ holds for any $\zeta\in D_{-}^{2}$, then $\tau_{g\boldsymbol{a}}\left(
q^{-1}\right)  =0$ for any polynomial $q\in\Gamma_{N}\left(  D_{+}\right)  $
if $q\left(  \zeta_{0}\right)  =0$ for a $\zeta_{0}\in D_{-}^{2}$.
\end{lemma}

\begin{proof}
We compute $\tau_{g\boldsymbol{a}}\left(  r\right)  $ for $r=\sum_{0\leq j\leq
M}r_{j}q_{\zeta_{j}}$ with $\left\{  \zeta_{j}\right\}  _{0\leq j\leq
M}\subset D_{-}$. Since $r$ is bounded analytic on $D_{+}$, one has $rH\left(
D_{+}\right)  \subset H\left(  D_{+}\right)  $, hence for $u\in H_{N,c}\left(
D_{+}\right)  $
\begin{align*}
&  \left(  rg\right)  ^{-1}T\left(  rg\boldsymbol{a}\right)  u=\boldsymbol{f}%
u+\left(  rg\right)  ^{-1}\mathfrak{p}_{+}\left(  rg\left(  \boldsymbol{a}%
-\boldsymbol{f}\right)  u\right) \\
&  =\boldsymbol{f}u+\left(  rg\right)  ^{-1}\mathfrak{p}_{+}\left(  r\left(
\mathfrak{p}_{+}+\mathfrak{p}_{-}\right)  g\left(  \boldsymbol{a}%
-\boldsymbol{f}\right)  u\right) \\
&  =g^{-1}T\left(  g\boldsymbol{a}\right)  u+\left(  rg\right)  ^{-1}%
\mathfrak{p}_{+}\left(  r\mathfrak{p}_{-}g\left(  \boldsymbol{a}%
-\boldsymbol{f}\right)  u\right)
\end{align*}
holds. Generally for $v\in H\left(  D_{-}\right)  $ one has%
\[
rv=\sum_{0\leq j\leq M}r_{j}q_{\zeta_{j}}v\left(  \zeta_{j}\right)
+\sum_{0\leq j\leq M}r_{j}q_{\zeta_{j}}\left(  v-v\left(  \zeta_{j}\right)
\right)  \text{,}%
\]
which yields a decomposition in $H\left(  D_{+}\right)  \oplus H\left(
D_{-}\right)  $, hence%
\[
\left(  rg\right)  ^{-1}T\left(  rg\boldsymbol{a}\right)  u=g^{-1}T\left(
g\boldsymbol{a}\right)  u+\sum_{0\leq j\leq M}r_{j}\left(  \mathfrak{p}%
_{-}g\left(  \boldsymbol{a}-\boldsymbol{f}\right)  u\right)  \left(  \zeta
_{j}\right)  g^{-1}r^{-1}q_{\zeta_{j}}%
\]
implies that $\left(  rg\right)  ^{-1}T\left(  rg\boldsymbol{a}\right)
\left(  g^{-1}T\left(  g\boldsymbol{a}\right)  \right)  ^{-1}-I$ has its image
in a finite dimensional space with basis $\left\{  g^{-1}r^{-1}q_{\zeta_{j}%
}\right\}  _{0\leq j\leq M}$. Therefore, one has
\begin{equation}
\tau_{g\boldsymbol{a}}\left(  r\right)  =\det\left(  \delta_{ij}+r_{j}\left(
\mathfrak{p}_{-}g\left(  \boldsymbol{a}-\boldsymbol{f}\right)  \left(
g^{-1}T\left(  g\boldsymbol{a}\right)  \right)  ^{-1}g^{-1}r^{-1}q_{\zeta_{i}%
}\right)  \left(  \zeta_{j}\right)  \right)  _{0\leq i,j\leq M}\text{.}
\label{2-23}%
\end{equation}

We apply the formula (\ref{2-23}) to $r=q_{\zeta}\widetilde{q}_{\eta}$, which
is the case (iii). The cases (i), (ii) are simpler. Observe%
\[
r=\text{\ }r_{1}q_{\zeta}+r_{2}q_{\eta^{-1}}\text{ with }r_{1}=-r_{2}%
=\eta\left(  \eta-\zeta^{-1}\right)  ^{-1}%
\]
and%
\[
\left\{
\begin{array}
[c]{l}%
\mathfrak{p}_{-}g\left(  \boldsymbol{a}-\boldsymbol{f}\right)  \left(
g^{-1}T\left(  g\boldsymbol{a}\right)  \right)  ^{-1}g^{-1}r^{-1}q_{\zeta
}=\varphi_{g\boldsymbol{a}}^{\left(  0\right)  }-\eta^{-1}\varphi
_{g\boldsymbol{a}}^{\left(  -1\right)  }\\
\mathfrak{p}_{-}g\left(  \boldsymbol{a}-\boldsymbol{f}\right)  \left(
g^{-1}T\left(  g\boldsymbol{a}\right)  \right)  ^{-1}g^{-1}r^{-1}q_{\eta^{-1}%
}=\eta^{-1}\zeta^{-1}\varphi_{g\boldsymbol{a}}^{\left(  0\right)  }-\eta
^{-1}\varphi_{g\boldsymbol{a}}^{\left(  -1\right)  }%
\end{array}
\right.  \text{.}%
\]
Hence%
\[
\tau_{g\boldsymbol{a}}\left(  r\right)  =\det\left(
\begin{array}
[c]{cc}%
1+\eta\dfrac{\varphi_{g\boldsymbol{a}}^{\left(  0\right)  }\left(
\zeta\right)  -\eta^{-1}\varphi_{g\boldsymbol{a}}^{\left(  -1\right)  }\left(
\zeta\right)  }{\eta-\zeta^{-1}} & -\eta\dfrac{\varphi_{g\boldsymbol{a}%
}^{\left(  0\right)  }\left(  \eta^{-1}\right)  -\eta^{-1}\varphi
_{g\boldsymbol{a}}^{\left(  -1\right)  }\left(  \eta^{-1}\right)  }{\eta
-\zeta^{-1}}\\
\dfrac{\zeta^{-1}\varphi_{g\boldsymbol{a}}^{\left(  0\right)  }\left(
\zeta\right)  -\varphi_{g\boldsymbol{a}}^{\left(  -1\right)  }\left(
\zeta\right)  }{\eta-\zeta^{-1}} & 1-\dfrac{\zeta^{-1}\varphi_{g\boldsymbol{a}%
}^{\left(  0\right)  }\left(  \eta^{-1}\right)  -\varphi_{g\boldsymbol{a}%
}^{\left(  -1\right)  }\left(  \eta^{-1}\right)  }{\eta-\zeta^{-1}}%
\end{array}
\right)
\]
holds, from which (iii) follows.

To show (iv) first note that from (ii) of Lemma \ref{l2-9} $\zeta^{j}%
+\varphi_{g\boldsymbol{a}}^{\left(  j\right)  }\left(  \zeta\right)  =0$ on
$D_{-}^{2}$ for any $j\geq1$ follows inductively if $1+\varphi
_{g\boldsymbol{a}}^{\left(  0\right)  }\left(  \zeta\right)  =0$ on $D_{-}%
^{2}$. Then, to apply (\ref{2-23}) observe that for any polynomial
$p=\sum_{0\leq j\leq K}p_{j}z^{j}$ and $\zeta_{0}\in D_{-}^{2}$
\begin{align*}
&  \left(  \mathfrak{p}_{-}g\left(  \boldsymbol{a}-\boldsymbol{f}\right)
\left(  g^{-1}T\left(  g\boldsymbol{a}\right)  \right)  ^{-1}g^{-1}p\right)
\left(  \zeta_{0}\right) \\
&  =\sum_{0\leq j\leq K}p_{j}\varphi_{g\boldsymbol{a}}^{\left(  j\right)
}\left(  \zeta_{0}\right)  =-\sum_{0\leq j\leq K}p_{j}\zeta_{0}^{j}=-p\left(
\zeta_{0}\right)
\end{align*}
holds. Hence letting $p=r^{-1}q_{\zeta_{i}}=qq_{\zeta_{i}}$ one has%
\[
p\left(  \zeta_{0}\right)  =\left\{
\begin{array}
[c]{c}%
r_{0}^{-1}\text{ \ if \ }i=0\\
0\text{ \ \ \ \ if \ }i\neq0
\end{array}
\right.  \text{,}%
\]
which implies for any $i$%
\[
\delta_{i0}+r_{0}\left(  \mathfrak{p}_{-}g\left(  \boldsymbol{a}%
-\boldsymbol{f}\right)  \left(  g^{-1}T\left(  g\boldsymbol{a}\right)
\right)  ^{-1}g^{-1}r^{-1}q_{\zeta_{i}}\right)  \left(  \zeta_{0}\right)
=0\text{.}%
\]
Therefore, $\tau_{g\boldsymbol{a}}\left(  q^{-1}\right)  =0$ is valid.\bigskip
\end{proof}

The $m$-function is defined by%
\[
m_{g\boldsymbol{a}}\left(  z\right)  =\dfrac{z+\varphi_{g\boldsymbol{a}%
}^{\left(  1\right)  }\left(  z\right)  }{1+\varphi_{g\boldsymbol{a}}^{\left(
0\right)  }\left(  z\right)  }+\lim_{\zeta\rightarrow\infty}\zeta
\varphi_{g\boldsymbol{a}}^{\left(  0\right)  }\left(  \zeta\right)
=\dfrac{z+\varphi_{g\boldsymbol{a}}^{\left(  1\right)  }\left(  z\right)
}{1+\varphi_{g\boldsymbol{a}}^{\left(  0\right)  }\left(  z\right)  }%
+\varphi_{\widetilde{g\boldsymbol{a}}}^{\left(  -1\right)  }\left(  0\right)
\]
for $g\in\Gamma_{N}\left(  D_{+}\right)  $, $\boldsymbol{a}\in\boldsymbol{A}%
_{N}^{inv}\left(  C\right)  $ with $\tau_{\boldsymbol{a}}\left(  g\right)
\neq0$. The constant term is added so that $m_{g\boldsymbol{a}}$ satisfies%
\[
m_{g\boldsymbol{a}}\left(  z\right)  =z+O\left(  z^{-1}\right)  \text{ \ as
\ }z\rightarrow\infty\text{.}%
\]
$m_{g\boldsymbol{a}}$ is meromorphic on $D_{-}^{1}$ since $1+\varphi
_{g\boldsymbol{a}}^{\left(  0\right)  }\left(  z\right)  $ does not vanish
identically on $D_{-}^{1}$ due to $\varphi_{g\boldsymbol{a}}^{\left(
0\right)  }\left(  z\right)  \rightarrow0$ as $z\rightarrow\infty$, but
$m_{g\boldsymbol{a}}\left(  z\right)  $ might be identically $\infty$ on
$D_{-}^{2}$ since there is a possibility that $1+\varphi_{g\boldsymbol{a}%
}^{\left(  0\right)  }\left(  z\right)  =0$ identically on $D_{-}^{2}$. If we
assume $\tau_{g\boldsymbol{a}}\left(  z^{-1}\right)  =1+\varphi
_{g\boldsymbol{a}}^{\left(  0\right)  }\left(  0\right)  \neq0$, we can avoid
this inconvenience.

For later purpose it is convenient to introduce an auxiliary function%
\begin{equation}
n_{g\boldsymbol{a}}\left(  z\right)  =\dfrac{z^{-1}+\varphi_{g\boldsymbol{a}%
}^{\left(  -1\right)  }\left(  z\right)  }{1+\varphi_{g\boldsymbol{a}%
}^{\left(  0\right)  }\left(  z\right)  }=\dfrac{1}{1+\varphi_{g\boldsymbol{a}%
}^{\left(  0\right)  }\left(  0\right)  }\left(  z+z^{-1}-m_{g\boldsymbol{a}%
}\left(  z\right)  \right)  \text{ \ (Lemma \ref{l2-9}),} \label{2-26}%
\end{equation}
since it has a tau-function's expression:%
\begin{equation}
n_{g\boldsymbol{a}}\left(  \zeta\right)  =\dfrac{1+\varphi_{\widetilde
{g\boldsymbol{a}}}^{\left(  0\right)  }\left(  \zeta^{-1}\right)  }%
{\zeta\left(  1+\varphi_{g\boldsymbol{a}}^{\left(  0\right)  }\left(
\zeta\right)  \right)  }=\dfrac{\tau_{\widetilde{g\boldsymbol{a}}}\left(
q_{\zeta^{-1}}\right)  }{\zeta\tau_{g\boldsymbol{a}}\left(  q_{\zeta}\right)
}=\dfrac{\tau_{g\boldsymbol{a}}\left(  zq_{\zeta}\right)  }{\zeta
\tau_{g\boldsymbol{a}}\left(  q_{\zeta}\right)  }\text{ \ (Lemma \ref{l2-9},
\ref{l2-10}). } \label{2-27}%
\end{equation}

For simplicity of notation we define an operation%
\begin{equation}
\left\{
\begin{array}
[c]{l}%
\left(  d_{\eta}m\right)  \left(  z\right)  =\phi\left(  z\right)  -\left(
m\left(  \eta\right)  -m\left(  0\right)  \right)  \left(  1-\dfrac
{\phi\left(  z\right)  -\phi\left(  \eta\right)  }{m\left(  z\right)
-m\left(  \eta\right)  }\right) \\
\left(  d_{0}m\right)  \left(  z\right)  =\phi\left(  z\right)  -\dfrac
{m^{\prime}\left(  0\right)  }{m\left(  z\right)  -m\left(  0\right)  }\left(
=\lim_{\eta\rightarrow0}\left(  d_{\eta}m\right)  \left(  z\right)  \right)
\end{array}
\right.  \text{.} \label{2-28}%
\end{equation}
One can show

\begin{lemma}
\label{l2-12}(i) If $g\in\Gamma_{N}\left(  D_{+}\right)  $ satisfies
$g(z)=g\left(  z^{-1}\right)  $, then it holds that%
\[
m_{g\boldsymbol{a}}=m_{\boldsymbol{a}}\text{, \ }n_{g\boldsymbol{a}%
}=n_{\boldsymbol{a}}\text{.}%
\]
(ii) Assume $\tau_{g\boldsymbol{a}}\left(  z\right)  \tau_{g\boldsymbol{a}%
}\left(  z^{-1}\right)  \neq0$ for $\boldsymbol{a}\in\boldsymbol{A}_{N}%
^{inv}\left(  C\right)  $ and $g\in\Gamma_{N}\left(  D_{+}\right)  $ with
$\tau_{\boldsymbol{a}}\left(  g\right)  \neq0$. Then, it holds that
\begin{equation}
\left\{
\begin{array}
[c]{l}%
m_{\widetilde{g\boldsymbol{a}}}\left(  \zeta\right)  =\phi\left(
\zeta\right)  -\dfrac{\tau_{g\boldsymbol{a}}\left(  z\right)  \tau
_{g\boldsymbol{a}}\left(  z^{-1}\right)  }{\phi\left(  \zeta\right)
-m_{g\boldsymbol{a}}\left(  \zeta^{-1}\right)  }\text{, \ \ }n_{\widetilde
{g\boldsymbol{a}}}\left(  \zeta\right)  =n_{g\boldsymbol{a}}\left(  \zeta
^{-1}\right)  ^{-1}\\
m_{q_{\eta}g\boldsymbol{a}}\left(  \zeta\right)  =d_{\eta}m_{g\boldsymbol{a}%
}\left(  \zeta\right)  \text{, \ }m_{z^{-1}g\boldsymbol{a}}\left(
\zeta\right)  =d_{0}m_{g\boldsymbol{a}}\left(  \zeta\right)
\end{array}
\right.  \text{,} \label{2-29}%
\end{equation}
where%
\begin{equation}
\tau_{g\boldsymbol{a}}\left(  z\right)  \tau_{g\boldsymbol{a}}\left(
z^{-1}\right)  =\lim_{\zeta\rightarrow\infty}\zeta\left(  \phi\left(
\zeta\right)  -m_{g\boldsymbol{a}}(\zeta)\right)  \label{2-30}%
\end{equation}
(iii) For $\boldsymbol{a}_{1}$, $\boldsymbol{a}_{2}\in\boldsymbol{A}_{N}%
^{inv}\left(  C\right)  $ assume $\tau_{\boldsymbol{a}_{j}}\left(  z\right)
\tau_{\boldsymbol{a}_{j}}\left(  z^{-1}\right)  \neq0$ for $j=1$, $2$. Then,
$m_{\boldsymbol{a}_{1}}=m_{\boldsymbol{a}_{2}}$ implies $m_{g\boldsymbol{a}%
_{1}}=m_{g\boldsymbol{a}_{2}}$ for $g\in\Gamma_{N}\left(  D_{+}\right)  $ as
long as $\tau_{\boldsymbol{a}_{j}}\left(  g\right)  \neq0$ and $\tau
_{g\boldsymbol{a}_{j}}\left(  z\right)  \tau_{g\boldsymbol{a}_{j}}\left(
z^{-1}\right)  \neq0$ for $j=1$, $2$ hold.
\end{lemma}

\begin{proof}
(i) follows from (\ref{2-26}), (\ref{2-27}) and Lemma \ref{l2-5}.

(\ref{2-27}) says%
\[
\dfrac{\tau_{\widetilde{g\boldsymbol{a}}}\left(  zq_{\zeta}\right)  }%
{\zeta\tau_{\widetilde{g\boldsymbol{a}}}\left(  q_{\zeta}\right)  }%
=\dfrac{\tau_{g\boldsymbol{a}}\left(  z^{-1}\widetilde{q}_{\zeta}\right)
}{\zeta\tau_{g\boldsymbol{a}}\left(  \widetilde{q}_{\zeta}\right)  }%
=\dfrac{\zeta^{-1}\tau_{g\boldsymbol{a}}\left(  q_{\zeta^{-1}}\right)  }%
{\tau_{g\boldsymbol{a}}\left(  zq_{\zeta^{-1}}\right)  }=n_{g\boldsymbol{a}%
}\left(  \zeta^{-1}\right)  ^{-1}\text{,}%
\]
hence from (\ref{2-26})%
\[
m_{\widetilde{g\boldsymbol{a}}}\left(  \zeta\right)  =\phi\left(
\zeta\right)  -\tau_{\widetilde{g\boldsymbol{a}}}\left(  z^{-1}\right)
n_{\widetilde{g\boldsymbol{a}}}\left(  \zeta\right)  =\phi\left(
\zeta\right)  -\frac{\tau_{g\boldsymbol{a}}\left(  z\right)  \tau
_{g\boldsymbol{a}}\left(  z^{-1}\right)  }{\phi\left(  \zeta\right)
-m_{g\boldsymbol{a}}\left(  \zeta^{-1}\right)  }%
\]
follows. $\tau_{g\boldsymbol{a}}\left(  z\right)  \tau_{g\boldsymbol{a}%
}\left(  z^{-1}\right)  $ is recovered from $m_{g\boldsymbol{a}}$, namely%
\begin{align*}
\lim_{\zeta\rightarrow\infty}\zeta\left(  \phi\left(  \zeta\right)
-m_{g\boldsymbol{a}}(\zeta)\right)   &  =\tau_{g\boldsymbol{a}}\left(
z^{-1}\right)  \lim_{\zeta\rightarrow\infty}\zeta n_{g\boldsymbol{a}}\left(
\zeta\right) \\
&  =\tau_{g\boldsymbol{a}}\left(  z^{-1}\right)  \tau_{\widetilde
{g\boldsymbol{a}}}\left(  z^{-1}\right)  =\tau_{g\boldsymbol{a}}\left(
z\right)  \tau_{g\boldsymbol{a}}\left(  z^{-1}\right)  \text{.}%
\end{align*}
On the other hand, $n_{q_{\eta}\boldsymbol{a}}$ is%
\[
n_{q_{\eta}g\boldsymbol{a}}\left(  \zeta\right)  =\text{ }\dfrac{\tau
_{q_{\eta}g\boldsymbol{a}}\left(  zq_{\zeta}\right)  }{\zeta\tau_{q_{\eta
}g\boldsymbol{a}}\left(  q_{\zeta}\right)  }=\text{ }\dfrac{\tau
_{g\boldsymbol{a}}\left(  zq_{\eta}q_{\zeta}\right)  }{\zeta\tau
_{g\boldsymbol{a}}\left(  q_{\eta}q_{\zeta}\right)  }\text{,}%
\]
hence from Lemma \ref{l2-9} one has%
\begin{align*}
n_{q_{\eta}g\boldsymbol{a}}\left(  \zeta\right)   &  =\eta\frac
{n_{g\boldsymbol{a}}\left(  \eta\right)  -n_{g\boldsymbol{a}}\left(
\zeta\right)  }{m_{g\boldsymbol{a}}\left(  \zeta\right)  -m_{g\boldsymbol{a}%
}\left(  \eta\right)  }\\
&  =\frac{\eta}{\tau_{g\boldsymbol{a}}\left(  z^{-1}\right)  }\left(
1-\frac{\phi\left(  \zeta\right)  -\phi\left(  \eta\right)  }%
{m_{g\boldsymbol{a}}\left(  \zeta\right)  -m_{g\boldsymbol{a}}\left(
\eta\right)  }\right)  \text{,}%
\end{align*}
which yields%
\begin{align*}
m_{q_{\eta}g\boldsymbol{a}}\left(  \zeta\right)   &  =\phi\left(
\zeta\right)  -\tau_{q_{\eta}g\boldsymbol{a}}\left(  z^{-1}\right)
n_{q_{\eta}g\boldsymbol{a}}\left(  \zeta\right) \\
&  =\phi\left(  \zeta\right)  -\frac{\tau_{q_{\eta}g\boldsymbol{a}}\left(
z^{-1}\right)  \eta}{\tau_{g\boldsymbol{a}}\left(  z^{-1}\right)  }\left(
1-\frac{\phi\left(  \zeta\right)  -\phi\left(  \eta\right)  }%
{m_{g\boldsymbol{a}}\left(  \zeta\right)  -m_{g\boldsymbol{a}}\left(
\eta\right)  }\right) \\
&  =\phi\left(  \zeta\right)  -\left(  m_{g\boldsymbol{a}}\left(  \eta\right)
-m_{g\boldsymbol{a}}\left(  0\right)  \right)  \left(  1-\frac{\phi\left(
\zeta\right)  -\phi\left(  \eta\right)  }{m_{g\boldsymbol{a}}\left(
\zeta\right)  -m_{g\boldsymbol{a}}\left(  \eta\right)  }\right)  \text{,}%
\end{align*}
Here, we have used an identity%
\[
\tau_{q_{\eta}g\boldsymbol{a}}\left(  z^{-1}\right)  =\frac{\tau
_{g\boldsymbol{a}}\left(  q_{\eta}z^{-1}\right)  }{\tau_{g\boldsymbol{a}%
}\left(  q_{\eta}\right)  }=\frac{m_{g\boldsymbol{a}}\left(  \eta\right)
-m_{g\boldsymbol{a}}\left(  0\right)  }{\eta}\text{,}%
\]
which yields (ii).

If $g$ is a rational function $r$, (iii) obeys from (i) and (ii) by induction.
For instance $m_{q_{\eta_{1}}q_{\eta_{2}}\boldsymbol{a}}\left(  \zeta\right)
=d_{\eta_{1}}d_{\eta_{2}}m_{\boldsymbol{a}}\left(  \zeta\right)  $, and%
\[
m_{q_{\eta}^{-1}\boldsymbol{a}}\left(  \zeta\right)  =m_{\widetilde{q}_{\eta
}\widetilde{q}_{\eta}^{-1}q_{\eta}^{-1}\boldsymbol{a}}\left(  \zeta\right)
=m_{\widetilde{q}_{\eta}\boldsymbol{a}}\left(  \zeta\right)  =m_{\widetilde
{q_{\eta}\widetilde{\boldsymbol{a}}}}\left(  \zeta\right)  \text{.}%
\]
For a general $g$ one can show (iii) by approximating $g$ by rational
functions.\bigskip
\end{proof}

We compute $\varphi_{\boldsymbol{a}}^{\left(  -1\right)  }\left(  z\right)  $,
$\varphi_{\boldsymbol{a}}^{\left(  0\right)  }\left(  z\right)  $,
$m_{\boldsymbol{a}}\left(  z\right)  $ for $\boldsymbol{a}=\boldsymbol{m}%
=\left(  m_{1},m_{2}\right)  \in\boldsymbol{M}_{N}\left(  C\right)  $. It
should be remarked that expansions%
\begin{align}
m_{1}(z)  &  =\left\{
\begin{array}
[c]{l}%
1+\sum_{1\leq k\leq L-1}\mu_{k}^{\left(  1\right)  }z^{-k}+O\left(
z^{-L}\right)  \text{ \ as \ }z\rightarrow\infty\text{ in }D_{-}^{1}\\
1+\sum_{1\leq k\leq L-1}\nu_{k}^{\left(  1\right)  }z^{k}+O\left(
z^{L}\right)  \text{ \ as \ }z\rightarrow0\text{ in }D_{-}^{2}%
\end{array}
\right.  \text{,}\nonumber\\
m_{2}(z)  &  =\left\{
\begin{array}
[c]{l}%
\sum_{1\leq k\leq L-1}\mu_{k}^{\left(  2\right)  }z^{-k}+O\left(
z^{-L}\right)  \text{ \ as \ }z\rightarrow\infty\text{ in }D_{-}^{1}\\
\sum_{1\leq k\leq L-1}\nu_{k}^{\left(  2\right)  }z^{k}+O\left(  z^{L}\right)
\text{ \ as \ }z\rightarrow0\text{ in }D_{-}^{2}%
\end{array}
\right.  \label{2-31}%
\end{align}
hold for any $L\geq1$. This is because the conditions of (\ref{2-4}) imply for
$z\in D_{-}$%
\[
m_{1}(z)=1+\dfrac{1}{2\pi i}\int_{C}\dfrac{m_{1}(\lambda)-1}{z-\lambda
}d\lambda=\dfrac{1}{2\pi i}\int_{C}\dfrac{m_{1}(\lambda)-f_{1}(\lambda
)}{z-\lambda}d\lambda\text{,}%
\]
since
\[
\dfrac{1}{2\pi i}\int_{C}\dfrac{f_{1}(\lambda)-1}{z-\lambda}d\lambda=0
\]
holds due to $f_{1}(z)-1$ is bounded and analytic in $D_{+}$, and
$f_{1}(\lambda)-1\in L^{2}\left(  C\right)  $. Since $m_{1}(\lambda
)-f_{1}(\lambda)$ is rapidly decreasing on $C$, one easily have the first
expansion of (\ref{2-31}). Applying this argument to $\widetilde{m}_{1}$ one
has the second expansion. The expansions for $m_{2}$ can be obtained similarly.

\begin{lemma}
\label{l2-13}For $\boldsymbol{m}=\left(  m_{1},m_{2}\right)  \in
\boldsymbol{M}_{N}\left(  C\right)  $ it holds that%
\begin{equation}
\left\{
\begin{array}
[c]{l}%
z^{-1}+\varphi_{\boldsymbol{m}}^{\left(  -1\right)  }\left(  z\right)
=z^{-1}m_{1}\left(  z\right)  +m_{2}\left(  z\right)  \text{, \ }%
1+\varphi_{\boldsymbol{m}}^{\left(  0\right)  }\left(  z\right)  =m_{1}\left(
z\right)  +z^{-1}m_{2}\left(  z\right) \\
m_{\boldsymbol{m}}=z+\dfrac{\left(  z^{-2}-1\right)  m_{2}-m_{2}^{\prime
}(0)\left(  z^{-1}m_{1}+m_{2}\right)  }{m_{1}+m_{2}z^{-1}}%
\end{array}
\right.  \label{2-32}%
\end{equation}

\end{lemma}

\begin{proof}
Note%
\[
\left(  \boldsymbol{m}1\right)  \left(  z\right)  =1+\left(  m_{1}\left(
z\right)  -1+z^{-1}\left(  m_{2}\left(  z\right)  -m_{2}\left(  0\right)
\right)  \right)  \text{.}%
\]
Since $m_{1}(\infty)=1$, $m_{2}(0)=0$ imply%
\[
m_{1}\left(  z\right)  -1+z^{-1}\left(  m_{2}\left(  z\right)  -m_{2}\left(
0\right)  \right)  \in H\left(  D_{-}\right)  \text{,}%
\]
one has $T\left(  \boldsymbol{m}\right)  1=1$. Similarly $T\left(
\boldsymbol{m}\right)  z^{-1}=m_{1}\left(  0\right)  z^{-1}=z^{-1}$ holds.
Since $T\left(  \boldsymbol{m}\right)  $ is invertible, we easily see%
\[
T\left(  \boldsymbol{m}\right)  ^{-1}1=1\text{, \ \ }T\left(  \boldsymbol{m}%
\right)  ^{-1}z^{-1}=z^{-1}\text{,}%
\]
which implies the identities on the first line of (\ref{2-32}) by multiplying
$\boldsymbol{m}$ to the both sides. Therefore, $n_{\boldsymbol{m}}$ is
computable, hence so is $m_{\boldsymbol{m}}$.
\end{proof}

\subsection{Non-vanishing of tau-function on $\boldsymbol{A}_{N,++}%
^{inv}\left(  C\right)  $}

The invertibility of $g^{-1}T\left(  g\boldsymbol{a}\right)  $ for
$\boldsymbol{a}\in\boldsymbol{A}_{N}^{inv}\left(  C\right)  $ and $g\in
\Gamma_{N}\left(  D_{+}\right)  $ is crucial in this article. We prove this by
showing $\tau_{\boldsymbol{a}}\left(  g\right)  >0$ if $\boldsymbol{a}%
\in\boldsymbol{A}_{N,++}^{inv}\left(  C\right)  $ and $g\in\Gamma
_{N}^{\operatorname{real}}\left(  D_{+}\right)  $. In the proof the
$m$-functions play an important role.

Recall $D_{-}^{1}=D_{-}\cap\left\{  \left\vert z\right\vert >1\right\}  $,
$D_{-}^{2}=D_{-}\cap\left\{  \left\vert z\right\vert <1\right\}  $. For
$\zeta\in D_{-}$ set
\[
r_{\zeta}\left(  z\right)  =q_{\zeta}\left(  z\right)  q_{\overline{\zeta}%
}\left(  z\right)  \in\Gamma_{N}^{\operatorname{real}}\left(  D_{+}\right)
\text{.}%
\]

\begin{lemma}
\label{l2-14}For $\boldsymbol{a}\in\boldsymbol{A}_{N,+}^{inv}\left(  C\right)
$ and $g\in\Gamma_{N}^{\operatorname{real}}\left(  D_{+}\right)  $ assume
$\tau_{\boldsymbol{a}}\left(  g\right)  >0$. \newline(i) It holds that%
\[
1+\varphi_{g\boldsymbol{a}}^{\left(  0\right)  }\left(  \zeta\right)
\neq0\text{, \ }\tau_{g\boldsymbol{a}}\left(  r_{\zeta}\right)  >0\text{,\ \ }%
\dfrac{\operatorname{Im}m_{g\boldsymbol{a}}\left(  \zeta\right)
}{\operatorname{Im}\zeta}>0\text{\ on }D_{-}^{1}\text{.\ }%
\]
\newline(ii) $\tau_{g\boldsymbol{a}}\left(  z^{-1}\right)  =0$ holds if and
only if $1+\varphi_{g\boldsymbol{a}}^{\left(  0\right)  }\left(  z\right)  =0$
identically on $D_{-}^{2}$. If $\tau_{g\boldsymbol{a}}\left(  z^{-1}\right)
=0$, then $\tau_{g\boldsymbol{a}}\left(  q^{-1}\right)  =0$ for any polynomial
$q$ having a zero in $D_{-}^{2}$. In particular $\tau_{g\boldsymbol{a}}\left(
z^{m}\right)  =0$ holds for any negative integer $m$.
\end{lemma}

\begin{proof}
Set
\[
\mathcal{Z}=\left\{  \zeta\in D_{-}\text{; \ }1+\varphi_{g\boldsymbol{a}%
}^{\left(  0\right)  }\left(  \zeta\right)  =0\right\}  \text{.}%
\]
$\mathcal{Z}$ is a discrete set on $D_{-}^{1}$, since $1+\varphi
_{g\boldsymbol{a}}^{\left(  0\right)  }\left(  \zeta\right)  \rightarrow1$ as
$\zeta\rightarrow\infty$, and hence $m_{g\boldsymbol{a}}$, $n_{g\boldsymbol{a}%
}$ are meromorphic on $D_{-}^{1}$. First observe for a $g\in\Gamma
_{N}^{\operatorname{real}}\left(  D_{+}\right)  $ such that $\tau
_{\boldsymbol{a}}\left(  g\right)  >0$ the property%
\begin{equation}
\tau_{g\boldsymbol{a}}\left(  r\right)  \geq0 \label{2-33}%
\end{equation}
holds for any rational function $r\in\Gamma_{N}^{\operatorname{real}}\left(
D_{+}\right)  $. This is because $g\in\Gamma_{N}^{\operatorname{real}}\left(
D_{+}\right)  $ can be approximated by rational functions $r_{n}\in\Gamma
_{N}^{\operatorname{real}}\left(  D_{+}\right)  $ due to Lemma \ref{l2-8}, and
(\ref{2-33}) is valid due to the cocycle property and the continuity of
tau-functions:%
\[
\tau_{g\boldsymbol{a}}\left(  r\right)  =\dfrac{\tau_{\boldsymbol{a}}\left(
gr\right)  }{\tau_{\boldsymbol{a}}\left(  g\right)  }=\lim_{n\rightarrow
\infty}\dfrac{\tau_{\boldsymbol{a}}\left(  r_{n}r\right)  }{\tau
_{\boldsymbol{a}}\left(  g\right)  }\geq0\text{.}%
\]
Since $r_{\zeta}$, $zr_{\zeta}$ are rational functions of $\Gamma
_{N}^{\operatorname{real}}\left(  D_{+}\right)  $, (\ref{2-33}) and Lemma
\ref{l2-11} yield%
\[
\left\{
\begin{array}
[c]{l}%
0\leq\tau_{g\boldsymbol{a}}\left(  r_{\zeta}\right)  =\left\vert
1+\varphi_{g\boldsymbol{a}}^{\left(  0\right)  }\left(  \zeta\right)
\right\vert ^{2}\dfrac{\operatorname{Im}m_{g\boldsymbol{a}}\left(
\zeta\right)  }{\operatorname{Im}\zeta}\\
0\leq\tau_{g\boldsymbol{a}}\left(  zr_{\zeta}\right)  =\left\vert
1+\varphi_{g\boldsymbol{a}}^{\left(  0\right)  }\left(  \zeta\right)
\right\vert ^{2}\dfrac{\operatorname{Im}n_{g\boldsymbol{a}}\left(
\zeta\right)  }{\operatorname{Im}\zeta^{-1}}\text{ \ (}zq_{\zeta}%
=-\zeta\widetilde{q}_{\zeta^{-1}}\text{)}%
\end{array}
\right.  \text{ \ on }D_{-}^{1}\backslash\mathcal{Z}\text{.}%
\]
Hence
\[
\operatorname{Im}m_{g\boldsymbol{a}}\left(  \zeta\right)  \geq0\text{,
\ }\operatorname{Im}n_{g\boldsymbol{a}}\left(  \zeta\right)  \leq0\text{ \ on
}\left(  D_{-}^{1}\cap\mathbb{C}_{+}\right)  \backslash\mathcal{Z}%
\]
holds. On the other hand, Lemmas \ref{l2-5}, \ref{l2-9} imply
\begin{align*}
&  \dfrac{\left(  \zeta+\varphi_{g\boldsymbol{a}}^{\left(  -1\right)  }\left(
\zeta^{-1}\right)  \right)  \left(  1+\varphi_{g\boldsymbol{a}}^{\left(
0\right)  }\left(  \zeta\right)  \right)  -\left(  \zeta^{-1}+\varphi
_{g\boldsymbol{a}}^{\left(  -1\right)  }\left(  \zeta\right)  \right)  \left(
1+\varphi_{g\boldsymbol{a}}^{\left(  0\right)  }\left(  \zeta^{-1}\right)
\right)  }{\zeta-\zeta^{-1}}\\
&  =\tau_{g\boldsymbol{a}}\left(  q_{\zeta}\widetilde{q}_{\zeta}\right)
\neq0\text{,}%
\end{align*}
which shows $\zeta_{0}^{-1}+\varphi_{g\boldsymbol{a}}^{\left(  -1\right)
}(\zeta_{0})\neq0$ holds for $\zeta_{0}\in\mathcal{Z}\cap D_{-}^{1}$ and
$n_{g\boldsymbol{a}}\left(  \zeta\right)  $ has a pole at $\zeta_{0}$. Since
$\operatorname{Im}n_{g\boldsymbol{a}}\left(  \zeta\right)  ^{-1}\geq0$ and
$n_{g\boldsymbol{a}}\left(  \zeta_{0}\right)  ^{-1}=0$ are valid, the maximum
principle for harmonic functions implies $n_{g\boldsymbol{a}}\left(
\zeta\right)  ^{-1}=0$ identically on $D_{-}^{1}\cap\mathbb{C}_{+}$, which is
impossible. Therefore, $D_{-}^{1}\cap\mathbb{C}_{+}\cap\mathcal{Z}%
=\varnothing$ should hold. Hence, one has%
\[
1+\varphi_{g\boldsymbol{a}}^{\left(  0\right)  }\left(  \zeta\right)
\neq0\text{ \ on }D_{-}^{1}%
\]
due to $1+\overline{\varphi_{g\boldsymbol{a}}^{\left(  0\right)  }\left(
\overline{\zeta}\right)  }=1+\varphi_{g\boldsymbol{a}}^{\left(  0\right)
}\left(  \zeta\right)  $. A similar argument for $m_{g\boldsymbol{a}}$ yields
$\operatorname{Im}m_{g\boldsymbol{a}}\left(  \zeta\right)  >0$ on \ $D_{-}%
^{1}\cap\mathbb{C}_{+}$, since the property $m_{g\boldsymbol{a}}\left(
\zeta\right)  -\zeta\rightarrow0$ as $\zeta\rightarrow\infty$ does not admit
the possibility of $m_{g\boldsymbol{a}}\left(  \zeta\right)  $ being a
constant identically on $D_{-}^{1}\cap\mathbb{C}_{+}$, which shows (i).

\ \ $\ \ \ $(ii) of Lemma \ref{l2-9} for $n=0$ implies%
\[
z+\varphi_{g\boldsymbol{a}}^{\left(  1\right)  }\left(  z\right)  =\left(
z+z^{-1}-\varphi_{\widetilde{g\boldsymbol{a}}}^{\left(  -1\right)  }\left(
0\right)  \right)  \left(  1+\varphi_{g\boldsymbol{a}}^{\left(  0\right)
}\left(  z\right)  \right)  \text{ on }D_{-}\text{,}%
\]
if $1+\varphi_{g\boldsymbol{a}}^{\left(  0\right)  }\left(  0\right)
=\tau_{g\boldsymbol{a}}\left(  z^{-1}\right)  =0$ holds. Hence, if
$1+\varphi_{g\boldsymbol{a}}^{\left(  0\right)  }\left(  z\right)  $ is not
identically $0$ on $D_{-}^{2}$, then the above identity implies
$m_{g\boldsymbol{a}}\left(  z\right)  =z+z^{-1}$ identically on $D_{-}^{2}$,
which contradicts $\operatorname{Im}m_{g\boldsymbol{a}}\left(  z\right)
/\operatorname{Im}z\geq0$ due to $\tau_{g\boldsymbol{a}}\left(  r_{z}\right)
\geq0$, which implies $1+\varphi_{g\boldsymbol{a}}^{\left(  0\right)  }\left(
z\right)  =0$ identically on $D_{-}^{2}$. Then, (iv) of Lemma \ref{l2-11}
yields $\tau_{g\boldsymbol{a}}\left(  q^{-1}\right)  =0$. Setting
$q^{-1}=q_{\zeta_{1}}q_{\zeta_{2}}\cdots q_{\zeta_{m}}$ with $\zeta_{j}\in
D_{-}^{2}$, one has $\tau_{g\boldsymbol{a}}\left(  q_{\zeta_{1}}q_{\zeta_{2}%
}\cdots q_{\zeta_{m}}\right)  =0$. Then, letting $\zeta_{j}\rightarrow0$ shows
$\tau_{g\boldsymbol{a}}\left(  z^{-m}\right)  =0$. Consequently, we have (ii).
$\bigskip$
\end{proof}

\begin{lemma}
\label{l2-15}For $\boldsymbol{a}\in\boldsymbol{A}_{N,+}^{inv}\left(  C\right)
$ and $g\in\Gamma_{N}^{\operatorname{real}}\left(  D_{+}\right)  $ assume
$\tau_{\boldsymbol{a}}\left(  g\right)  >0$. Then, $\tau_{g\boldsymbol{a}%
}\left(  g_{1}g_{2}\right)  >0$ holds for $g_{j}\in\Gamma_{N}%
^{\operatorname{real}}\left(  D_{+}^{j}\right)  $ for $j=1$, $2$.
\end{lemma}

\begin{proof}
Let $\zeta_{1}$, $\zeta_{2}\in D_{-}^{1}$. Since $\tau_{\boldsymbol{a}}\left(
r_{\zeta_{1}}g\right)  =\tau_{\boldsymbol{a}}\left(  g\right)  \tau
_{g\boldsymbol{a}}\left(  r_{\zeta_{1}}\right)  >0$ holds due to (i) of Lemma
\ref{l2-14}, \ applying this Lemma to $r_{\zeta_{1}}g$ again, one has
$\tau_{r_{\zeta_{1}}g\boldsymbol{a}}\left(  r_{\zeta_{2}}\right)  >0$. Then,
for $r=r_{\zeta_{1}}r_{\zeta_{2}}$ one sees
\[
\tau_{g\boldsymbol{a}}\left(  r_{\zeta_{1}}r_{\zeta_{2}}\right)
=\tau_{g\boldsymbol{a}}\left(  r_{\zeta_{1}}\right)  \tau_{r_{\zeta_{1}%
}g\boldsymbol{a}}\left(  r_{\zeta_{2}}\right)  >0\text{,}%
\]
and inductively one can know $\tau_{g\boldsymbol{a}}\left(  r_{\zeta_{1}%
}r_{\zeta_{2}}\cdots r_{\zeta_{m}}\right)  >0$ if $\zeta_{j}\in D_{-}^{1}$ for
$j=1$,$\cdots$, $m$. A rational function $r\in\Gamma_{N}^{\operatorname{real}%
}\left(  D_{+}^{1}\right)  $ can be expressed by $r=$ $r_{1}r_{2}^{-1}$ with%
\[
r_{1}=r_{\zeta_{1}}r_{\zeta_{2}}\cdots r_{\zeta_{m}}\text{, \ }r_{2}%
=r_{\eta_{1}}r_{\eta_{2}}\cdots r_{\eta_{n}}\text{ for }\zeta_{j}\text{, }%
\eta_{j}\in D_{-}^{1}\text{.}%
\]
Then, Lemma \ref{l2-5} and the property $\widetilde{\boldsymbol{a}}%
\in\boldsymbol{A}_{N,+}^{inv}\left(  C\right)  $ imply
\[
\tau_{r_{1}g\boldsymbol{a}}\left(  r_{2}^{-1}\right)  =\tau_{r_{1}%
g\boldsymbol{a}}\left(  r_{2}^{-1}\widetilde{r}_{2}^{-1}\widetilde{r}%
_{2}\right)  =\tau_{r_{1}g\boldsymbol{a}}\left(  r_{2}^{-1}\widetilde{r}%
_{2}^{-1}\right)  \tau_{r_{1}g\boldsymbol{a}}\left(  \widetilde{r}_{2}\right)
\text{.}%
\]
Note $\tau_{r_{1}g\boldsymbol{a}}\left(  r_{2}^{-1}\widetilde{r}_{2}%
^{-1}\right)  >0$. The property $\tau_{r_{1}g\boldsymbol{a}}\left(
\widetilde{r}_{2}\right)  =\tau_{\widetilde{r_{1}g}\widetilde{\boldsymbol{a}}%
}\left(  r_{2}\right)  >0$ can be shown by applying the above argument to
$\widetilde{\boldsymbol{a}}$ and $\widetilde{r_{1}g}$. Therefore, one has
\begin{equation}
\tau_{g\boldsymbol{a}}\left(  r\right)  =\tau_{g\boldsymbol{a}}\left(
r_{1}\right)  \tau_{r_{1}g\boldsymbol{a}}\left(  r_{2}^{-1}\right)  >0\text{.}
\label{2-34}%
\end{equation}
Lemma \ref{l2-8} makes it possible to approximate $g_{1}\in\Gamma
_{N}^{\operatorname{real}}\left(  D_{+}^{1}\right)  $ by rational functions
$r_{n}\in\Gamma_{N}^{\operatorname{real}}\left(  D_{+}^{1}\right)  $. Since
$\tau_{g\boldsymbol{a}}\left(  g_{1}r_{n}^{-1}\right)  \rightarrow
\tau_{g\boldsymbol{a}}\left(  1\right)  =1$ due to $\mathrm{d}_{c}\left(
g_{1}r_{n}^{-1},1\right)  \rightarrow0$, one has $\tau_{g\boldsymbol{a}%
}\left(  g_{1}r_{n_{1}}^{-1}\right)  >0$ for a sufficiently large $n_{1}$.
Applying (\ref{2-34}) to $g_{1}r_{n_{1}}^{-1}g$, $r_{n_{1}}$ we have
$\tau_{g_{1}r_{n_{1}}^{-1}g\boldsymbol{a}}\left(  r_{n_{1}}\right)  >0$. Thus,
the cocycle property yields%
\[
\tau_{g\boldsymbol{a}}\left(  g_{1}\right)  =\tau_{g\boldsymbol{a}}\left(
g_{1}r_{n_{1}}^{-1}r_{n_{1}}\right)  =\tau_{g\boldsymbol{a}}\left(
g_{1}r_{n_{1}}^{-1}\right)  \tau_{g_{1}r_{n_{1}}^{-1}g\boldsymbol{a}}\left(
r_{n_{1}}\right)  >0\text{.}%
\]
The properties $\boldsymbol{a}\in\boldsymbol{A}_{N,+}^{inv}\left(  C\right)
\Longrightarrow\widetilde{\boldsymbol{a}}\in\boldsymbol{A}_{N,+}^{inv}\left(
C\right)  $ and $\tau_{\widetilde{\boldsymbol{a}}}\left(  \widetilde
{g}\right)  =\tau_{\boldsymbol{a}}\left(  g\right)  $ yield%
\[
\tau_{g\boldsymbol{a}}\left(  g_{2}\right)  =\tau_{\widetilde{g}%
\widetilde{\boldsymbol{a}}}\left(  \widetilde{g}_{2}\right)  >0\text{ for
}g_{2}\in\Gamma_{N}^{\operatorname{real}}\left(  D_{+}^{2}\right)  \text{,}%
\]
since $\widetilde{g}_{2}\in\Gamma_{N}^{\operatorname{real}}\left(  D_{+}%
^{1}\right)  $. Observing $\tau_{\boldsymbol{a}}\left(  gg_{1}\right)
=\tau_{\boldsymbol{a}}\left(  g\right)  \tau_{g\boldsymbol{a}}\left(
g_{1}\right)  >0$, one has $\tau_{g\boldsymbol{a}}\left(  g_{1}g_{2}\right)
=\tau_{g\boldsymbol{a}}\left(  g_{1}\right)  \tau_{g_{1}g\boldsymbol{a}%
}\left(  g_{2}\right)  >0$.\bigskip
\end{proof}

The following proposition is our first main result.

\begin{proposition}
\label{p1}Let $\boldsymbol{a}\in\boldsymbol{A}_{N,++}^{inv}\left(  C\right)
$. Then, one has $\tau_{\boldsymbol{a}}\left(  g\right)  >0$ for any
$g\in\Gamma_{N}^{\operatorname{real}}\left(  D_{+}\right)  $, hence
$g^{-1}T\left(  g\boldsymbol{a}\right)  $ is invertible on $H_{N,c}\left(
D_{+}\right)  $ for any $c>\mathrm{\delta}_{D_{+}}\left(  g\right)  $.
\end{proposition}

\begin{proof}
A general $g\in\Gamma_{N}^{\operatorname{real}}\left(  D_{+}\right)  $ has an
expression $g=z^{n}g_{1}g_{2}$ thanks to Lemma \ref{l2-7}. Hence, applying
Lemma \ref{l2-15} to $g=z^{n}$, we have immediately%
\[
\tau_{\boldsymbol{a}}\left(  g\right)  =\tau_{\boldsymbol{a}}\left(
z^{n}\right)  \tau_{z^{n}\boldsymbol{a}}\left(  g_{1}g_{2}\right)  >0\text{.}%
\]
\bigskip
\end{proof}

\begin{remark}
\label{r1}There is an example $\boldsymbol{a}\in\boldsymbol{A}_{N,+}%
^{inv}\left(  C\right)  $ but $\boldsymbol{a}\notin\boldsymbol{A}_{N,++}%
^{inv}\left(  C\right)  $. For such an $\boldsymbol{a}\in\boldsymbol{A}%
_{N,+}^{inv}\left(  C\right)  $ there exists $n_{1}\in\mathbb{Z}$ such that
$\tau_{\boldsymbol{a}}\left(  z^{n_{1}}\right)  =0$. Then, (ii) of Lemma
\ref{l2-14} implies%
\[
n_{1}\geq1\Longrightarrow\tau_{\boldsymbol{a}}\left(  z^{n}\right)  =0\text{
if }n\geq n_{1}\text{, and }n_{1}\leq-1\Longrightarrow\tau_{\boldsymbol{a}%
}\left(  z^{n}\right)  =0\text{ if }n\leq n_{1}\text{.}%
\]

\end{remark}

\section{Identification of $\boldsymbol{A}_{N,++}^{inv}\left(  C\right)  $
with $Q_{N}$ and Toda flow}

In this section we characterize the important set $\boldsymbol{A}_{N,++}%
^{inv}\left(  C\right)  $ in terms of Jacobi operators, which will make it
possible to construct the Toda flow on $Q_{N}$.

To describe $m_{\boldsymbol{a}}$ for $\boldsymbol{a}\in\boldsymbol{A}%
_{N,++}^{inv}\left(  C\right)  $ we introduce
\begin{equation}
\mathcal{M}=\left\{
\begin{array}
[c]{l}%
m\text{; }m\text{ is analytic on }\mathbb{C}\backslash\Sigma\text{ satisfying
}m=\overline{m}\text{ and}\\
\text{(i)}\operatorname{Im}m(z)>0\text{ \ on }\mathbb{C}_{+}\backslash
\text{\ }\Sigma\\
\text{(ii) }m(z)\text{ is not a rational function of }\phi\left(  z\right)
\text{ neither on }\left\{  \left\vert z\right\vert \lessgtr1\right\}
\text{.}\\
\text{(iii) For any }n\geq1\text{ it holds that}\\
\text{ \ }m(z)=z+\sum_{1\leq k\leq n-1}m_{+,k}z^{-k}+O\left(  z^{-n}\right)
\text{ as }z\rightarrow\infty\text{ on }i\mathbb{R}\\
\text{ \ }m(z)=\sum_{0\leq k\leq n-1}m_{-,k}z^{k}+O\left(  z^{n}\right)
\text{ \ \ \ \ \ \ \ \ as }z\rightarrow0\text{ on }i\mathbb{R}%
\end{array}
\right\}  \text{, } \label{3-1}%
\end{equation}
where $\Sigma=\mathbb{R}\cup\left\{  \left\vert z\right\vert =1\right\}  $,
$\phi\left(  z\right)  =z+z^{-1}$. We have the following characterization of
$\mathcal{M}$.

\begin{lemma}
\label{l3-1}An analytic function $m$ on $\mathbb{C}\backslash\Sigma$ is an
element of $\mathcal{M}$ if and only if $m$ has an expression%
\[
m(z)=\left\{
\begin{array}
[c]{c}%
z+z^{-1}+a_{1}^{2}%
%TCIMACRO{\dint _{\mathbb{R}}}%
%BeginExpansion
{\displaystyle\int_{\mathbb{R}}}
%EndExpansion
\dfrac{\sigma_{+}\left(  d\lambda\right)  }{\lambda-\left(  z+z^{-1}\right)
}\text{ \ \ if \ }\left\vert z\right\vert >1\\
-a_{0}^{2}%
%TCIMACRO{\dint _{\mathbb{R}}}%
%BeginExpansion
{\displaystyle\int_{\mathbb{R}}}
%EndExpansion
\dfrac{\sigma_{-}\left(  d\lambda\right)  }{\lambda-\left(  z+z^{-1}\right)
}+b_{0}\text{ \ \ \ \ \ \ \ if \ }\left\vert z\right\vert <1
\end{array}
\right.
\]
with positive numbers $a_{0}$, $a_{1}$, a real number $b_{0}$ and probability
measures $\sigma_{\pm}$ on $\mathbb{R}$ ($\sigma_{\pm}\left(  \mathbb{R}%
\right)  =1$) satisfying%
\begin{equation}
\int_{\mathbb{R}}\left\vert \lambda\right\vert ^{n}\sigma_{\pm}\left(
d\lambda\right)  <\infty\text{\ for any }n\geq0\text{ and \textrm{supp}}%
\sigma_{\pm}\text{ are infinite.} \label{3-2}%
\end{equation}

\end{lemma}

\begin{proof}
$\phi\left(  z\right)  $ is a conformal map:%
\[
\phi:\text{ }\mathbb{C}_{+}\cap\left\{  \left\vert z\right\vert >1\right\}
\rightarrow\mathbb{C}_{+}\text{, \ \ }\mathbb{C}_{+}\cap\left\{  \left\vert
z\right\vert <1\right\}  \rightarrow\mathbb{C}_{-}\text{.}%
\]
Therefore, analytic functions $\pm m\left(  \phi^{-1}\left(  z\right)
\right)  $ on $\mathbb{C}_{\pm}$ satisfy $\pm\operatorname{Im}m\left(
\phi^{-1}\left(  z\right)  \right)  >0$ on $\mathbb{C}_{\pm}$ due to (i), and
the Herglotz representation theorem yields%
\[
m\left(  z\right)  =\text{ }\left\{
\begin{array}
[c]{c}%
\alpha_{+}+\beta_{+}\phi\left(  z\right)  +%
%TCIMACRO{\dint _{\mathbb{R}}}%
%BeginExpansion
{\displaystyle\int_{\mathbb{R}}}
%EndExpansion
\left(  \dfrac{1}{\lambda-\phi\left(  z\right)  }-\dfrac{\lambda}{\lambda
^{2}+1}\right)  \sigma_{+}\left(  d\lambda\right)  \text{ \ if }\left\vert
z\right\vert >1\\
-\alpha_{-}-\beta_{-}\phi\left(  z\right)  +%
%TCIMACRO{\dint _{\mathbb{R}}}%
%BeginExpansion
{\displaystyle\int_{\mathbb{R}}}
%EndExpansion
\left(  \dfrac{1}{\lambda-\phi\left(  z\right)  }-\dfrac{\lambda}{\lambda
^{2}+1}\right)  \sigma_{-}\left(  d\lambda\right)  \text{ \ if }\left\vert
z\right\vert <1
\end{array}
\right.  \text{.}%
\]
Then, the property (iii) implies%
\[
\left\{
\begin{array}
[c]{l}%
\sigma_{+}\left(  \mathbb{R}\right)  <\infty\text{, }\alpha_{+}-%
%TCIMACRO{\dint _{\mathbb{R}}}%
%BeginExpansion
{\displaystyle\int_{\mathbb{R}}}
%EndExpansion
\dfrac{\lambda}{\lambda^{2}+1}\sigma_{+}\left(  d\lambda\right)  =0\text{,
}\beta_{+}=1\\
\sigma_{-}\left(  \mathbb{R}\right)  <\infty\text{, }\beta_{-}=0
\end{array}
\right.  \text{. }%
\]
The property (ii) shows $\sigma_{\pm}\left(  \mathbb{R}\right)  >0$, hence we
define $a_{1}^{2}=\sigma_{+}\left(  \mathbb{R}\right)  $, $a_{0}^{2}%
=\sigma_{-}\left(  \mathbb{R}\right)  $ and redefine $\sigma_{\pm}$ by
$a_{1}^{-2}\sigma_{+}$, $a_{0}^{-2}\sigma_{-}$. The property (\ref{3-2})
follows from (iii) easily. Conversely, if $m$ is given by (\ref{3-1}), it is
clear that $m\in\mathcal{M}$. \bigskip
\end{proof}

Our next task is to give a sufficient condition for $\boldsymbol{a}%
\in\boldsymbol{A}_{N}^{inv}\left(  C\right)  $ to be an element of
$\boldsymbol{A}_{N,++}^{inv}\left(  C\right)  $. The definition of $d_{\eta}$
is in (\ref{2-28}).

\begin{lemma}
\label{l3-2}$\mathcal{M}$ is closed under the following operations:%
\[
\phi\left(  z\right)  -\frac{\lim_{\eta\rightarrow\infty}\eta\left(
\phi\left(  \eta\right)  -m(\eta)\right)  }{\phi\left(  z\right)  -m\left(
z^{-1}\right)  }\text{, }\left(  d_{0}m\right)  \left(  z\right)  \text{,
}\left(  d_{\eta}d_{\overline{\eta}}m\right)  \left(  z\right)  \text{.}%
\]

\end{lemma}

\begin{proof}
In any case the properties (ii), (iii) are easily verified, so we check only
the property (i). Lemma \ref{l3-1} yields%
\[
\lim_{\eta\rightarrow\infty}\eta\left(  \phi\left(  \eta\right)
-m(\eta)\right)  =a_{1}^{2}\lim_{\eta\rightarrow\infty}%
%TCIMACRO{\dint _{\mathbb{R}}}%
%BeginExpansion
{\displaystyle\int_{\mathbb{R}}}
%EndExpansion
\dfrac{-\eta\sigma_{+}\left(  d\lambda\right)  }{\lambda-\left(  \eta
+\eta^{-1}\right)  }=a_{1}^{2}>0\text{,}%
\]
hence we have%
\begin{align*}
m_{1}(z)  &  \equiv\phi\left(  z\right)  -\frac{a_{1}^{2}}{\phi\left(
z\right)  -m\left(  z^{-1}\right)  }\\
&  =\left\{
\begin{array}
[c]{ll}%
\phi\left(  z\right)  -a_{1}^{2}\left(  \phi\left(  z\right)  +a_{0}^{2}%
%TCIMACRO{\dint _{\mathbb{R}}}%
%BeginExpansion
{\displaystyle\int_{\mathbb{R}}}
%EndExpansion
\dfrac{\sigma_{-}\left(  d\lambda\right)  }{\lambda-\phi\left(  z\right)
}-b_{0}\right)  ^{-1} & \text{if \ }\left\vert z\right\vert >1\\
\phi\left(  z\right)  +\left(
%TCIMACRO{\dint _{\mathbb{R}}}%
%BeginExpansion
{\displaystyle\int_{\mathbb{R}}}
%EndExpansion
\dfrac{\sigma_{+}\left(  d\lambda\right)  }{\lambda-\phi\left(  z\right)
}\right)  ^{-1} & \text{if \ }\left\vert z\right\vert <1
\end{array}
\right.  \text{.}%
\end{align*}
Since $\operatorname{Im}\phi\left(  z\right)  >0$ for $z\in\mathbb{C}_{+}%
\cap\left\{  \left\vert z\right\vert >1\right\}  $, we easily have
$\operatorname{Im}m_{1}(z)>0$ there. To show $\operatorname{Im}m_{1}(z)>0$ on
$\mathbb{C}_{+}\cap\left\{  \left\vert z\right\vert <1\right\}  $ observe
\[
-\left(
%TCIMACRO{\dint _{\mathbb{R}}}%
%BeginExpansion
{\displaystyle\int_{\mathbb{R}}}
%EndExpansion
\dfrac{\sigma_{+}\left(  d\lambda\right)  }{\lambda-z}\right)  ^{-1}=z+\alpha+%
%TCIMACRO{\dint _{\mathbb{R}}}%
%BeginExpansion
{\displaystyle\int_{\mathbb{R}}}
%EndExpansion
\dfrac{\nu\left(  d\lambda\right)  }{\lambda-z}%
\]
with $\alpha\in\mathbb{R}$ and a finite measure $\nu$ on $\mathbb{R}$. This is
due to the fact that $\sigma_{+}$ is a probability measure having every
moment. Then we have $\operatorname{Im}m_{1}(z)>0$ on $\mathbb{C}_{+}%
\cap\left\{  \left\vert z\right\vert <1\right\}  $ since $\operatorname{Im}%
\phi\left(  z\right)  <0$ there.

To verify $d_{0}m\in\mathcal{M}$ for $m\in\mathcal{M}$ note from Lemma
\ref{l3-1}%
\[
m\left(  0\right)  =b_{0}\in\mathbb{R}\text{, \ }m^{\prime}(0)=a_{0}%
^{2}>0\text{.}%
\]
Then, similarly to the case $m_{1}$ we easily have $\operatorname{Im}%
d_{0}m\left(  z\right)  >0$ on $\mathbb{C}_{+}\backslash\Sigma$. The most
complicated case $d_{\eta}d_{\overline{\eta}}m$ relies on Lemma 25 of
\cite{z}.\bigskip
\end{proof}

For $\boldsymbol{a}\in\boldsymbol{A}_{N}^{inv}\left(  C\right)  $ set%
\begin{equation}
\boldsymbol{A}_{\mathcal{M}}=\left\{  \boldsymbol{a}\in\boldsymbol{A}%
_{N}^{inv}\left(  C\right)  \text{; \ }\boldsymbol{a}=\overline{\boldsymbol{a}%
}\text{ and }m_{\boldsymbol{a}}\in\mathcal{M}\right\}  \label{3-3}%
\end{equation}
The precise meaning of $m_{\boldsymbol{a}}\in\mathcal{M}$ is that there exists
$m\in\mathcal{M}$ such that $m_{\boldsymbol{a}}=m$ on $D_{-}$. For
$\boldsymbol{a}\in\boldsymbol{A}_{\mathcal{M}}$ set
\[
\Gamma_{\boldsymbol{a}}=\left\{  g\in\Gamma_{N}^{\operatorname{real}}\left(
D_{+}\right)  \text{; }\tau_{\boldsymbol{a}}\left(  g\right)  >0\text{ and
}m_{g\boldsymbol{a}}\in\mathcal{M}\right\}  \text{.}%
\]
Clearly, we have $1\in\Gamma_{\boldsymbol{a}}$. Since $1+\varphi
_{g\boldsymbol{a}}^{\left(  0\right)  }\left(  \zeta\right)  \rightarrow1$ as
$\zeta\rightarrow\infty$ in $D_{-}$, $m_{g\boldsymbol{a}}$ is meromorphic at
least on $D_{-}^{1}$.

\begin{lemma}
\label{l3-3}$\boldsymbol{A}_{\mathcal{M}}\subset\boldsymbol{A}_{N,++}%
^{inv}\left(  C\right)  $ holds.
\end{lemma}

\begin{proof}
The proof proceeds similarly to those of Lemma \ref{l2-14} and Proposition
\ref{p1}. First observe that (\ref{2-29}) and Lemma \ref{l3-2} implies%
\begin{equation}
\ \widetilde{\boldsymbol{a}}\in\boldsymbol{A}_{\mathcal{M}}\text{ and
}\widetilde{g}\in\Gamma_{\widetilde{\boldsymbol{a}}}\text{ for }%
\boldsymbol{a}\in\boldsymbol{A}_{\mathcal{M}}\text{, }g\in\Gamma
_{\boldsymbol{a}}\text{.} \label{3-4}%
\end{equation}
Note that (\ref{2-30}) and Lemma \ref{l3-1} imply%
\[
\tau_{g\boldsymbol{a}}\left(  z\right)  \tau_{g\boldsymbol{a}}\left(
z^{-1}\right)  =\lim_{\zeta\rightarrow\infty}\zeta\left(  \phi\left(
\zeta\right)  -m_{g\boldsymbol{a}}(\zeta)\right)  =a_{1}^{2}>0\text{,}%
\]
hence $\tau_{g\boldsymbol{a}}\left(  z^{\pm1}\right)  >0$. Suppose%
\[
1+\varphi_{g\boldsymbol{a}}^{\left(  0\right)  }\left(  \zeta_{0}\right)
=\zeta_{0}+\varphi_{g\boldsymbol{a}}^{\left(  1\right)  }\left(  \zeta
_{0}\right)  =0
\]
for some $\zeta_{0}\in D_{-}$. Then, one has
\[
\left(  1+\varphi_{g\boldsymbol{a}}^{\left(  0\right)  }\left(  0\right)
\right)  \left(  \zeta_{0}^{-1}+\varphi_{g\boldsymbol{a}}^{\left(  -1\right)
}\left(  \zeta_{0}\right)  \right)  =0
\]
from (ii) of Lemma \ref{l2-9} for $n=0$:
\[
\left(  z+z^{-1}-\varphi_{\widetilde{g\boldsymbol{a}}}^{\left(  -1\right)
}\left(  0\right)  \right)  \left(  1+\varphi_{g\boldsymbol{a}}^{\left(
0\right)  }\right)  =z+\varphi_{g\boldsymbol{a}}^{\left(  1\right)  }+\left(
1+\varphi_{g\boldsymbol{a}}^{\left(  0\right)  }\left(  0\right)  \right)
\left(  z^{-1}+\varphi_{g\boldsymbol{a}}^{\left(  -1\right)  }\right)
\text{.}%
\]
Since $1+\varphi_{g\boldsymbol{a}}^{\left(  0\right)  }\left(  0\right)
=\tau_{g\boldsymbol{a}}\left(  z^{-1}\right)  \neq0$, one has $\zeta_{0}%
^{-1}+\varphi_{g\boldsymbol{a}}^{\left(  -1\right)  }\left(  \zeta_{0}\right)
=0$, which contradicts
\begin{align*}
0  &  \neq\tau_{g\boldsymbol{a}}\left(  q_{\zeta}\widetilde{q}_{\zeta}\right)
\\
&  =\dfrac{\left(  \zeta+\varphi_{g\boldsymbol{a}}^{\left(  -1\right)
}\left(  \zeta^{-1}\right)  \right)  \left(  1+\varphi_{g\boldsymbol{a}%
}^{\left(  0\right)  }\left(  \zeta\right)  \right)  -\left(  \zeta
^{-1}+\varphi_{g\boldsymbol{a}}^{\left(  -1\right)  }\left(  \zeta\right)
\right)  \left(  1+\varphi_{g\boldsymbol{a}}^{\left(  0\right)  }\left(
\zeta^{-1}\right)  \right)  }{\zeta-\zeta^{-1}}\text{.}%
\end{align*}
Hence, $1+\varphi_{g\boldsymbol{a}}^{\left(  0\right)  }$ and $z+\varphi
_{g\boldsymbol{a}}^{\left(  1\right)  }$ do not vanish simultaneously. Since
$m_{g\boldsymbol{a}}=m\in\mathcal{M}$ and $m$ is analytic on $D_{-}$, we have
\[
1+\varphi_{g\boldsymbol{a}}^{\left(  0\right)  }\left(  \zeta\right)
\neq0\text{ \ for any }\zeta\in D_{-}\text{.}%
\]
Recalling the identity%
\[
\tau_{g\boldsymbol{a}}\left(  r_{\zeta}\right)  =\left\vert \tau
_{g\boldsymbol{a}}\left(  q_{\zeta}\right)  \right\vert ^{2}\frac
{\operatorname{Im}m_{g\boldsymbol{a}}\left(  \zeta\right)  }{\operatorname{Im}%
\zeta}\text{ \ \ (}r_{\zeta}=q_{\zeta}q_{\overline{\zeta}}\text{)}%
\]
we have $\tau_{g\boldsymbol{a}}\left(  r_{\zeta}\right)  >0$ on $D_{-}$, since
$\operatorname{Im}m_{g\boldsymbol{a}}\left(  \zeta\right)  =\operatorname{Im}%
m\left(  \zeta\right)  $. On the other hand, the identities
\[
\tau_{g\boldsymbol{a}}\left(  z^{-2}\right)  =\tau_{g\boldsymbol{a}}\left(
z^{-1}\right)  ^{2}m_{g\boldsymbol{a}}^{\prime}(0)=\tau_{g\boldsymbol{a}%
}\left(  z^{-1}\right)  ^{2}m^{\prime}(0)=\tau_{g\boldsymbol{a}}\left(
z^{-1}\right)  ^{2}a_{0}^{2}>0
\]
imply $\tau_{z^{-1}g\boldsymbol{a}}\left(  z^{-1}\right)  >0$, and%
\[
m_{z^{-1}g\boldsymbol{a}}(\zeta)=d_{0}m_{g\boldsymbol{a}}(\zeta)=d_{0}%
m(\zeta)\in\mathcal{M}\text{ (due to Lemma \ref{l3-2})}%
\]
shows%
\begin{equation}
z^{-1}g\in\Gamma_{\boldsymbol{a}}\text{ if }g\in\Gamma_{\boldsymbol{a}}
\label{3-5}%
\end{equation}
and $\tau_{g\boldsymbol{a}}\left(  r_{\zeta}z^{-1}\right)  =\tau
_{g\boldsymbol{a}}\left(  z^{-1}\right)  \tau_{z^{-1}g\boldsymbol{a}}\left(
r_{\zeta}\right)  >0$. Therefore, we have%
\[
\tau_{r_{\zeta}g\boldsymbol{a}}\left(  z^{-1}\right)  =\frac{\tau
_{g\boldsymbol{a}}\left(  r_{\zeta}z^{-1}\right)  }{\tau_{g\boldsymbol{a}%
}\left(  r_{\zeta}\right)  }>0\text{.}%
\]
Since $m_{r_{\zeta}g\boldsymbol{a}}=d_{\zeta}d_{\overline{\zeta}%
}m_{g\boldsymbol{a}}=d_{\zeta}d_{\overline{\zeta}}m\in\mathcal{M}$ is valid
due to Lemma \ref{l3-2}, one has%
\begin{equation}
r_{\zeta}g\in\Gamma_{\boldsymbol{a}}\text{ if }g\in\Gamma_{\boldsymbol{a}%
}\text{.} \label{3-6}%
\end{equation}
Then, applying (\ref{3-5}), (\ref{3-6}) iteratively yields%
\[
\text{ \ }r_{1}\in\Gamma_{\boldsymbol{a}}\text{ and }\tau_{g\boldsymbol{a}%
}\left(  r_{1}\right)  >0
\]
for $r_{1}\equiv z^{-m}r_{\zeta_{1}}r_{\zeta_{2}}\cdots r_{\zeta_{n}}$ with
$m\geq0$, $\zeta_{j}\in D_{-}$. Let $r_{2}\equiv z^{-k}r_{\eta_{1}}r_{\eta
_{2}}\cdots r_{\eta_{\ell}}$ with $k\geq0$, $\eta_{j}\in D_{-}$. Then%
\begin{align*}
\tau_{g\boldsymbol{a}}\left(  r_{1}r_{2}^{-1}\right)   &  =\tau
_{g\boldsymbol{a}}\left(  r_{1}\right)  \tau_{r_{1}g\boldsymbol{a}}\left(
r_{2}^{-1}\right)  =\tau_{g\boldsymbol{a}}\left(  r_{1}\right)  \tau
_{r_{1}g\boldsymbol{a}}\left(  r_{2}^{-1}\widetilde{r}_{2}^{-1}\widetilde
{r}_{2}\right) \\
&  =\tau_{g\boldsymbol{a}}\left(  r_{1}\right)  \tau_{r_{1}g\boldsymbol{a}%
}\left(  r_{2}^{-1}\widetilde{r}_{2}^{-1}\right)  \tau_{\widetilde
{r_{1}g\boldsymbol{a}}}\left(  r_{2}\right)  >0
\end{align*}
hold, where in the last identity we have used (\ref{3-4}). This implies
$\tau_{\boldsymbol{a}}\left(  r\right)  >0$ for any rational $r\in\Gamma
_{N}^{\operatorname{real}}\left(  D_{+}\right)  $ by setting $g=1$, which
completes the proof.\bigskip
\end{proof}

Now $m_{\boldsymbol{a}}\in\mathcal{M}$ is sufficient for $\boldsymbol{a}%
\in\boldsymbol{A}_{N}^{inv}\left(  C\right)  $ to be an element of
$\boldsymbol{A}_{N,++}^{inv}\left(  C\right)  $. We show below it is also
necessary. Our strategy is to employ the Jacobi operator associated with
$\boldsymbol{a}\in\boldsymbol{A}_{N,++}^{inv}\left(  C\right)  $ and to
describe $m_{\boldsymbol{a}}$ by its Weyl functions.

For $\boldsymbol{a}\in\boldsymbol{A}_{N,++}^{inv}\left(  C\right)  $,
$g\in\Gamma_{N}^{\operatorname{real}}\left(  D_{+}\right)  $ define $\left\{
a_{n}\left(  g\boldsymbol{a}\right)  ,b_{n}\left(  g\boldsymbol{a}\right)
\right\}  $ by (\ref{1-13}), namely%
\[
\left\{
\begin{array}
[c]{l}%
a_{n}\left(  g\boldsymbol{a}\right)  =\sqrt{\dfrac{\tau_{g\boldsymbol{a}%
}\left(  z^{n}\right)  \tau_{g\boldsymbol{a}}\left(  z^{n-2}\right)  }%
{\tau_{g\boldsymbol{a}}\left(  z^{n-1}\right)  ^{2}}}\\
b_{n}\left(  g\boldsymbol{a}\right)  =\dfrac{\left.  \partial_{\varepsilon
}\tau_{g\boldsymbol{a}}\left(  z^{n}q_{\varepsilon^{-1}}\right)  \right\vert
_{\varepsilon=0}}{\tau_{g\boldsymbol{a}}\left(  z^{n}\right)  }-\dfrac{\left.
\partial_{\varepsilon}\tau_{g\boldsymbol{a}}\left(  z^{n-1}q_{\varepsilon
^{-1}}\right)  \right\vert _{\varepsilon=0}}{\tau_{g\boldsymbol{a}}\left(
z^{n-1}\right)  }\\
\text{ \ \ \ \ \ \ \ \ }=\varphi_{\widetilde{z^{n}g\boldsymbol{a}}}^{\left(
-1\right)  }\left(  0\right)  -\varphi_{\widetilde{z^{n-1}g\boldsymbol{a}}%
}^{\left(  -1\right)  }\left(  0\right)
\end{array}
\right.
\]
The last identity obeys from%
\[
\tau_{g\boldsymbol{a}}\left(  z^{n}q_{\varepsilon^{-1}}\right)  =\tau
_{g\boldsymbol{a}}\left(  z^{n}\right)  \tau_{z^{n}g\boldsymbol{a}}\left(
q_{\varepsilon^{-1}}\right)  =\tau_{g\boldsymbol{a}}\left(  z^{n}\right)
\left(  1+\varepsilon\varphi_{\widetilde{z^{n}g\boldsymbol{a}}}^{\left(
-1\right)  }\left(  \varepsilon\right)  \right)  \text{.}%
\]
Since $\tau_{g\boldsymbol{a}}\left(  z^{n}\right)  =\tau_{\boldsymbol{a}%
}\left(  z^{n}g\right)  /\tau_{\boldsymbol{a}}\left(  g\right)  >0$ holds for
$\boldsymbol{a}\in\boldsymbol{A}_{N,++}^{inv}\left(  C\right)  $, $g\in
\Gamma_{N}^{\operatorname{real}}\left(  C\right)  $, the quantities
$a_{n}\left(  g\boldsymbol{a}\right)  $, $b_{n}\left(  g\boldsymbol{a}\right)
$ are well-defined. Set%
\[
q\left(  g\boldsymbol{a}\right)  =\left\{  a_{n}\left(  g\boldsymbol{a}%
\right)  ,b_{n}\left(  g\boldsymbol{a}\right)  \right\}  _{n\in\mathbb{Z}}%
\]
and call it by Jacobi coefficients associated with a symbol $\boldsymbol{a}$
and $g$ since it define a Jacobi operator:%
\[
\left(  H_{q\left(  g\boldsymbol{a}\right)  }u\right)  _{n}=a_{n+1}\left(
g\boldsymbol{a}\right)  u_{n+1}+a_{n}\left(  g\boldsymbol{a}\right)
u_{n-1}+b_{n}\left(  g\boldsymbol{a}\right)  u_{n}\text{.}%
\]
First we show that any exponential function $e^{c\lambda^{N}}$ is integrable
with respect to the spectral measures $\sigma_{\pm}$, which leads us naturally
to the fact that the boundaries $\pm\infty$ are of limit point type. For
$n\in\mathbb{Z}$ set%
\[
f_{n}=f_{n}\left(  \zeta\right)  =\frac{\left(  g\boldsymbol{a}T\left(
z^{n}g\boldsymbol{a}\right)  ^{-1}1\right)  \left(  \zeta\right)  }{\sqrt
{\tau_{z^{n}g\boldsymbol{a}}\left(  z^{-1}\right)  }}=\zeta^{-n}%
\frac{1+\varphi_{z^{n}g\boldsymbol{a}}^{\left(  0\right)  }\left(
\zeta\right)  }{\sqrt{1+\varphi_{z^{n}g\boldsymbol{a}}^{\left(  0\right)
}\left(  0\right)  }}\text{.\ }%
\]
Then, it holds that
\begin{equation}
\left(  H_{q\left(  g\boldsymbol{a}\right)  }f\right)  _{n}=\left(
\zeta+\zeta^{-1}\right)  f_{n}\text{.} \label{3-7}%
\end{equation}
To show this identity we describe $f_{n}$ by tau-functions. For simplicity of
notations set $\boldsymbol{b}=z^{n}g\boldsymbol{a}$. Then, $f_{n}$ and the
coefficients can be written as%
\begin{align*}
f_{n+1}  &  =\zeta^{-n-1}\dfrac{\tau_{z\boldsymbol{b}}\left(  q_{\zeta
}\right)  }{\sqrt{\tau_{z\boldsymbol{b}}\left(  z^{-1}\right)  }}\text{,
}f_{n}=\zeta^{-n}\dfrac{\tau_{\boldsymbol{b}}\left(  q_{\zeta}\right)  }%
{\sqrt{\tau_{\boldsymbol{b}}\left(  z^{-1}\right)  }}\text{, }f_{n-1}%
=\zeta^{-n+1}\dfrac{\tau_{z^{-1}\boldsymbol{b}}\left(  q_{\zeta}\right)
}{\sqrt{\tau_{z^{-1}\boldsymbol{b}}\left(  z^{-1}\right)  }}\text{\ ,}\\
a_{n}  &  =\sqrt{\dfrac{\tau_{z^{-1}\boldsymbol{b}}\left(  z^{-1}\right)
}{\tau_{\boldsymbol{b}}\left(  z^{-1}\right)  }}\text{, \ }a_{n+1}=\sqrt
{\tau_{\boldsymbol{b}}\left(  z\right)  \tau_{\boldsymbol{b}}\left(
z^{-1}\right)  }\text{, \ }b_{n}=\varphi_{\widetilde{\boldsymbol{b}}}^{\left(
-1\right)  }\left(  0\right)  -\varphi_{\widetilde{z^{-1}\boldsymbol{b}}%
}^{\left(  -1\right)  }\left(  0\right)  \text{,}%
\end{align*}
and the identity (\ref{3-7}) is equivalent to%
\begin{align}
&  \tau_{\boldsymbol{b}}\left(  z^{-1}\right)  \zeta^{-1}\tau_{\boldsymbol{b}%
}\left(  zq_{\zeta}\right)  +\dfrac{\zeta\tau_{\boldsymbol{b}}\left(
z^{-1}q_{\zeta}\right)  }{\tau_{\boldsymbol{b}}\left(  z^{-1}\right)
}+\left(  \varphi_{\widetilde{\boldsymbol{b}}}^{\left(  -1\right)  }\left(
0\right)  -\varphi_{\widetilde{z^{-1}\boldsymbol{b}}}^{\left(  -1\right)
}\left(  0\right)  \right)  \tau_{\boldsymbol{b}}\left(  q_{\zeta}\right)
\nonumber\\
&  =\left(  \zeta+\zeta^{-1}\right)  \tau_{\boldsymbol{b}}\left(  q_{\zeta
}\right)  \text{.} \label{3-8}%
\end{align}
Since we have%
\[
\tau_{\boldsymbol{b}}\left(  zq_{\zeta}\right)  =\tau_{\boldsymbol{b}}\left(
q_{\zeta}\right)  \zeta n_{\boldsymbol{b}}\left(  \zeta\right)  \text{,
\ }\dfrac{\tau_{\boldsymbol{b}}\left(  z^{-1}q_{\zeta}\right)  }%
{\tau_{\boldsymbol{b}}\left(  z^{-1}\right)  }=\tau_{\boldsymbol{b}}\left(
q_{\zeta}\right)  \zeta^{-1}\left(  m_{\boldsymbol{b}}\left(  \zeta\right)
-m_{\boldsymbol{b}}\left(  0\right)  \right)  \text{,}%
\]
(\ref{3-8}) turns to%
\[
\tau_{\boldsymbol{b}}\left(  z^{-1}\right)  n_{\boldsymbol{b}}\left(
\zeta\right)  +m_{\boldsymbol{b}}\left(  \zeta\right)  -m_{\boldsymbol{b}%
}\left(  0\right)  +\varphi_{\widetilde{\boldsymbol{b}}}^{\left(  -1\right)
}\left(  0\right)  -\varphi_{\widetilde{z^{-1}\boldsymbol{b}}}^{\left(
-1\right)  }\left(  0\right)  =\zeta+\zeta^{-1}\text{,}%
\]
which is%
\begin{equation}
-m_{\boldsymbol{b}}\left(  0\right)  +\varphi_{\widetilde{\boldsymbol{b}}%
}^{\left(  -1\right)  }\left(  0\right)  -\varphi_{\widetilde{z^{-1}%
\boldsymbol{b}}}^{\left(  -1\right)  }\left(  0\right)  =0\text{,} \label{3-9}%
\end{equation}
since $\tau_{\boldsymbol{b}}\left(  z^{-1}\right)  n_{\boldsymbol{b}}\left(
\zeta\right)  +m_{\boldsymbol{b}}\left(  \zeta\right)  =\zeta+\zeta^{-1}$. Now
we compute $\varphi_{\widetilde{z^{-1}\boldsymbol{b}}}^{\left(  -1\right)
}\left(  0\right)  $ by tau-function
\begin{align*}
\varphi_{\widetilde{z^{-1}\boldsymbol{b}}}^{\left(  -1\right)  }\left(
0\right)   &  =\lim_{\zeta\rightarrow\infty}\zeta\varphi_{z^{-1}%
\boldsymbol{b}}^{\left(  0\right)  }\left(  \zeta\right)  =\lim_{\zeta
\rightarrow\infty}\zeta\left(  \tau_{z^{-1}\boldsymbol{b}}\left(  q_{\zeta
}\right)  -1\right)  =\lim_{\zeta\rightarrow\infty}\zeta\left(  \dfrac
{\tau_{\boldsymbol{b}}\left(  z^{-1}q_{\zeta}\right)  }{\tau_{\boldsymbol{b}%
}\left(  z^{-1}\right)  }-1\right) \\
&  =\lim_{\zeta\rightarrow\infty}\zeta\left(  \zeta^{-1}\left(  1+\varphi
_{\boldsymbol{b}}^{\left(  0\right)  }\left(  \zeta\right)  \right)  \left(
m_{\boldsymbol{b}}\left(  \zeta\right)  -m_{\boldsymbol{b}}\left(  0\right)
\right)  -1\right)  =\varphi_{\widetilde{\boldsymbol{b}}}^{\left(  -1\right)
}\left(  0\right)  -m_{\boldsymbol{b}}\left(  0\right)  \text{,}%
\end{align*}
which leads us to (\ref{3-9}) and we have (\ref{3-7}).

$f_{n}$ for $n=0$, $\pm1$ are related to the $m$-function:%
\begin{equation}
\left\{
\begin{array}
[c]{l}%
\dfrac{f_{1}\left(  \zeta\right)  }{f_{0}\left(  \zeta\right)  }=a_{1}%
\zeta^{-1}\dfrac{1+\varphi_{zg\boldsymbol{a}}^{\left(  0\right)  }\left(
\zeta\right)  }{1+\varphi_{g\boldsymbol{a}}^{\left(  0\right)  }\left(
\zeta\right)  }=a_{1}\dfrac{n_{g\boldsymbol{a}}\left(  \zeta\right)  }%
{\tau_{g\boldsymbol{a}}\left(  z\right)  }=a_{1}^{-1}\left(  \phi\left(
\zeta\right)  -m_{g\boldsymbol{a}}\left(  \zeta\right)  \right) \\
\dfrac{f_{-1}\left(  \zeta\right)  }{f_{0}\left(  \zeta\right)  }=a_{0}%
^{-1}\zeta\dfrac{1+\varphi_{z^{-1}g\boldsymbol{a}}^{\left(  0\right)  }\left(
\zeta\right)  }{1+\varphi_{g\boldsymbol{a}}^{\left(  0\right)  }\left(
\zeta\right)  }=a_{0}^{-1}\left(  m_{g\boldsymbol{a}}\left(  \zeta\right)
-m_{g\boldsymbol{a}}\left(  0\right)  \right)
\end{array}
\right.  \text{,} \label{3-10}%
\end{equation}
and they satisfy

\begin{lemma}
\label{l3-4}For $\boldsymbol{a}\in\boldsymbol{A}_{N,++}^{inv}\left(  C\right)
$, $g\in\Gamma_{N}^{\operatorname{real}}\left(  D_{+}\right)  $ we have
expansions%
\begin{equation}
\left\{
\begin{array}
[c]{l}%
\dfrac{f_{1}\left(  \zeta\right)  }{f_{0}\left(  \zeta\right)  }=\sum_{1\leq
k\leq L-1}\mu_{k}\zeta^{-k}+O\left(  \zeta^{-L}\right)  \text{ \ as }%
\zeta\rightarrow\infty\text{ on }D_{-}\\
\dfrac{f_{-1}\left(  \zeta\right)  }{f_{0}\left(  \zeta\right)  }=\sum_{1\leq
k\leq L-1}\nu_{k}\zeta^{k}+O\left(  \zeta^{L}\right)  \text{\ \ \ \ as }%
\zeta\rightarrow0\text{ on }D_{-}%
\end{array}
\right.  \label{3-11}%
\end{equation}
for any $L\geq1$. For any $c\geq1$, there exists a constant $c_{1}>0$ such
that%
\[
\left\{
\begin{array}
[c]{c}%
\left\vert \mu_{k}\right\vert \leq c_{1}c^{-k}\Gamma\left(  k/N+1\right) \\
\left\vert \nu_{k}\right\vert \leq c_{1}c^{-k}\Gamma\left(  k/N+1\right)
\end{array}
\right.  \text{ \ holds for any }k\geq1\text{.}%
\]

\end{lemma}

\begin{proof}
Recall
\[
\dfrac{f_{1}\left(  \zeta\right)  }{f_{0}\left(  \zeta\right)  }=a_{1}%
\zeta^{-1}\dfrac{1+\varphi_{zg\boldsymbol{a}}^{\left(  0\right)  }\left(
\zeta\right)  }{1+\varphi_{g\boldsymbol{a}}^{\left(  0\right)  }\left(
\zeta\right)  }%
\]
and the expression%
\[
\varphi_{z^{j}g\boldsymbol{a}}^{\left(  0\right)  }\left(  \zeta\right)
=\dfrac{1}{2\pi i}\int_{C}\dfrac{u_{j}\left(  \lambda\right)  }{\zeta-\lambda
}d\lambda\text{, \ }%
\]
\ with $u_{j}\left(  \lambda\right)  =\lambda^{j}g\left(  \lambda\right)
\left(  \boldsymbol{a}\left(  \lambda\right)  -\boldsymbol{f}\left(
\lambda\right)  \right)  \left(  T\left(  z^{j}g\boldsymbol{a}\right)
^{-1}1\right)  \left(  \lambda\right)  $. We have for $L\geq1$ and $j=0$, $1$%
\[
\varphi_{z^{j}g\boldsymbol{a}}^{\left(  0\right)  }\left(  \zeta\right)
=\sum_{1\leq k\leq L}\alpha_{k,j}\zeta^{-k}+\zeta^{-L}\dfrac{1}{2\pi i}%
\int_{C}\dfrac{\lambda^{L}u_{j}\left(  \lambda\right)  }{\lambda-\zeta
}d\lambda
\]
with%
\[
\text{ }\alpha_{k,j}=-\dfrac{1}{2\pi i}\int_{C}\lambda^{k-1}u_{j}\left(
\lambda\right)  d\lambda\text{.}%
\]
Here one can assume that for arbitrary $c>0$ there exists $v_{j}\in
L^{2}\left(  C\right)  $ such that
\[
\left\vert u_{j}\left(  \lambda\right)  \right\vert \leq e^{-c\left(
\left\vert \lambda\right\vert ^{N}+\left\vert \lambda\right\vert ^{-N}\right)
}\left\vert v_{j}\left(  \lambda\right)  \right\vert
\]
holds on $C$ due to $\boldsymbol{a}\in\boldsymbol{A}_{N,++}^{inv}\left(
C\right)  $, which yields without difficulty (see Lemma \ref{a-4}) estimates
for any $c\geq1$%
\[
\left\vert \alpha_{k,j}\right\vert \leq c_{1}c^{-k}\Gamma\left(  k/N+1\right)
\text{ \ for any }k\geq1\text{. }%
\]
Applying Lemma \ref{a-2} completes the proof for $f_{1}/f_{0}$.\bigskip\ The
proof for $f_{-1}/f_{0}$ is similar.
\end{proof}

Then, we can prove an important

\begin{lemma}
\label{l3-5}Suppose $\boldsymbol{a}\in\boldsymbol{A}_{N,++}^{inv}\left(
C\right)  $, $g\in\Gamma_{N}^{\operatorname{real}}\left(  D_{+}\right)  $.
Then the boundaries $\pm\infty$ are of limit point type for $H_{q\left(
g\boldsymbol{a}\right)  }$. The spectral measures $\sigma_{\pm}$ of the Weyl
functions $m_{\pm}$ satisfy%
\[
\int_{-\infty}^{\infty}e^{c\left\vert \lambda\right\vert ^{N}}\sigma_{\pm
}\left(  d\lambda\right)  <\infty\text{ \ for any }c>0\text{.}%
\]
The $m$-function for $g\boldsymbol{a}$ is given by%
\begin{equation}
m_{g\boldsymbol{a}}(z)=\left\{
\begin{array}
[c]{ll}%
z+z^{-1}+a_{1}^{2}m_{+}\left(  z+z^{-1}\right)  & \text{if }\left\vert
z\right\vert >1\\
-a_{0}^{2}m_{-}\left(  z+z^{-1}\right)  +b_{0} & \text{if }\left\vert
z\right\vert <1
\end{array}
\right.  \text{.} \label{3-12}%
\end{equation}

\end{lemma}

\begin{proof}
Set $u_{n}=f_{n+1}/f_{n}$. Then, one has%
\[
u_{n}=\sqrt{\dfrac{\tau_{\boldsymbol{b}}\left(  z^{-1}\right)  }%
{\tau_{z\boldsymbol{b}}\left(  z^{-1}\right)  }}\dfrac{\tau_{z\boldsymbol{b}%
}\left(  q_{\zeta}\right)  }{\zeta\tau_{\boldsymbol{b}}\left(  q_{\zeta
}\right)  }=\sqrt{\dfrac{\tau_{\boldsymbol{b}}\left(  z^{-1}\right)  }%
{\tau_{\boldsymbol{b}}\left(  z\right)  }}n_{\boldsymbol{b}}\left(
\zeta\right)  \text{ \ (}\boldsymbol{b}=z^{n}g\boldsymbol{a}\text{).}%
\]
Since (\ref{3-7}) implies%
\[
\zeta+\zeta^{-1}=a_{n+1}u_{n}+a_{n}u_{n-1}^{-1}+b_{n}\text{,}%
\]
taking the imaginary part yields%
\[
\left\vert u_{n-1}\right\vert ^{2}=a_{n}a_{n+1}^{-1}\dfrac{\operatorname{Im}%
u_{n-1}}{\operatorname{Im}u_{n}}\left(  1-\dfrac{sa_{n+1}^{-1}}%
{\operatorname{Im}u_{n}}\right)  ^{-1}\text{ \ \ (}s=\operatorname{Im}\left(
\zeta+\zeta^{-1}\right)  \text{),}%
\]
and%
\[
\left\vert \dfrac{f_{n}}{f_{0}}\right\vert ^{2}=%
%TCIMACRO{\dprod _{0\leq k\leq n-1}}%
%BeginExpansion
{\displaystyle\prod_{0\leq k\leq n-1}}
%EndExpansion
\left\vert u_{k}\right\vert ^{2}=BA_{n}%
%TCIMACRO{\dprod _{1\leq k\leq n}}%
%BeginExpansion
{\displaystyle\prod_{1\leq k\leq n}}
%EndExpansion
\left(  1+A_{k}\right)  ^{-1}%
\]
with $A_{k}=-s\left(  a_{k+1}\operatorname{Im}u_{k}\right)  ^{-1}$,
$B=-s^{-1}a_{1}\operatorname{Im}u_{0}$. Since $\operatorname{Im}%
n_{\boldsymbol{b}}\left(  \zeta\right)  <0$ for $\zeta\in D_{-}^{1}%
\cap\mathbb{C}_{+}$, one has $A_{k}$, $B>0$, and
\[
\sum_{1\leq n\leq N}\left\vert \dfrac{f_{n}}{f_{0}}\right\vert ^{2}%
=\sum_{1\leq n\leq N}BA_{n}%
%TCIMACRO{\dprod _{1\leq k\leq n}}%
%BeginExpansion
{\displaystyle\prod_{1\leq k\leq n}}
%EndExpansion
\left(  1+A_{k}\right)  ^{-1}=B\left(  1-%
%TCIMACRO{\dprod _{1\leq k\leq N}}%
%BeginExpansion
{\displaystyle\prod_{1\leq k\leq N}}
%EndExpansion
\left(  1+A_{k}\right)  ^{-1}\right)  \leq B
\]
is valid for any $N\geq1$, which implies $-f_{1}\left(  \zeta\right)
/f_{0}\left(  \zeta\right)  \in D_{L}\left(  \phi\left(  \zeta\right)
\right)  $ (Lemma \ref{a-1}) for any $L\geq1$, and $f_{n}\in\ell^{2}\left(
\mathbb{Z}_{+}\right)  $. Let $m_{+}$ be any Weyl function of $H_{q\left(
g\boldsymbol{a}\right)  }$ on $\mathbb{Z}_{+}$. Then, due to $a_{1}%
m_{+}\left(  \phi\left(  \zeta\right)  \right)  \in D_{L}\left(  \phi\left(
\zeta\right)  \right)  $ Lemma \ref{a-1} implies%
\begin{equation}
\left\vert \dfrac{f_{1}\left(  \zeta\right)  }{f_{0}\left(  \zeta\right)
}+a_{1}m_{+}\left(  \phi\left(  \zeta\right)  \right)  \right\vert \leq
\dfrac{a_{1}}{\sqrt{a_{L+1}}\left\vert \operatorname{Im}s_{L}\left(
\phi\left(  \zeta\right)  \right)  \overline{s_{L+1}\left(  \phi\left(
\zeta\right)  \right)  }\right\vert }=O\left(  \zeta^{-2L-1}\right)
\label{3-13}%
\end{equation}
as $\zeta\rightarrow\infty$ since $s_{L}$ is a polynomial of degree $L$.
Substituting (\ref{3-13}) to (\ref{3-11}) yields%
\[
m_{+}\left(  z+z^{-1}\right)  =\sum_{0\leq k\leq L-1}\mu_{k}z^{-k-1}+O\left(
z^{-L-1}\right)  \text{ as }z\rightarrow\infty\text{ on }i\mathbb{R}%
\]
with $\mu_{k}$ satisfying
\[
\left\vert \mu_{k}\right\vert \leq c_{3}c^{-k}\Gamma\left(  k/N+1\right)
\text{ \ for any }k\geq1\text{,}%
\]
where $c_{3}$ is another constant. Lemma \ref{a-3} shows%
\[
m_{+}\left(  z\right)  =\sum_{0\leq k\leq L-1}\sigma_{k}z^{-k-1}+O\left(
z^{-L-1}\right)  \text{ as }z\rightarrow\infty\text{ on }i\mathbb{R}%
\]
with other coefficients $\sigma_{k}$ satisfying%
\[
\left\vert \sigma_{k}\right\vert \leq c_{3}2^{k}\left(  kc^{-1}\Gamma\left(
\alpha+1\right)  +kc^{-k}\Gamma\left(  \alpha k+1\right)  \right)  \text{.}%
\]
Then, Lemma \ref{a-4} implies%
\[
\int_{-\infty}^{\infty}e^{c\left\vert \lambda\right\vert ^{N}}\sigma
_{+}\left(  d\lambda\right)  <\infty\text{ for any }c>1\text{.}%
\]
Hence the boundary $+\infty$ is of limit point type, which means
$-f_{1}\left(  \zeta\right)  /f_{0}\left(  \zeta\right)  =a_{1}m_{+}\left(
\phi\left(  \zeta\right)  \right)  $ if $\zeta\in D_{-}^{1}$. Then,
(\ref{3-10}) yields the first identity of (\ref{3-12}). The second identity
can be deduced similarly, which completes the proof.\bigskip
\end{proof}

To identify $Q_{N}$ with $\boldsymbol{A}_{N,++}^{inv}\left(  C\right)  $ it is
required to define a symbol $\boldsymbol{m}\in$ $\boldsymbol{A}_{N,++}%
^{inv}\left(  C\right)  $ from $q\in Q_{N}$ or equivalently the Weyl functions
$m_{\pm}$ and coefficients $a_{0}$, $a_{1}$, $b_{0}$. In view of (\ref{3-12})
set%
\begin{align}
m(z)  &  =\left\{
\begin{array}
[c]{ll}%
z+z^{-1}+a_{1}^{2}m_{+}\left(  z+z^{-1}\right)  & \text{if }\left\vert
z\right\vert >1\\
-a_{0}^{2}m_{-}\left(  z+z^{-1}\right)  +b_{0} & \text{if }\left\vert
z\right\vert <1
\end{array}
\right.  \text{ and}\nonumber\\
\boldsymbol{m}\left(  z\right)   &  =\left(  m_{1}(z),m_{2}(z)\right)
=\left(  \dfrac{zm(z)-1}{z^{2}-1},z^{2}\dfrac{z-m(z)}{z^{2}-1}\right)
\text{.} \label{3-14}%
\end{align}
This symbol $\boldsymbol{m}$ satisfies (i) and (iii) of (\ref{2-4}) due to%
\[
m_{1}(z)-1=\dfrac{z\left(  m(z)-z\right)  }{z^{2}-1}\text{.}%
\]
(iv) is verified from the identity
\[
M(z)\equiv m_{1}(z)\widetilde{m}_{1}(z)-m_{2}(z)\widetilde{m}_{2}%
(z)=\dfrac{m(z)-m(z^{-1})}{z-z^{-1}}\text{ }%
\]
and the properties $\operatorname{Im}m\left(  z\right)  /\operatorname{Im}z>0$
if $\left\vert z\right\vert >1$ and $\operatorname{Im}m\left(  z\right)
/\operatorname{Im}z<0$ if $\left\vert z\right\vert <1$.

To show (ii) set%
\[
m_{\pm}^{\left(  0\right)  }\left(  z\right)  =\int_{-\infty}^{\infty}%
\dfrac{1-\left(  \cosh z^{N}\right)  ^{-c}\left(  \cosh\lambda^{N}\right)
^{c}}{\lambda-z}\sigma_{\pm}\left(  d\lambda\right)  \text{.}%
\]
Then, $m_{\pm}^{\left(  0\right)  }\left(  \phi\left(  z\right)  \right)  $ is
analytic on $D_{+}$, and,\ if $\left\vert z\right\vert >1$, an identity
\[
\rho_{c}(z)\left(  m\left(  z\right)  -\left(  \phi\left(  z\right)
+a_{1}^{2}m_{+}^{\left(  0\right)  }\left(  \phi\left(  z\right)  \right)
\right)  \right)  =a_{1}^{2}\int_{-\infty}^{\infty}\dfrac{\left(  \cosh
\lambda^{N}\right)  ^{c}}{\lambda-\phi\left(  z\right)  }\sigma_{+}\left(
d\lambda\right)
\]
yields%
\[
\sup_{z\in C_{1}}\left\vert \rho_{c}(z)\left(  m\left(  z\right)  -\left(
\phi\left(  z\right)  +a_{1}^{2}m_{+}^{\left(  0\right)  }\left(  \phi\left(
z\right)  \right)  \right)  \right)  \right\vert <\infty\text{.}%
\]
Here we have used the fact that $\inf_{z\in C_{1}}\left\vert \operatorname{Im}%
\phi\left(  z\right)  \right\vert >0$ due to (\ref{73}). Similarly one has
\[
\sup_{z\in C_{2}}\left\vert \rho_{c}(z)\left(  m\left(  z\right)  -\left(
-a_{0}^{2}m_{-}^{\left(  0\right)  }\left(  \phi\left(  z\right)  \right)
+b_{0}\right)  \right)  \right\vert <\infty\text{.}%
\]
To obtain a compensator $f$ for $m$ on $C$ set%
\[
f(z)=\eta_{1}(z)\left(  \phi\left(  z\right)  +a_{1}^{2}m_{+}^{\left(
0\right)  }\left(  \phi\left(  z\right)  \right)  \right)  +\eta_{1}%
(z^{-1})\left(  -a_{0}^{2}m_{-}^{\left(  0\right)  }\left(  \phi\left(
z\right)  \right)  +b_{0}\right)
\]
with%
\[
\eta_{1}(z)=\dfrac{e^{z^{2N}}}{e^{z^{2N}}+e^{z^{-2N}}}\text{.}%
\]
Then, $f$ is analytic on $D_{+}$, and one has%
\[
\sup_{\lambda\in C}\left\vert \rho_{c}(\lambda)\left(  m\left(  \lambda
\right)  -f(\lambda)\right)  \right\vert <\infty\text{.}%
\]
$m_{1}$, $m_{2}$ defined by (\ref{3-14}) have singulalities at $z=\pm1$ and we
have to remove them. For this purpose define%
\begin{equation}
f_{1}(z)=\dfrac{zf(z)-1}{z^{2}-1}\eta_{2}(z)\text{, \ \ }f_{2}(z)=z^{2}%
\dfrac{z-f(z)}{z^{2}-1}\eta_{2}(z) \label{3-15}%
\end{equation}
with%
\[
\eta_{2}(z)=1-\dfrac{1}{2}\dfrac{\left(  \rho_{2c}(1)-\rho_{2c}(-1)\right)
z+\rho_{2c}(1)+\rho_{2c}(-1)}{\rho_{2c}(z)}\text{.}%
\]
Then, clearly $f_{1}$, $f_{2}$ are bounded and analytic on $D_{+}$, and the
function
\[
\rho_{c}(\lambda)\left(  m_{1}(\lambda)-f_{1}(\lambda)\right)  =\rho
_{c}(\lambda)\dfrac{\lambda\left(  m(\lambda)-f(\lambda)\eta_{2}%
(\lambda)\right)  +\eta_{2}(\lambda)-1}{\lambda^{2}-1}%
\]
satisfy%
\[
\sup_{\lambda\in C}\left\vert \rho_{c}(\lambda)\left(  m_{1}(\lambda
)-f_{1}(\lambda)\right)  \right\vert +\sup_{\lambda\in C_{1}}\left\vert
\partial_{\lambda}\left(  \rho_{c}(\lambda)\left(  m_{1}(\lambda
)-f_{1}(\lambda)\right)  \right)  \right\vert <\infty\text{.}%
\]
The rest of the relevant estimates can be obtained similarly, which proves
(ii). Consequently, we have $\boldsymbol{m}\in\boldsymbol{M}_{N}\left(
C\right)  $.

To have $\boldsymbol{m}^{-1}\in\boldsymbol{M}_{N}\left(  C\right)  $ we have
to find a compensator for $M\left(  z\right)  $. One candidate is
\[
F(z)\equiv f_{1}(z)\widetilde{f}_{1}(z)-f_{2}(z)\widetilde{f}_{2}%
(z)=\dfrac{f(z)-f(z^{-1})}{z-z^{-1}}\eta_{2}(z)\eta_{2}(z^{-1})\text{,}%
\]
and%
\begin{gather*}
\dfrac{f(z)-f(z^{-1})}{z-z^{-1}}=\dfrac{\eta_{1}(z)-\eta_{1}(z^{-1})}%
{z-z^{-1}}A(\phi\left(  z\right)  )\text{ with}\\
A(z)\equiv z+a_{1}^{2}m_{+}^{\left(  0\right)  }\left(  z\right)  +a_{0}%
^{2}m_{-}^{\left(  0\right)  }\left(  z\right)  -b_{0}\text{.}%
\end{gather*}
It is possible for $F$ to have zeros in $\overline{D}_{+}$. Such zeros are
generated by $A(z)$ on the strip $S=\left\{  \left\vert \operatorname{Im}%
z\right\vert \leq\vartheta\right\}  $, where $A(z)$ is analytic. Since
$A(z)=z+o\left(  1\right)  $ as $z\rightarrow\infty$ in $S$, zeros of $A(z)$
are confined on a compact set $K=\left\{  \left\vert \operatorname{Re}%
z\right\vert \leq a\right\}  \cap S$ for an $a>0$. Let them be $\left\{
\alpha_{j}\right\}  _{1\leq j\leq n}$ and set $p(z)=\prod_{1\leq j\leq
n}\left(  z-\alpha_{j}\right)  $. Define%
\[
F_{0}(z)=F(z)\left(  1-\dfrac{1}{\rho_{2c}\left(  z\right)  p(\phi\left(
z\right)  )}\right)  \text{.}%
\]
Then, without difficulty we can verify (\ref{2-5}) for the compensator $F_{0}%
$, which yields $\boldsymbol{m}^{-1}\in\boldsymbol{M}_{N}\left(  C\right)  $.
Consequently, we have $\boldsymbol{m}\in\boldsymbol{A}_{N}^{inv}\left(
C\right)  $.

\begin{lemma}
\label{l3-6}For $q\in Q_{N}$ define a symbol $\boldsymbol{m}$ by (\ref{3-14}).
Then, $\boldsymbol{m}\in\boldsymbol{A}_{N,++}^{inv}\left(  C\right)  $ is
valid and the $m$-function of $\boldsymbol{m}$ is $m$.
\end{lemma}

\begin{proof}
$T\left(  \boldsymbol{m}\right)  $ is invertible on $H_{N,c}\left(
D_{+}\right)  $. All we have to do is to compute its $m$-function. We have
from Lemma \ref{l2-13}%
\[
m_{\boldsymbol{m}}=z+\dfrac{\left(  z^{-2}-1\right)  m_{2}-m_{2}^{\prime
}(0)\left(  z^{-1}m_{1}+m_{2}\right)  }{m_{1}+m_{2}z^{-1}}=m
\]
due to $m_{2}^{\prime}(0)=0$. Lemma \ref{l3-3} implies $\boldsymbol{m}%
\in\boldsymbol{A}_{N,++}^{inv}\left(  C\right)  $.\bigskip
\end{proof}

\begin{remark}
\label{r2}There are many symbols $\boldsymbol{m}$ whose $m$-functions coincide
with the given $m$. For instance, the symbol below is another example:%
\[
\left(  m_{1}(z),m_{2}(z)\right)  =\left(  1,z^{2}\dfrac{m(z)-z}%
{1-zm(z)}\right)  \text{.}%
\]

\end{remark}

\section{Proof of Theorems}

We start the proof of Theorems by defining the flow on $Q_{N}$. For $q\in
Q_{N}$, $\boldsymbol{m}$ of (\ref{3-14}) and $g\in$ $\Gamma_{N}%
^{\operatorname{real}}$ set%
\begin{equation}
\mathrm{Toda}\left(  g\right)  q=\left\{  a_{n}\left(  g\boldsymbol{m}\right)
,b_{n}\left(  g\boldsymbol{m}\right)  \right\}  _{n\in\mathbb{Z}}\in
Q_{N}\text{,} \label{4-1}%
\end{equation}
where%
\[
\left\{
\begin{array}
[c]{l}%
a_{n}\left(  g\boldsymbol{m}\right)  =\sqrt{\dfrac{\tau_{g\boldsymbol{m}%
}\left(  z^{n}\right)  \tau_{g\boldsymbol{m}}\left(  z^{n-2}\right)  }%
{\tau_{g\boldsymbol{m}}\left(  z^{n-1}\right)  ^{2}}}\\
b_{n}\left(  g\boldsymbol{m}\right)  =\dfrac{\left.  \partial_{\varepsilon
}\tau_{g\boldsymbol{m}}\left(  z^{n}q_{\varepsilon^{-1}}\right)  \right\vert
_{\varepsilon=0}}{\tau_{g\boldsymbol{m}}\left(  z^{n}\right)  }-\dfrac{\left.
\partial_{\varepsilon}\tau_{g\boldsymbol{m}}\left(  z^{n-1}q_{\varepsilon
^{-1}}\right)  \right\vert _{\varepsilon=0}}{\tau_{g\boldsymbol{m}}\left(
z^{n-1}\right)  }%
\end{array}
\right.  \text{.}%
\]
There remain two problems to be solved. The first one is to show the flow
property $\mathrm{Toda}\left(  g_{1}g_{2}\right)  =\mathrm{Toda}\left(
g_{1}\right)  \mathrm{Toda}\left(  g_{2}\right)  $. The second is to show the
continuity by imposing suitable metrics on $Q_{N}$, $\Gamma_{N}%
^{\operatorname{real}}$.

\begin{lemma}
$\left\{  \mathrm{Toda}\left(  g\right)  \right\}  _{g\in\Gamma_{N}%
^{\operatorname{real}}}$ defines a flow on $Q_{N}$.
\end{lemma}

\begin{proof}
We show the flow property:%
\begin{equation}
\mathrm{Toda}\left(  g_{2}g_{1}\right)  =\mathrm{Toda}\left(  g_{2}\right)
\mathrm{Toda}\left(  g_{1}\right)  \text{ \ for }g_{1},g_{2}\in\Gamma
_{N}^{\operatorname{real}}\text{.} \label{4-2}%
\end{equation}
\ For $q\in Q_{N}$ let $\boldsymbol{m}\in\boldsymbol{A}_{N,++}^{inv}\left(
C\right)  $ be the symbol defined in (\ref{3-14}) and $\boldsymbol{m}_{1}%
\in\boldsymbol{A}_{N,++}^{inv}\left(  C\right)  $ be the other symbol derived
from the $m$-function $m_{g_{1}\boldsymbol{m}}$. Then, (iii) of Lemma
\ref{l2-12} implies $m_{g_{2}\boldsymbol{m}_{1}}=m_{g_{2}g_{1}\boldsymbol{m}}%
$. Observing the correspondences
\[
q=\left\{  a_{k},b_{k}\right\}  _{k\in\mathbb{Z}}\in Q_{N}\Longleftrightarrow
\left\{  m_{\pm},a_{0},a_{1}\right\}  \Longleftrightarrow m\text{
(}m\text{-function)}%
\]
are one-to-one, one has (\ref{4-2}) since the $m$-functions for $\mathrm{Toda}%
\left(  g_{2}\right)  \mathrm{Toda}\left(  g_{1}\right)  q$ and $\mathrm{Toda}%
\left(  g_{2}g_{1}\right)  q$ are $m_{g_{2}\boldsymbol{m}_{1}}$ and
$m_{g_{2}g_{1}\boldsymbol{m}}$ respectively.\bigskip
\end{proof}

The metric on $Q_{N}$ is introduced by
\[
\mathrm{d}\left(  q_{1},q_{2}\right)  =\sum_{k\in\mathbb{Z}}2^{-\left\vert
k\right\vert }\left(
\begin{array}
[c]{c}%
\left\vert a_{k}\left(  q_{1}\right)  -a_{k}\left(  q_{2}\right)  \right\vert
+\left\vert b_{k}\left(  q_{1}\right)  -b_{k}\left(  q_{2}\right)  \right\vert
\\
+\left\vert \left(  \left(  \cosh\left(  kH_{q_{1}}^{N}\right)  -\cosh\left(
kH_{q_{2}}^{N}\right)  \right)  \delta_{0},\delta_{0}\right)  \right\vert
\end{array}
\right)  \wedge1\text{.}%
\]

\begin{lemma}
\label{l3-7}Let $q$, $q_{n}\in$ $Q_{N}$.\newline(i) If the coefficients
$a_{k}\left(  q_{n}\right)  $, $b_{k}\left(  q_{n}\right)  $ converge to
$a_{k}\left(  q\right)  $, $b_{k}\left(  q\right)  $ for any $k\in\mathbb{Z}$
respectively, the $m$-functions $m_{q_{n}}\left(  z\right)  $ defined by
(\ref{3-14}) converges to $m_{q}\left(  z\right)  $ for any $z\in D_{-}$.
Conversely, if $m_{q_{n}}\left(  z\right)  \rightarrow m_{q}\left(  z\right)
$ and $\sup_{n\geq1}\left(  \cosh\left(  cH_{q_{n}}\right)  \delta_{0}%
,\delta_{0}\right)  <\infty$ \ for some $c>0$, the coefficients $a_{k}\left(
q_{n}\right)  $, $b_{k}\left(  q_{n}\right)  $ converge to $a_{k}\left(
q\right)  $, $b_{k}\left(  q\right)  $ for any $k\in\mathbb{Z}$.\newline(ii)
$\mathrm{d}\left(  q_{n},q\right)  \rightarrow0$ holds if and only if
$m_{q_{n}}\left(  z\right)  \rightarrow m_{q}\left(  z\right)  $ for any $z\in
D_{-}$, and
\begin{equation}
\text{ }\sup_{n\geq1}\left(  \cosh\left(  c\left(  H_{q_{n}}\right)
^{N}\right)  \delta_{0},\delta_{0}\right)  <\infty\text{ \ for any
}c>0\text{.} \label{4-3}%
\end{equation}
(iii) $\mathrm{d}\left(  q_{n},q\right)  \rightarrow0$ implies that
$\left\Vert \boldsymbol{m}_{q_{n}}-\boldsymbol{m}_{q}\right\Vert
_{c,j}\rightarrow0$ for $j=1$, $2$ and for any $c>0$.
\end{lemma}

\begin{proof}
(ii) of Lemma \ref{a-2} and Lemma \ref{a-4} imply that (i) is equivalent to%
\[
\sup_{n\geq1}\int_{-\infty}^{\infty}e^{c\left\vert \lambda\right\vert ^{N}%
}\sigma_{\pm}^{q_{n}}\left(  d\lambda\right)  <\infty\text{ \ for any
}c>0\text{.}%
\]
The $k$-th moment of $\sigma_{+}^{q_{n}}\left(  d\lambda\right)  $, which is
$\left(  \left(  H_{q_{n}}^{+}\right)  ^{k}\delta_{0},\delta_{0}\right)  $,
can be described by a polynomial of $\left\{  a_{j}\left(  q_{n}\right)
,b_{j}\left(  q_{n}\right)  \right\}  _{0\leq j\leq k}$, hence (ii) implies
that%
\[
\left(  \left(  H_{q_{n}}^{+}\right)  ^{k}\delta_{0},\delta_{0}\right)
\rightarrow\left(  \left(  H_{q}^{+}\right)  ^{k}\delta_{0},\delta_{0}\right)
\text{ as }n\rightarrow\infty\text{ for any }k\geq0
\]
Then, it is clear that $m_{+}^{q_{n}}\left(  z\right)  \rightarrow m_{+}%
^{q}\left(  z\right)  $ for any $z\in\mathbb{C}\backslash\mathbb{R}$, since
the moment problem is unique for $\sigma_{+}^{q}\left(  d\lambda\right)  $.
The same conclusion is valid also for $m_{-}^{q_{n}}\left(  z\right)  $,
$m_{-}^{q}\left(  z\right)  $, hence $m_{q_{n}}\left(  z\right)  \rightarrow
m_{q}\left(  z\right)  $ holds. Conversely, suppose $m_{q_{n}}\left(
z\right)  \rightarrow m_{q}\left(  z\right)  $. Then, $m_{+}^{q_{n}}\left(
z\right)  \rightarrow m_{+}^{q}\left(  z\right)  $ and $a_{k}\left(
q_{n}\right)  \rightarrow a_{k}\left(  q\right)  $ for $k=0$, $1$ and
$b_{0}\left(  q_{n}\right)  \rightarrow b_{0}\left(  q\right)  $ hold. Suppose
$\sup_{n\geq1}\left(  \cosh\left(  cH_{q_{n}}\right)  \delta_{0},\delta
_{0}\right)  <\infty$ additionally, which means $\sup_{n\geq1}\left(
\cosh\left(  cH_{q_{n}}^{\pm}\right)  \delta_{0},\delta_{0}\right)  <\infty$.
Clearly, we have the convergence of the every moment of $\sigma_{\pm}^{q_{n}%
}\left(  d\lambda\right)  $, and so are the coefficients $a_{k}\left(
q_{n}\right)  $, $b_{k}\left(  q_{n}\right)  $.

As for (ii) of the lemma, only the converse remains to be proved. The
convergence of $a_{k}\left(  q_{n}\right)  $, $b_{k}\left(  q_{n}\right)  $ is
clear from (i). The convergence of $\left(  \cosh\left(  kH_{q_{n}}%
^{N}\right)  \delta_{0},\delta_{0}\right)  $ for fixed $k\in\mathbb{Z}$
follows from the tightness of the measures $\cosh\left(  k\lambda^{N}\right)
\sigma^{q_{n}}\left(  d\lambda\right)  $ ($\sigma^{q_{n}}\left(
d\lambda\right)  $ is the spectral measure of $H_{q_{n}}$ evaluated at
$\delta_{0}$) due to (\ref{4-3}), which shows (ii).

We show $\left\Vert \boldsymbol{m}_{q_{n}}-\boldsymbol{m}_{q}\right\Vert
_{c,1}\rightarrow0$. The precise computation is subtle, although it is
possible, hence we pretend as if the compensators $f_{1}$, $f_{2}$ of
(\ref{3-15}) were defined without the modifiers $\eta_{1}$, $\eta_{2}$. We
regard $m_{j}\left(  z\right)  -f_{j}(z)$ as%
\begin{align*}
m_{1}(z)-f_{1}(z)  &  =\left\{
\begin{array}
[c]{cc}%
\dfrac{a_{1}^{2}z\left(  m_{+}\left(  \phi\left(  z\right)  \right)
-m_{+}^{\left(  0\right)  }\left(  \phi\left(  z\right)  \right)  \right)
}{z^{2}-1} & \text{if \ }\left\vert z\right\vert >1\\
-\dfrac{a_{0}^{2}z(m_{-}\left(  \phi\left(  z\right)  \right)  -m_{-}^{\left(
0\right)  }\left(  \phi\left(  z\right)  \right)  )}{z^{2}-1} & \text{if
\ }\left\vert z\right\vert <1
\end{array}
\right.  \text{,}\\
m_{2}(z)-f_{2}(z)  &  =-z\left(  m_{1}(z)-f_{1}(z)\right)  \text{,}%
\end{align*}
where%
\[
m_{\pm}^{\left(  0\right)  }\left(  z\right)  =\int_{-\infty}^{\infty}%
\dfrac{1-\left(  \cosh z^{N}\right)  ^{-c}\left(  \cosh\lambda^{N}\right)
^{c}}{\lambda-z}\sigma_{\pm}\left(  d\lambda\right)  \text{.}%
\]
Hence, setting $\zeta=\phi\left(  z\right)  $, on $C_{1}$%
\begin{align}
&  \left(  m_{q,1}(z)-f_{q,1}(z)\right)  -\left(  m_{q_{n},1}(z)-f_{q_{n}%
,1}(z)\right) \nonumber\\
&  =\dfrac{z}{z^{2}-1}\left(  a_{1}\left(  q\right)  ^{2}\left(  m_{+}%
^{q}\left(  \zeta\right)  -m_{+}^{\left(  0\right)  ,q}\left(  \zeta\right)
\right)  -a_{1}\left(  q_{n}\right)  ^{2}\left(  m_{+}^{q_{n}}\left(
\zeta\right)  -m_{+}^{\left(  0\right)  ,q_{n}}\left(  \zeta\right)  \right)
\right) \nonumber\\
&  =\dfrac{z\rho_{c}(z)^{-1}}{z^{2}-1}\left(  a_{1}\left(  q\right)  ^{2}%
\int_{-\infty}^{\infty}\dfrac{\cosh c\lambda^{N}}{\lambda-\zeta}\sigma_{+}%
^{q}\left(  d\lambda\right)  -a_{1}\left(  q_{n}\right)  ^{2}\int_{-\infty
}^{\infty}\dfrac{\cosh c\lambda^{N}}{\lambda-\zeta}\sigma_{+}^{q_{n}}\left(
d\lambda\right)  \right)  \label{4-4}%
\end{align}
hold. We know $a_{1}\left(  q_{n}\right)  \rightarrow a_{1}\left(  q\right)  $
and $m_{+}^{q_{n}}\left(  \zeta\right)  \rightarrow m_{+}^{q}\left(
\zeta\right)  $. The convergence of $m_{+}^{q_{n}}$ implies the weak
convergence of the probability measures $\sigma_{+}^{q_{n}}$ on $\mathbb{R}$.
Choosing $k>c$, we know the measures $\cosh c\lambda^{N}\sigma_{+}^{q_{n}}$
also converge weakly to $\cosh c\lambda^{N}\sigma_{+}^{q}$. Thus, it is clear
that
\[
\sup_{z\in C_{1}}\left\vert \rho_{c}(z)\left(  \left(  m_{q,1}(z)-f_{q,1}%
(z)\right)  -\left(  m_{q_{n},1}(z)-f_{q_{n},1}(z)\right)  \right)
\right\vert \rightarrow0\text{ \ as }n\rightarrow\infty\text{ ,}%
\]
since $\left\vert \operatorname{Im}\zeta\right\vert =\vartheta$ on $C_{1}$.
Without difficulty one also sees%
\[
\sup_{z\in C_{1}}\left\vert f_{q,1}(z)-f_{q_{n},1}(z)\right\vert
\rightarrow0\text{ \ as }n\rightarrow\infty\text{.}%
\]
We can make a similar argument on $C_{2}$, and also for the $m_{2}$-parts,
hence we have%
\[
\left\Vert \boldsymbol{m}_{q_{n}}-\boldsymbol{m}_{q}\right\Vert _{c,1}%
\rightarrow0\text{ \ as \ }n\rightarrow\infty\text{.}%
\]
The convergence $\left\Vert \boldsymbol{m}_{q_{n}}-\boldsymbol{m}%
_{q}\right\Vert _{c,2}\rightarrow0$ follows from $\left\Vert \boldsymbol{m}%
_{q_{n}}-\boldsymbol{m}_{q}\right\Vert _{c,3}\rightarrow0$, which can be shown
similarly as $\left\Vert \boldsymbol{m}_{q_{n}}-\boldsymbol{m}_{q}\right\Vert
_{c,1}\rightarrow0$ by taking the derivative of $\boldsymbol{m}_{q_{n}%
}-\boldsymbol{m}_{q}$ in (\ref{4-4}).\bigskip
\end{proof}

\begin{proposition}
\label{p2}$\left\{  \mathrm{Toda}\left(  g\right)  \right\}  _{g\in\Gamma
_{N}^{\operatorname{real}}}$ defines a flow on $Q_{N}$. This flow is
continuous in the following sense. \newline(i) If $\mathrm{d}_{c}\left(
g_{n},g\right)  \rightarrow0$ holds keeping $c>\delta_{D_{+}}\left(  g\right)
$, $\sup_{n}\delta_{D_{+}}\left(  g_{n}\right)  $ for some $c>0$, then
$\mathrm{d}\left(  \mathrm{Toda}\left(  g_{n}\right)  q,\mathrm{Toda}\left(
g\right)  q\right)  \rightarrow0$.\newline(ii) $\mathrm{d}\left(
q_{n},q\right)  \rightarrow0$ implies $\mathrm{d}\left(  \mathrm{Toda}\left(
g\right)  q_{n},\mathrm{Toda}\left(  g\right)  q\right)  \rightarrow0$.
\end{proposition}

\begin{proof}
Suppose $\mathrm{d}_{c}\left(  g_{n},g\right)  \rightarrow0$. Then, the
continuity of tau-functions implies%
\[
a_{k}\left(  g_{n}\boldsymbol{m}\right)  \rightarrow a_{k}\left(
g\boldsymbol{m}\right)  \text{, }b_{k}\left(  g_{n}\boldsymbol{m}\right)
\rightarrow b_{k}\left(  g\boldsymbol{m}\right)  \text{, }m_{g_{n}%
\boldsymbol{m}}\left(  z\right)  \rightarrow m_{g\boldsymbol{m}}\left(
z\right)  \text{\ as }n\rightarrow\infty
\]
for any $k\in Z$ and $z\in D_{-}$. On the other hand, in the expression%
\begin{equation}
\varphi_{z^{j}g_{n}\boldsymbol{m}}^{\left(  0\right)  }\left(  z\right)
=\int_{C}\dfrac{\lambda^{j}g_{n}\left(  \lambda\right)  \left(  \boldsymbol{m}%
-\boldsymbol{f}\right)  \left(  \lambda\right)  \left(  \left(  z^{-j}%
g_{n}^{-1}T\left(  z^{j}g_{n}\boldsymbol{m}\right)  \right)  ^{-1}z^{-j}%
g_{n}^{-1}1\right)  \left(  \lambda\right)  }{2\pi i\left(  z-\lambda\right)
}d\lambda\label{4-5}%
\end{equation}
first observe that $\varphi_{z^{j}g_{n}\boldsymbol{m}}^{\left(  0\right)  }$
remains the same if $g_{n}$ is replaced by $cg_{n}$ for any non-zero constant
$c$, which means one can normalize $g_{n}$ by $g_{n}\left(  1\right)  =1$.
Then, we have
\begin{equation}
\left\vert g_{n}\left(  \lambda\right)  ^{\pm1}\right\vert \leq\exp\left(
\delta_{D_{+}}\left(  g_{n}\right)  \int_{1}^{\lambda}\theta\left(
\xi\right)  \left\vert d\xi\right\vert \right)  \leq c_{1}\exp\left(
c^{\prime}\left\vert \lambda\right\vert ^{N}+\left\vert \lambda\right\vert
^{-N}\right)  \label{4-6}%
\end{equation}
on $D_{+}$ for any $c^{\prime}>\delta_{D_{+}}\left(  g_{n}\right)  $ and a
constant $c_{1}$ depending on $c^{\prime}$. For $c^{\prime}>c$, one sees
$\mathrm{d}_{c^{\prime}}\left(  z^{j}g_{n},z^{j}g\right)  \rightarrow0$ due to
$\mathrm{d}_{c^{\prime}}\left(  z^{j}g_{1},z^{j}g_{2}\right)  \leq
c_{2}\mathrm{d}_{c}\left(  g_{1},g_{2}\right)  $ for some constant $c_{2}$,
which yields $z^{-j}g_{n}^{-1}T\left(  z^{j}g_{n}\boldsymbol{m}\right)  $
converges to $z^{-j}g^{-1}T\left(  z^{j}g\boldsymbol{m}\right)  $ in the HS
norm on $H_{N,c^{\prime}}\left(  D_{+}\right)  $ due to (\ref{2-8}),
(\ref{2-12}). This together with (\ref{4-6}) implies that the numerator of the
integrant of (\ref{4-5}) has a bound by $\rho_{c^{\prime\prime}}\left(
\lambda\right)  ^{-1}$ for any $c^{\prime\prime}>c^{\prime}$ uniformly in $n$,
since $\left\Vert \rho_{c^{\prime\prime}}\left(  \boldsymbol{m}-\boldsymbol{f}%
\right)  \left(  \lambda\right)  \right\Vert =O\left(  1\right)  $ as
$\lambda\rightarrow\infty$ or $\lambda\rightarrow0$ on $C$ by choosing the
compensator $\boldsymbol{f}$ suitably. Then, the coefficients $\mu_{k}$,
$\nu_{k}$ of (\ref{3-11}) are dominated by $c_{3}c^{-k}\Gamma\left(
k/N+1\right)  $ for any $k\in\mathbb{Z}_{+}$ independently in $n$. Therefore,
Lemmas \ref{a-2}, \ref{a-4} yield%
\begin{equation}
\sup_{n\geq1}\int_{-\infty}^{\infty}e^{c\left\vert \lambda\right\vert ^{N}%
}a_{\left(  1\pm1\right)  /2}\left(  q_{n}\right)  ^{2}\sigma_{\pm}^{q_{n}%
}\left(  d\lambda\right)  <\infty\text{.} \label{4-7}%
\end{equation}
This together with the convergence $m_{g_{n}\boldsymbol{m}}\left(  z\right)  $
shows (i) due to Lemma (ii) of Lemma \ref{l3-7}.

(iii) of Lemma \ref{l3-7} implies $\left\Vert \boldsymbol{m}_{{}%
}-\boldsymbol{m}_{q}\right\Vert _{c,2}\rightarrow0$ (refer (\ref{2-11}) for
$\left\Vert \cdot\right\Vert _{c,2}$) if $\mathrm{d}\left(  q_{n},q\right)
\rightarrow0$ holds. Then, the continuity of $\tau_{g\boldsymbol{m}}\left(
g_{1}\right)  $ implies the convergences $a_{k}\left(  \mathrm{Toda}\left(
g\right)  q_{n}\right)  \rightarrow a_{k}\left(  \mathrm{Toda}\left(
g\right)  q\right)  $, $b_{k}\left(  \mathrm{Toda}\left(  g\right)
q_{n}\right)  \rightarrow b_{k}\left(  \mathrm{Toda}\left(  g\right)
q\right)  $, and the uniform bound of (\ref{4-7}) replaced $q_{n}$ by
$\mathrm{Toda}\left(  g\right)  q_{n}$. The rest of the proof of (ii) is
similar to that of (i).\bigskip
\end{proof}

\subsection{Proof of Theorem \ref{t1}}

For Theorem \ref{t1} all we have to show is that $\mathrm{Toda}\left(
e^{-2t\widehat{p}}\right)  q$ solves the Cauchy problem of the Toda hierarchy:%
\[
\partial_{t}H_{q_{t}}=\left[  H_{q_{t}},p\left(  H_{q_{t}}\right)
_{a}\right]  \text{ \ with }q_{0}=q\text{.}%
\]
\ This was already proved if $q$ is bounded, and we employ this fact and the
approximation of unbounded $q$ by bounded ones to show the theorem. For
$q=\left\{  a_{k},b_{k}\right\}  _{k\in\mathbb{Z}}\in Q_{N}$ let $\sigma_{\pm
}$ be the spectral measures and set%
\[
\sigma_{\pm,n}\left(  d\lambda\right)  =\dfrac{I_{\left[  -n,n\right]
}\left(  \lambda\right)  \sigma_{\pm}\left(  d\lambda\right)  +n^{-1}%
\nu\left(  d\lambda\right)  }{\sigma_{\pm}\left(  \left[  -n,n\right]
\right)  +n^{-1}}%
\]
with a probability measure $\nu$ on $\left[  -1,1\right]  $ having non trivial
absolutely continuous part. $\nu$ is added so that the Herglotz functions
associated with $\sigma_{\pm,n}$ are not rational functions. Define%
\[
m_{\pm,n}\left(  z\right)  =\int_{-\infty}^{\infty}\dfrac{\sigma_{\pm
,n}\left(  d\lambda\right)  }{\lambda-z}%
\]
and%
\[
\left\{
\begin{array}
[c]{l}%
m_{n}\left(  z\right)  =\left\{
\begin{array}
[c]{c}%
\phi\left(  z\right)  +a_{1}^{2}m_{+,n}\left(  \phi\left(  z\right)  \right)
\text{ \ \ if \ }\left\vert z\right\vert >1\\
-a_{0}^{2}m_{-,n}\left(  \phi\left(  z\right)  \right)  +b_{0}\text{ \ \ \ if
\ }\left\vert z\right\vert <1
\end{array}
\right. \\
\boldsymbol{m}_{n}\left(  z\right)  =\left(  \dfrac{zm_{n}\left(  z\right)
-1}{z^{2}-1},z^{2}\dfrac{z-m_{n}\left(  z\right)  }{z^{2}-1}\right)
\end{array}
\right.  \text{.}%
\]
If $q_{n}$ denotes the coefficient associated with $m_{n}\left(  z\right)  $,
then $q_{n}$ is bounded for each fixed $n\geq1$ and its Weyl functions are
$m_{\pm,n}$.

\begin{lemma}
\label{l4-1}$\mathrm{d}\left(  q_{n},q\right)  \rightarrow0$ as $n\rightarrow
\infty$ holds.
\end{lemma}

\begin{proof}
(ii) of Lemma \ref{l3-7} implies immediately the convergence $\mathrm{d}%
\left(  q_{n},q\right)  \rightarrow0$ as $n\rightarrow\infty$.\bigskip
\end{proof}

The flow property of $\mathrm{Toda}\left(  g\right)  q$ and its continuity has
been proved in Proposition \ref{p2}. The Lax equation satisfied by
$\mathrm{Toda}\left(  g\right)  q$ is proved based on Theorem 1 of \cite{z}.
First we identify the $\mathrm{Toda}\left(  g\right)  q$ defined in \cite{z}
by a bounded curve $C_{b}$ with the $\mathrm{Toda}\left(  g\right)  q$ defined
in (\ref{4-1}) by a unbounded curve $C$. If we could show%
\begin{equation}
\tau_{\boldsymbol{m}}^{C}\left(  r\right)  =\tau_{\boldsymbol{m}}^{C_{b}%
}\left(  r\right)  \text{ \ for any rational function }r\in\Gamma_{N}\left(
D_{+}\right)  \text{,} \label{4-8}%
\end{equation}
then the identity $\tau_{g\boldsymbol{m}}^{C}\left(  r\right)  =\tau
_{g\boldsymbol{m}}^{C_{b}}\left(  r\right)  $ obeys from the cocycle property
and the continuity with respect to $g$, and the two$\mathrm{Toda}\left(
g\right)  q$ can be identified, since they are defined by $\tau
_{g\boldsymbol{m}}^{C}\left(  r\right)  $ and $\tau_{g\boldsymbol{m}}^{C_{b}%
}\left(  r\right)  $ respectively. Indeed (\ref{4-8}) follows from the fact
that $\tau_{\boldsymbol{m}}^{C}\left(  r\right)  $ can be described only by
the $m$-function, which is deduced from Lemma \ref{l2-9} and the identity
$1+\varphi_{\boldsymbol{m}}^{\left(  0\right)  }\left(  z\right)  =1$ by Lemma
\ref{l2-11}. The equation%
\begin{equation}
\partial_{t}H_{q_{t}}=\left[  H_{q_{t}},p\left(  H_{q_{t}}\right)
_{a}\right]  \text{, }q_{0}=q\Longleftrightarrow H_{q_{t}}=H_{q}+\int_{0}%
^{t}\left[  H_{q_{s}},p\left(  H_{q_{s}}\right)  _{a}\right]  ds \label{4-9}%
\end{equation}
in Theorem \ref{t1} is valid for bounded initial data $q$. For an unbounded
$q\in Q_{N}$ we approximate it by bounded $q_{n}$ in Lemma \ref{l4-1}. Then,
the continuity of $\tau_{g\boldsymbol{m}}\left(  r\right)  $ with respect to
$\boldsymbol{m}$ given in Lemma \ref{l2-6} implies (\ref{4-9}), which
completes the proof of the theorem.

\subsection{Proof of Theorem \ref{t2} and Corollary \ref{c1}}

Before proceeding to the proof of the theorem we need

\begin{lemma}
\label{l4-2}Assume there exist a constant $c_{1}>0$, $\alpha>0$ such that
coefficients $\left\{  a_{n}\text{, }b_{n}\right\}  _{n\geq1}$ satisfy%
\begin{equation}
a_{n},\text{ }\left\vert b_{n}\right\vert \leq c_{1}n^{1/\alpha}%
\ \ \ \text{for any }n\geq1\text{.}\ \label{4-10}%
\end{equation}
Then, the spectral measure $\sigma_{+}$ of associated Jacobi operator $H_{+}$
satisfies%
\begin{equation}
\int_{-\infty}^{\infty}e^{c\left\vert \lambda\right\vert ^{\alpha}}\sigma
_{+}\left(  d\lambda\right)  <\infty\ \ \ \ \text{for any }c<\left(
9c_{1}^{2}\right)  ^{-\alpha/2}\left(  \dfrac{2}{\alpha}\right)  \text{.}
\label{4-11}%
\end{equation}

\end{lemma}

\begin{proof}
Recall%
\[
\left(  H_{+}u\right)  _{n}=\left\{
\begin{array}
[c]{ll}%
a_{1}u_{2}+b_{1}u_{1} & \text{if \ }n=1\\
a_{n}u_{n+1}+a_{n-1}u_{n-1}+b_{n}u_{n} & \text{if \ }n\geq2
\end{array}
\right.  \text{.}%
\]
Set%
\[
\xi_{n}=a_{n}^{2}+a_{n-1}^{2}+b_{n}^{2}\text{, \ \ }x_{k}=\left(  H_{+}%
^{2k}\delta_{1},\delta_{1}\right)  =\left\Vert H_{+}^{k}\delta_{1}\right\Vert
^{2}\text{\ .}%
\]
First we estimate $x_{k}$ by induction. We have%
\[
x_{k+1}=\left\Vert H_{+}^{k+1}\delta_{1}\right\Vert ^{2}=\sum_{n\geq1}\left(
a_{n}\left(  H_{+}^{k}\delta_{1}\right)  _{n+1}+a_{n-1}\left(  H_{+}^{k}%
\delta_{1}\right)  _{n-1}+b_{n}\left(  H_{+}^{k}\delta_{1}\right)
_{n}\right)  ^{2}\text{,}%
\]
where we understand $a_{0}=0$. Schwarz inequality yields%
\begin{align*}
x_{k+1}  &  \leq\sum_{n\geq1}\xi_{n}\left(  \left(  H_{+}^{k}\delta
_{1}\right)  _{n+1}^{2}+\left(  H_{+}^{k}\delta_{1}\right)  _{n-1}^{2}+\left(
H_{+}^{k}\delta_{1}\right)  _{n}^{2}\right) \\
&  =\sum_{n\geq1}\left(  \xi_{n-1}+\xi_{n+1}+\xi_{n}\right)  \left(  H_{+}%
^{k}\delta_{1}\right)  _{n}^{2}\text{,}%
\end{align*}
regarding $\xi_{0}=0$. Here noting%
\[
\left(  H_{+}^{k}\delta_{1}\right)  _{n}=0\text{ \ if \ }n\geq k+2
\]
one has%
\[
x_{k+1}\leq\sum_{1\leq n\leq k+1}\left(  \xi_{n-1}+\xi_{n+1}+\xi_{n}\right)
\left(  H_{+}^{k}\delta_{1}\right)  _{n}^{2}\leq\left(  \max_{1\leq n\leq
k+1}\left(  \xi_{n-1}+\xi_{n+1}+\xi_{n}\right)  \right)  x_{k}%
\]
holds. Since the assumption yields%
\[
\max_{1\leq n\leq k+1}\left(  \xi_{n-1}+\xi_{n+1}+\xi_{n}\right)  \leq
3\max_{1\leq n\leq k+2}\xi_{n}\leq9c_{1}^{2}\left(  k+2\right)  ^{2/\alpha
}\text{,}%
\]
we have%
\[
x_{k}=x_{1}%
%TCIMACRO{\dprod _{1\leq j\leq k-1}}%
%BeginExpansion
{\displaystyle\prod_{1\leq j\leq k-1}}
%EndExpansion
\dfrac{x_{j+1}}{x_{j}}\leq x_{1}%
%TCIMACRO{\dprod _{1\leq j\leq k-1}}%
%BeginExpansion
{\displaystyle\prod_{1\leq j\leq k-1}}
%EndExpansion
9c_{1}^{2}\left(  j+2\right)  ^{2/\alpha}=x_{1}\left(  9c_{1}^{2}\right)
^{k-1}\left(  \dfrac{\left(  k+1\right)  !}{2}\right)  ^{2/\alpha}\text{.}%
\]
Then, Stirling's formula $\Gamma\left(  x+1\right)  \sim\sqrt{2\pi x}\left(
x/e\right)  ^{x}$ shows that there exists a constant $c_{2}$ such that%
\[
x_{k}\leq c_{2}k^{3/\alpha-1/2}\left(  9c_{1}^{2}\left(  \dfrac{\alpha}%
{2}\right)  ^{2/\alpha}\right)  ^{k}\Gamma\left(  2k/\alpha+1\right)  \text{
\ for }k\geq1\text{.}%
\]
Then, Lemma \ref{a-4} completes the proof. \bigskip\ 
\end{proof}

Turning to the proof of Theorem \ref{t2} we assume $\left\{  a_{n}%
,b_{n}\right\}  $ satisfies for some $\alpha\geq1$%
\[
a_{n},\text{ }\left\vert b_{n}\right\vert =o\left(  n^{1/\alpha}\right)
\text{ \ as }n\rightarrow\infty\text{.}%
\]
We have to show%
\[
\int_{-\infty}^{\infty}e^{c\left\vert \lambda\right\vert ^{\alpha}}\sigma
_{+}\left(  d\lambda\right)  <\infty\ \ \ \ \text{for any }c>0\text{.}%
\]
(\ref{4-10}) implies for any $\varepsilon>0$ there exists a positive integer
$L$ such that%
\[
a_{n},\text{ }\left\vert b_{n}\right\vert \leq\varepsilon n^{1/\alpha}\text{
\ for any }n\geq L
\]
holds. Note under (\ref{4-10}) the boundary $+\infty$ for $H_{+}$ is of limit
point type due to Lemma \ref{a-4}. Hence, if we restrict $H_{+}$ on
$\mathbb{Z}_{\geq L}=\left\{  L,L+1,L+2,\cdots\right\}  $, Lemma \ref{l4-2}
says that the spectral measure $\sigma_{+}^{L}$ for $\left.  H_{+}\right\vert
_{\mathbb{Z}_{\geq L}}$ satisfies%
\[
\int_{-\infty}^{\infty}e^{c\left\vert \lambda\right\vert ^{\alpha}}\sigma
_{+}^{L}\left(  d\lambda\right)  <\infty\ \ \ \ \text{for any }c<\left(
9\varepsilon^{2}\right)  ^{-\alpha/2}\left(  \dfrac{2}{\alpha}\right)
\text{.}%
\]
Let $m_{+}^{L}$ be the Weyl function of $\left.  H_{+}\right\vert
_{\mathbb{Z}_{\geq L}}$, namely%
\[
m_{+}^{L}\left(  z\right)  =\int_{-\infty}^{\infty}\dfrac{1}{\lambda-z}%
\sigma_{+}^{L}\left(  d\lambda\right)  \text{.}%
\]
Then, from Lemma \ref{a-1} $m_{+}^{L}$ can be given by%
\[
m_{+}^{L}\left(  z\right)  =-\dfrac{f_{L+1}(z)}{a_{1}f_{L}(z)}\text{.}%
\]
with $f_{n}(z)=c_{n}(z)-m_{+}\left(  z\right)  s_{n}(z)$ (for $c_{n}$, $s_{n}$
refer to Appendix). Therefore, it holds that%
\[
m_{+}^{L}\left(  z\right)  =-\dfrac{c_{L+1}(z)-m_{+}\left(  z\right)
s_{L+1}(z)}{a_{1}\left(  c_{L}(z)-m_{+}\left(  z\right)  s_{L}(z)\right)
}\text{,}%
\]
hence%
\begin{equation}
m_{+}\left(  z\right)  =\dfrac{c_{L}(z)m_{+}^{L}\left(  z\right)  +c_{L+1}%
(z)}{a_{1}\left(  s_{L}(z)m_{+}^{L}\left(  z\right)  +s_{L+1}(z)\right)
}=a_{1}^{-1}z\dfrac{c_{L}(z)}{s_{L}(z)}\dfrac{z^{-2}m_{+}^{L}\left(  z\right)
+z^{-2}\dfrac{c_{L+1}(z)}{c_{L}(z)}}{z^{-1}m_{+}^{L}\left(  z\right)
+z^{-1}\dfrac{s_{L+1}(z)}{s_{L}(z)}}\text{.} \label{4-12}%
\end{equation}
Observe $c_{L}$ and $s_{L}$ are polynomials of degree $L-2$ and $L-1$
respectively. Hence $c_{L}/s_{L}$, $c_{L+1}/c_{L}$ and $s_{L+1}/s_{L}$ have
expansions at $\infty$
\[
z\dfrac{c_{L}(z)}{s_{L}(z)}=\sum\limits_{k\geq0}\alpha_{k}z^{-k}\text{,
\ }z^{-2}\dfrac{c_{L+1}(z)}{c_{L}(z)}=\sum\limits_{k\geq1}\beta_{k}%
z^{-k}\text{, \ }z^{-1}\dfrac{s_{L+1}(z)}{s_{L}(z)}=\sum\limits_{k\geq0}%
\gamma_{k}z^{-k}%
\]
with $\alpha_{0}$, $\beta_{1}$, $\gamma_{0}\neq0$ and%
\[
\left\vert \alpha_{k}\right\vert \text{, \ }\left\vert \beta_{k+1}\right\vert
\text{, \ }\left\vert \gamma_{k}\right\vert \leq c_{1}r^{k}\text{ \ for some
}r>0\text{ and any }k\geq0\text{.}%
\]
On the other hand, Lemma \ref{a-4} implies%
\[
m_{+}^{L}\left(  z\right)  =\sum_{1\leq k\leq L-1}\sigma_{k}^{L}%
z^{-k}+O\left(  z^{-L}\right)  \text{ with }\sigma_{k}^{L}\leq c_{1}%
c^{-k/\alpha}\Gamma\left(  k/\alpha+1\right)  \text{ for }k\geq1\text{,}%
\]
where $c$ is of (\ref{4-11}). Since $\Gamma\left(  k/\alpha+1\right)  $ grows
faster than any power as $k\rightarrow\infty$, applying Lemma \ref{a-2} to
(\ref{4-12}) yields%
\[
m_{+}\left(  z\right)  =\sum_{1\leq k\leq L-1}\sigma_{k}z^{-k}+O\left(
z^{-L}\right)  \text{ with }\sigma_{k}\leq c_{2}\left(  c/2\right)
^{-k/\alpha}\Gamma\left(  k/\alpha+1\right)  \text{.}%
\]
Then, from Lemma \ref{a-4} again
\[
\int_{-\infty}^{\infty}e^{c\left\vert \lambda\right\vert ^{\alpha}}\sigma
_{+}\left(  d\lambda\right)  <\infty\ \ \ \text{for any }c<\dfrac{1}{2}\left(
9\varepsilon^{2}\right)  ^{-\alpha/2}\left(  \dfrac{2}{\alpha}\right)
\]
follows. Since $\varepsilon$ can be arbitrarily small, we complete the proof
of Theorem \ref{t2}.\bigskip

We start proving Corollary \ref{c1} by a simple lemma

\begin{lemma}
\label{l4-3}Let $\left\{  X_{n}\left(  \omega\right)  \right\}  _{n\geq1}$ be
identically distributed non-negative random variables satisfying%
\[
\mathbb{E}X_{n}\left(  \omega\right)  <\infty\text{.}%
\]
Then, it holds that for a.e. $\omega$%
\begin{equation}
X_{n}\left(  \omega\right)  =o\left(  n\right)  \text{ \ as }n\rightarrow
\infty\text{.} \label{4-13}%
\end{equation}

\end{lemma}

\begin{proof}
For $\varepsilon>0$ observe%
\begin{align*}
\sum_{k\geq0}P\left(  X_{1}>\varepsilon k\right)   &  =\sum_{k\geq1}k\left(
P\left(  X_{1}>\varepsilon\left(  k-1\right)  \right)  -P\left(
X_{1}>\varepsilon k\right)  \right) \\
&  =\sum_{k\geq1}kP\left(  \varepsilon k\geq X_{1}>\varepsilon\left(
k-1\right)  \right)  \text{,}%
\end{align*}
and%
\begin{align*}
\mathbb{E}X_{1}  &  =\sum_{k\geq1}\mathbb{E}\left(  X_{1};\varepsilon k\geq
X_{1}>\varepsilon\left(  k-1\right)  \right) \\
&  \geq\varepsilon\sum_{k\geq1}\left(  k-1\right)  P\left(  \varepsilon k\geq
X_{1}>\varepsilon\left(  k-1\right)  \right)  \text{.}%
\end{align*}
Hence, the equality of the distributions yields%
\begin{equation}
\sum_{k\geq0}P\left(  X_{k}>\varepsilon k\right)  =\sum_{k\geq0}P\left(
X_{1}>\varepsilon k\right)  <\infty\text{,} \label{4-14}%
\end{equation}
if $\mathbb{E}X_{1}<\infty$. On the other hand, Borel-Cantelli lemma implies
\[
P\left(  \bigcap\limits_{n\geq1}\bigcup\limits_{k\geq n}\left\{
X_{k}>\varepsilon k\right\}  \right)  \leq P\left(  \bigcup\limits_{k\geq
n}\left\{  X_{k}>\varepsilon k\right\}  \right)  \leq\sum\limits_{k\geq
n}P\left(  X_{k}>\varepsilon k\right)  \text{.}%
\]
Since the right side tends to $0$ as $n\rightarrow\infty$ by (\ref{4-14}), the
probability of the left side is $0$ for any $\varepsilon>0$, which implies
(\ref{4-13}).\bigskip
\end{proof}

Applying Lemma \ref{l4-3} to $a\left(  \theta^{n}\omega\right)  ^{N}$,
$\left\vert b\left(  \theta^{n}\omega\right)  \right\vert ^{N}$ one has%
\[
a\left(  \theta^{n}\omega\right)  \text{, \ }b\left(  \theta^{n}\omega\right)
=o\left(  \left\vert n\right\vert ^{1/N}\right)  \text{ \ as \ }%
n\rightarrow\pm\infty\text{,}%
\]
and one can apply Theorem \ref{t2} to construct Toda flow on $Q_{N}$ starting
from $q_{\omega}=\left\{  a\left(  \theta^{n}\omega\right)  ,b\left(
\theta^{n}\omega\right)  \right\}  _{n\in\mathbb{Z}}$for any $N\geq1$.

If $\left\{  a\left(  \theta^{n}\omega\right)  ,b\left(  \theta^{n}%
\omega\right)  \right\}  _{n\in\mathbb{Z}}$ are independent random sequences
with distributions%
\[
\left\{
\begin{array}
[c]{l}%
P\left(  a\left(  \theta^{n}\omega\right)  <x\right)  =\dfrac{2}{\sigma
\Gamma\left(  \nu/2\right)  }\int_{0}^{x}e^{-y^{2}/\sigma^{2}}y^{\nu
-1}dy\text{ \ }(a\left(  \theta^{n}\omega\right)  ^{2}\text{ }\sim\chi
^{2}\text{-distribution})\\
P\left(  b\left(  \theta^{n}\omega\right)  <x\right)  =\dfrac{1}{\sqrt
{2\pi\sigma^{2}}}\int_{-\infty}^{x}e^{-\left(  y-m\right)  ^{2}/2\sigma^{2}%
}dy\text{ \ }\left(  \sim\mathcal{N}\left(  m,\sigma^{2}\right)  \right)
\end{array}
\right.
\]
for $\sigma>0$, $\nu>0$, $m\in\mathbb{R}$, then, certainly $\mathbb{E}\left(
a\left(  \omega\right)  ^{N}+\left\vert b\left(  \omega\right)  \right\vert
^{N}\right)  <\infty$ holds for any $N\geq1$, and one can construct Toda flow
starting from these random sequences.

The invariance of the probability measures induced by $\left\{  a\left(
\theta^{n}\omega\right)  ,b\left(  \theta^{n}\omega\right)  \right\}
_{n\in\mathbb{Z}}$ can be verified by confirming it on periodic lattice. For a
positive integer $L$ let%
\[
\Omega_{L}=\left\{  \left\{  a_{n},b_{n}\right\}  _{n\in\mathbb{Z}}\text{;
\ }a_{n}>0\text{, }b_{n}\in\mathbb{R}\text{, }a_{n+L}=a_{n}\text{, }%
b_{n+L}=b_{n}\text{ for any }n\in\mathbb{Z}\right\}  \text{.}%
\]
For $\left\{  a_{n},b_{n}\right\}  _{n\in\mathbb{Z}}\in\Omega_{L}$ define a
periodic Jacobi operator:%
\[
\left(  Hu\right)  _{n}=a_{n+1}u_{n+1}+a_{n}u_{n-1}+b_{n}u_{n}\text{ for
}u\text{ satisfying }u_{n+L}=u_{n}\text{,}%
\]
and for $k\geq1$
\[
I_{k}=\mathrm{tr}\left(  H^{k}\right)  \text{.}%
\]
The original Toda lattice has a finite dimensional Hamiltonian structure on
periodic lattice and it has infinitely many invariants $\left\{
I_{k}\right\}  _{k\geq1}$. The original variables $\left\{  q_{i}%
,p_{i}\right\}  $ are related to $\left\{  a_{i},b_{i}\right\}  $ by%
\[
a_{i}=\dfrac{1}{2}e^{-\left(  q_{i}-q_{i-1}\right)  /2}\text{, \ \ \ }%
b_{i}=-\dfrac{1}{2}p_{i}\text{.}%
\]
Since%
\[
\partial_{p_{i}}f=-\dfrac{1}{2}\partial_{b_{i}}f\text{, \ \ }\partial_{q_{i}%
}f=-\dfrac{1}{2}a_{i}\partial_{a_{i}}f+\dfrac{1}{2}a_{i+1}\partial_{a_{i+1}%
}f\text{,}%
\]
if $f$ is a smooth function of $\left\{  a_{i},b_{i}\right\}  $, the Poisson
bracket in the variables $\left\{  q_{i},p_{i}\right\}  $%
\[
\left\{  f,g\right\}  =\sum_{1\leq i\leq L}\left(  \partial_{q_{i}}%
f\partial_{p_{i}}g-\partial_{q_{i}}g\partial_{p_{i}}f\right)
\]
turns to%
\[
\left\{  f,g\right\}  =\dfrac{1}{4}\sum_{1\leq i\leq L}\left(  a_{i}\left(
\partial_{a_{i}}f\partial_{b_{i}}g-\partial_{a_{i}}g\partial_{b_{i}}f\right)
-a_{i+1}\left(  \partial_{a_{i+1}}f\partial_{b_{i}}g-\partial_{a_{i+1}%
}g\partial_{b_{i}}f\right)  \right)  \text{,}%
\]
if $f$, $g$ are functions of $\left\{  a_{i},b_{i}\right\}  $. The invariants
$\left\{  I_{k}\right\}  _{k\geq1}$ are polynomials of $\left\{  a_{i}%
,b_{i}\right\}  $ and they satisfy%
\[
\left\{  I_{k},I_{\ell}\right\}  =0\text{ \ for any }k\text{, }\ell
\geq1\text{.}%
\]
There is another invariant $I_{0}$:%
\[
I_{0}=a_{1}a_{2}\cdots a_{L}\text{, }%
\]
since one can verify%
\[
\left\{  I_{k},I_{0}\right\}  =\dfrac{1}{4}\sum_{1\leq i\leq L}\left(
I_{0}\partial_{b_{i}}I_{k}-I_{0}\partial_{b_{i}}I_{k}\right)  =0\text{ \ due
to }a_{i}\partial_{a_{i}}I_{0}=I_{0}\text{.}%
\]
Then, if we consider an ODE%
\begin{equation}
\partial_{t}a_{i}=\dfrac{1}{4}a_{i}\left(  \partial_{b_{i}}I_{k}%
-\partial_{b_{i-1}}I_{k}\right)  \text{, \ \ }\partial_{t}b_{i}=\dfrac{1}%
{4}\left(  a_{i+1}\partial_{a_{i+1}}I_{k}-a_{i}\partial_{a_{i}}I_{k}\right)
\text{,} \label{4-15}%
\end{equation}
one has for any smooth function $f$ of $\left\{  a_{i},b_{i}\right\}  $
\[
\partial_{t}f=\sum_{1\leq i\leq L}\left(  \partial_{a_{i}}f\partial_{t}%
a_{i}+\partial_{b_{i}}f\partial_{t}b_{i}\right)  =\left\{  f,I_{k}\right\}
\text{,}%
\]
which shows that for any positive smooth positive functions $\left\{  \rho
_{k}\right\}  _{1\leq k\leq K}$ on $\mathbb{R}$ and $\rho_{0}$ on
$\mathbb{R}_{+}$
\[
\left(  \prod\limits_{0\leq k\leq K}\rho_{k}\left(  I_{k}\right)  \right)
da_{1}da_{2}\cdots da_{L}db_{1}db_{2}\cdots db_{L}%
\]
provides an invariant measure for the ODE (\ref{4-15}). This ODE for $k=2$
generates the TODA lattice. If we choose $\rho_{k}\left(  x\right)
=e^{c_{k}x}$ and $\rho_{0}\left(  x\right)  =x^{\nu-1}$ with $c_{k}$, $\nu
\in\mathbb{R}$, the resulting measures are called (generalized) Gibbs
measures. A few terms of $I_{k}$ are%
\[
I_{1}=\sum_{1\leq j\leq L}b_{j}\text{, \ }I_{2}=\sum_{1\leq j\leq L}\left(
2b_{j}^{2}+4a_{j}^{2}\right)  \text{, \ }I_{3}=\sum_{1\leq j\leq L}\left(
3a_{j+1}^{2}+3a_{j}^{2}+b_{j}^{2}\right)  b_{j}\text{,}%
\]
hence Gibbs measures $\mu$ using $I_{0}$, $I_{1}$, $I_{2}$ are%
\[
\mu\equiv\Sigma^{-1}\exp\left(  \sum_{1\leq j\leq L}\left(  c_{1}b_{j}%
-c_{2}\left(  b_{j}^{2}+2a_{j}^{2}\right)  \right)  \right)  a_{1}^{\nu
-1}\cdots a_{L}^{\nu-1}da_{1}\cdots da_{L}db_{1}\cdots db_{L}\text{,}%
\]
where $c_{1}\in\mathbb{R}$, $\nu$, $c_{2}>0$ and the normalization constant
$\Sigma$.

\begin{lemma}
\label{l4-4}The probability measure $\mu$ on $\Omega_{L}$ is invariant for any
$g\in\Gamma_{N}^{\operatorname{real}}$.
\end{lemma}

\begin{proof}
Since the equations (\ref{4-15}) are equivalent to those obtained by the Lax
equations in Theorem \ref{t1}, the previous argument implies that $\mu$ is
invariant under $g=e^{p}$ with real polynomial $p$. Assume $g=e^{h}\in
\Gamma_{N}^{\operatorname{real}}\left(  D_{+}^{1}\right)  $ for some $D_{+}$.
Fix $q\in\Omega_{L}$. Then, the associated Jacobi operator has a bounded
spectrum $\Sigma$ in $\mathbb{R}$. Then, we can assume the curve $C_{1}$ to be
a bounded and simple curve surrouding $\phi^{-1}\left(  \Sigma\right)  $ and
contained in the above $D_{+}^{1}$. Let $K$ be a compact set with boundary
$C_{1}$ and $p_{n}$ be a sequence of polynomials approximating $h$ uniformly
on $K$, which is possible by Runge's theorem. The resulting Toda map
\textrm{Toda}$\left(  e^{p_{n}}\right)  q$, \textrm{Toda}$\left(
e^{h}\right)  q\in\Omega_{L}$ generate Jacobi operators with the same spectrum
$\Sigma$. Since the coefficients are uniformly bounded, without difficulty,
one can show the convergence $\mathrm{d}\left(  \mathrm{Toda}\left(  e^{p_{n}%
}\right)  q,\mathrm{Toda}\left(  e^{h}\right)  q\right)  \rightarrow0$. Let
$f$ be a bounded continuous function on $Q_{N}$. Then, we have \
\[
\int_{Q_{N}}f\left(  q\right)  \mu\left(  dq\right)  =\int_{Q_{N}}f\left(
\mathrm{Toda}\left(  e^{p_{n}}\right)  q\right)  \mu\left(  dq\right)
\underset{n\rightarrow\infty}{\rightarrow}\int_{Q_{N}}f\left(  \mathrm{Toda}%
\left(  e^{h}\right)  q\right)  \mu\left(  dq\right)  \text{.}%
\]
Here, note the identity%
\begin{equation}
\mathrm{Toda}\left(  \widetilde{g}\right)  q=\mathrm{Toda}\left(
g^{-1}\right)  q\text{.} \label{4-16}%
\end{equation}
which comes from%
\[
\tau_{\widetilde{g}\boldsymbol{m}}\left(  g_{1}\right)  =\dfrac{\tau
_{\boldsymbol{m}}\left(  \widetilde{g}g_{1}\right)  }{\tau_{\boldsymbol{m}%
}\left(  \widetilde{g}\right)  }=\dfrac{\tau_{\boldsymbol{m}}\left(
\widetilde{g}gg^{-1}g_{1}\right)  }{\tau_{\boldsymbol{m}}\left(  \widetilde
{g}gg^{-1}\right)  }=\dfrac{\tau_{\boldsymbol{m}}\left(  g^{-1}g_{1}\right)
}{\tau_{\boldsymbol{m}}\left(  g^{-1}\right)  }=\tau_{g^{-1}\boldsymbol{m}%
}\left(  g_{1}\right)  \text{.}%
\]
Thus, for $g_{2}\in\Gamma_{N}^{\operatorname{real}}\left(  D_{+}^{2}\right)  $
(\ref{4-16}) implies%
\[
\int_{Q_{N}}f\left(  \mathrm{Toda}\left(  g_{2}\right)  q\right)  \mu\left(
dq\right)  =\int_{Q_{N}}f\left(  \mathrm{Toda}\left(  \widetilde{g}_{2}%
^{-1}\right)  q\right)  \mu\left(  dq\right)  =\int_{Q_{N}}f\left(  q\right)
\mu\left(  dq\right)  \text{,}%
\]
since $\widetilde{g}_{2}^{-1}\in\Gamma_{N}^{\operatorname{real}}\left(
D_{+}^{1}\right)  $ holds. The invariance of $\mu$ by the action $g=z^{k}$ is
immediate due to $\mathrm{Toda}\left(  z^{k}\right)  q$ is nothing but the
shift of the coefficients by $k$, which completes the proof.\bigskip
\end{proof}

The proof of the corollary is now clear. The variables $\left\{  a_{j}%
,b_{j}\right\}  _{1\leq j\leq L}$ can be regarded as independent random
variables with distributions%
\[
\left\{
\begin{array}
[c]{l}%
P\left(  a_{j}\left(  \omega\right)  <x\right)  =\dfrac{2}{\sigma\Gamma\left(
\nu/2\right)  }\int_{0}^{x}e^{-y^{2}/\sigma^{2}}y^{\nu-1}dy\text{ \ (}%
\nu>0\text{)}\\
P\left(  b_{j}\left(  \omega\right)  <x\right)  =\dfrac{1}{\sqrt{2\pi
\sigma^{2}}}\int_{-\infty}^{x}e^{-\left(  y-m\right)  ^{2}/2\sigma^{2}%
}dy\text{ \ }\left(  \sim\mathcal{N}\left(  m,\sigma^{2}\right)  \right)
\end{array}
\text{.}\right.
\]
Letting the period $L\rightarrow\infty$, one has the conclusion of the corollary.

The random Jacobi operator with coefficients $\left\{  a_{j}\left(
\omega\right)  ,b_{j}\left(  \omega\right)  \right\}  $ was first employed by
I. Dumitriu - A. Edelman \cite{d} to have a random Jacobi matrix having the
same distributions of the eigen-values as that of Gaussian orthogonal
ensembles (GOE). This invariant measures also were used by H. Spohn \cite{sp}
to explain a chaotic behavior of fluid.

\section{Appendix}

\subsection{Weyl disk}

The notion of Weyl disk was introduced by H. Weyl to investigate boundary
value problems for Sturm-Liouville differential operators. In this section we
give its definition for Jacobi operators and show several properties which
will be necessary for our purpose.

Let $c_{n}\left(  z\right)  $, $s_{n}(z)$ be the solutions to%
\begin{gather*}
a_{n+1}f_{n+1}\left(  z\right)  +a_{n}f_{n-1}\left(  z\right)  +b_{n}%
f_{n}\left(  z\right)  =zf_{n}\left(  z\right)  \text{ for }n\geq1\text{,
}z\in\mathbb{C}\\
c_{0}\left(  z\right)  =s_{1}\left(  z\right)  =1\text{, \ }c_{1}\left(
z\right)  =s_{0}\left(  z\right)  =0\text{.}%
\end{gather*}
Then, its equivalent matrix form is%
\[
\left(
\begin{array}
[c]{cc}%
c_{n} & s_{n}\\
c_{n+1} & s_{n+1}%
\end{array}
\right)  =\left(
\begin{array}
[c]{cc}%
0 & 1\\
-\dfrac{a_{n}}{a_{n+1}} & \dfrac{z-b_{n}}{a_{n+1}}%
\end{array}
\right)  \left(
\begin{array}
[c]{cc}%
c_{n-1} & s_{n-1}\\
c_{n} & s_{n}%
\end{array}
\right)  \text{,}%
\]
hence%
\begin{equation}
\det\left(
\begin{array}
[c]{cc}%
c_{n} & s_{n}\\
c_{n+1} & s_{n+1}%
\end{array}
\right)  =\dfrac{a_{n}}{a_{n+1}}\det\left(
\begin{array}
[c]{cc}%
c_{n-1} & s_{n-1}\\
c_{n} & s_{n}%
\end{array}
\right)  =\cdots=\dfrac{a_{1}}{a_{n+1}} \label{a1}%
\end{equation}
hold for any $n\geq0$. With these $\left\{  c_{n},s_{n}\right\}  $ set%
\[
D_{L}(z)=\left\{  m\in\mathbb{C}_{+}\text{; }\sum_{1\leq n\leq L}\left\vert
c_{n}(z)-a_{1}^{-1}ms_{n}(z)\right\vert ^{2}\leq\dfrac{\operatorname{Im}%
m}{\operatorname{Im}z}\right\}  \text{.}%
\]
Then, $D_{L}(z)$ is a disk in $\mathbb{C}_{+}$, although it might be empty.

To compute its radius let $\left\{  f_{n}(z)\right\}  _{n\geq0}$ be a solution
to%
\[
a_{n+1}f_{n+1}\left(  z\right)  +a_{n}f_{n-1}\left(  z\right)  +b_{n}%
f_{n}\left(  z\right)  =zf_{n}\left(  z\right)  \text{ for }n\geq1\text{,}%
\]
where $z$ moves in a domain $D_{+}$ in $\mathbb{C}_{+}$. Then, for $z_{1}$,
$z_{2}\in D_{+}$ it holds that%
\begin{gather}
\left(  z_{1}-\overline{z}_{2}\right)  \sum_{1\leq n\leq L}f_{n}\left(
z_{1}\right)  \overline{f_{n}\left(  z_{2}\right)  }=a_{1}\left(  f_{0}\left(
z_{1}\right)  \overline{f_{1}\left(  z_{2}\right)  }-f_{1}\left(
z_{1}\right)  \overline{f_{0}\left(  z_{2}\right)  }\right) \nonumber\\
-a_{L+1}\left(  f_{L}\left(  z_{1}\right)  \overline{f_{L+1}\left(
z_{2}\right)  }-f_{L+1}\left(  z_{1}\right)  \overline{f_{L}\left(
z_{2}\right)  }\right)  \text{.} \label{a2}%
\end{gather}
Applying (\ref{a2}) to $f_{n}(z)=c_{n}(z)-a_{1}^{-1}ms_{n}(z)$ ($f_{0}%
(z)=1,f_{1}(z)=-a_{1}^{-1}m$), one has%
\[
\operatorname{Im}m\geq\left(  \operatorname{Im}z\right)  \sum_{1\leq n\leq
L}\left\vert c_{n}(z)-a_{1}^{-1}ms_{n}(z)\right\vert ^{2}=\operatorname{Im}%
m-a_{L+1}\operatorname{Im}\left(  f_{L}\left(  z\right)  \overline
{f_{L+1}\left(  z\right)  }\right)  \text{,}%
\]
hence the following equivalences are valid.%
\begin{align}
m  &  \in D_{L}(z)\Longleftrightarrow\operatorname{Im}\left(  f_{L}\left(
z\right)  \overline{f_{L+1}\left(  z\right)  }\right)  \geq0\nonumber\\
&  \Longleftrightarrow\operatorname{Im}\left(  \left(  c_{L}-a_{1}^{-1}%
ms_{L}\right)  \left(  \overline{c_{L+1}-a_{1}^{-1}ms_{L+1}}\right)  \right)
\geq0\nonumber\\
&  \Longleftrightarrow\operatorname{Im}\left(  s_{L}\overline{s_{L+1}}%
a_{1}^{-2}m\overline{m}-s_{L}\overline{c_{L+1}}a_{1}^{-1}m-c_{L}%
\overline{s_{L+1}}a_{1}^{-1}\overline{m}+c_{L}\overline{c_{L+1}}\right)  \geq0
\label{a3}%
\end{align}
Here note from (\ref{a2})%
\begin{align*}
-a_{L+1}\operatorname{Im}\left(  c_{L}\overline{c_{L+1}}\right)   &  =\left(
\operatorname{Im}z\right)  \sum_{1\leq n\leq L}\left\vert c_{n}\right\vert
^{2}>0\text{,}\\
-a_{L+1}\operatorname{Im}\left(  s_{L}\overline{s_{L+1}}\right)   &  =\left(
\operatorname{Im}z\right)  \sum_{1\leq n\leq L}\left\vert s_{n}\right\vert
^{2}>0\text{.}%
\end{align*}
For simplicity set%
\[
A=-s_{L}\overline{s_{L+1}}\text{, \ }B=s_{L}\overline{c_{L+1}}\text{,
\ }C=c_{L}\overline{s_{L+1}}\text{, \ }D=-c_{L}\overline{c_{L+1}}\text{.}%
\]
Then, $\operatorname{Im}A>.0$, $\operatorname{Im}D>0$ hold and we have%
\begin{align*}
0  &  \geq\operatorname{Im}\left(  a_{1}^{-2}Am\overline{m}+a_{1}^{-1}%
Bm+a_{1}^{-1}C\overline{m}+D\right) \\
&  =a_{1}^{-2}m\overline{m}\operatorname{Im}A+\operatorname{Im}D+a_{1}%
^{-1}\dfrac{B}{2i}m-a_{1}^{-1}\dfrac{\overline{B}}{2i}\overline{m}+a_{1}%
^{-1}\dfrac{C}{2i}\overline{m}-a_{1}^{-1}\dfrac{\overline{C}}{2i}m\\
&  =\left\vert a_{1}^{-1}\sqrt{\operatorname{Im}A}m+E\right\vert ^{2}-F
\end{align*}
with%
\[
E=\dfrac{C-\overline{B}}{2i\sqrt{\operatorname{Im}A}}\text{, \ \ \ }%
F=\dfrac{\left\vert C-\overline{B}\right\vert ^{2}}{4\operatorname{Im}%
A}-D\text{.}%
\]
To compute $F$ we observe%
\begin{align*}
F  &  =\dfrac{\left\vert c_{L}\overline{s_{L+1}}-\overline{s_{L}%
\overline{c_{L+1}}}\right\vert ^{2}-4\left(  -\operatorname{Im}s_{L}%
\overline{s_{L+1}}\right)  \left(  -\operatorname{Im}\left(  c_{L}%
\overline{c_{L+1}}\right)  \right)  }{4\left(  -\operatorname{Im}%
s_{L}\overline{s_{L+1}}\right)  }\\
&  =\dfrac{\left\vert c_{L}s_{L+1}-s_{L}c_{L+1}\right\vert ^{2}}{4\left(
-\operatorname{Im}s_{L}\overline{s_{L+1}}\right)  }=\dfrac{1}{4a_{L+1}\left(
-\operatorname{Im}s_{L}\overline{s_{L+1}}\right)  }\text{,}%
\end{align*}
where we have used (\ref{a1}). Then, (\ref{a3}) implies $m\in D_{L}(z)$ is
equivalent to%
\begin{equation}
\left\vert a_{1}^{-1}m+\dfrac{E}{\sqrt{\operatorname{Im}A}}\right\vert
\leq\dfrac{1}{2\sqrt{a_{L+1}}\left\vert \operatorname{Im}s_{L}\overline
{s_{L+1}}\right\vert }\text{,} \label{a4}%
\end{equation}
which shows that $D_{L}(z)$ is a non-empty disk in $\mathbb{C}_{+}$. The disk
$D_{L}(z)$ is called \textbf{Weyl disk}. Clearly $D_{L}(z)$ is decreasing as
$L\rightarrow\infty$, and if its limit reduces to a single point, the Jacobi
operator is called of \textbf{limit point type} at the boundary $+\infty$. It
is known that $D_{\infty}(z)$ is a single point simultaneously for
$z\in\mathbb{C}_{+}$.

Summing up the above argument we have

\begin{lemma}
\label{a-1}The radius of the disk $D_{L}(z)$ is
\[
\dfrac{a_{1}}{2\sqrt{a_{L+1}}\left\vert \operatorname{Im}s_{L}\left(
z\right)  \overline{s_{L+1}\left(  z\right)  }\right\vert }\text{.}%
\]
For a fixed $z\in D_{+}$ let $\left\{  f_{n}\right\}  _{n\geq0,}$ be a
solution to%
\[
a_{n+1}f_{n+1}+a_{n}f_{n-1}+b_{n}f_{n}=zf_{n}\text{ \ for }n\geq1
\]
and assume $\operatorname{Im}\left(  f_{L+1}/f_{L}\right)  <0$. Then,
$-f_{1}/\left(  a_{1}f_{0}\right)  \in D_{L}(z)$ holds.
\end{lemma}

Suppose $D_{\infty}(z)$ consists of one point $m_{+}\left(  z\right)  $. Then,
by its definition of $D_{\infty}(z)$ we have $m_{+}\left(  z\right)
\in\mathbb{C}_{+}$, and the analyticity of $m_{+}\left(  z\right)  $ on
$\mathbb{C}_{+}$ is also known. Therefore, $m_{+}\left(  z\right)  $ is a
Herglotz function. Moreover, (\ref{a4}) implies that
\[
m_{+}\left(  z\right)  =\sum_{1\leq k\leq L-1}\sigma_{k}z^{-k}+O\left(
z^{-L}\right)  \text{ \ as }z\rightarrow\infty\text{ \ on }i\mathbb{R}%
\]
for any $L\geq1$, hence the measure $\sigma_{+}$ on $\mathbb{R}$ representing
$m_{+}\left(  z\right)  $ has any moment. Therefore, $m_{+}\left(  z\right)  $
can be expressed as%
\[
m_{+}\left(  z\right)  =\int_{-\infty}^{\infty}\dfrac{\sigma_{+}\left(
d\lambda\right)  }{\lambda-z}\text{.}%
\]
$m_{+}$ and $\sigma_{+}$ are called the (right) Weyl function and the spectral
measure respectively. Similarly, if we consider the operator $H_{q}$ on
$\mathbb{Z}_{-}$, we have the (left) Weyl function $m_{-}$ and its spectral
measure $\sigma_{-}$. If the boundaries $\pm\infty$ are of the limit point
type, the operator $H_{q}$ can be realized as a self-adjoint operator on
$\ell^{2}\left(  \mathbb{Z}\right)  $ and the kernel of the resolvent operator
$\left(  H_{q}-z\right)  ^{-1}$ is represented by%
\begin{align*}
&  \left(  H_{q}-z\right)  ^{-1}\left(  j,k\right) \\
&  =-\dfrac{\left(  c_{j}(z)+a_{1}^{-1}\left(  a_{0}^{2}m_{-}\left(  z\right)
+z-b_{0}\right)  s_{j}(z)\right)  \left(  c_{k}(z)-a_{1}m_{+}\left(  z\right)
s_{k}(z)\right)  }{a_{1}^{2}m_{+}\left(  z\right)  +a_{0}^{2}m_{-}\left(
z\right)  +z-b_{0}}%
\end{align*}
\ for $j\leq k$. Especially, one has%
\begin{equation}
\left(  H_{q}-z\right)  ^{-1}\left(  0,0\right)  =\dfrac{-1}{a_{1}^{2}%
m_{+}\left(  z\right)  +a_{0}^{2}m_{-}\left(  z\right)  +z-b_{0}}\text{.}
\label{a5}%
\end{equation}
Refer \cite{t}, \cite{z} for the detail.

\subsection{Some properties of sequences growing rapidly}

We consider formal power series%
\[
a\left(  z\right)  =\sum_{k\geq0}a_{k}z^{k}\text{, \ \ }b\left(  z\right)
=\sum_{k\geq0}b_{k}z^{k}\text{.}%
\]
We do not care about the convergence of this series. One can define a product
$f(z)g(z)$ and the coefficients for this product can be obtained by%
\begin{equation}
\left(  a\ast b\right)  _{k}\equiv\sum_{0\leq j\leq k}a_{j}b_{k-j}\text{.}
\label{a6}%
\end{equation}
If $a_{0}=1$, one can define a formal power series $b(z)=a\left(  z\right)
^{-1}$ so that $a(z)b(z)=1$ holds. The coefficients $b_{k}$ for $a^{-1}$ is
obtained by solving%
\[
b_{0}=1\text{, \ \ }\sum_{0\leq j\leq k}a_{j}b_{k-j}=0\text{ \ for }%
k\geq1\text{.}%
\]
However, a more practical way to get $b_{k}$ is to use the expansion of
$\left(  1+x\right)  ^{-1}$:%
\[
\left(  1+\sum_{k\geq1}a_{k}z^{k}\right)  ^{-1}=1-\sum_{k\geq1}a_{k}%
z^{k}+\left(  \sum_{k\geq1}a_{k}z^{k}\right)  ^{2}-\left(  \sum_{k\geq1}%
a_{k}z^{k}\right)  ^{3}+\cdots\text{,}%
\]
and calculate the coefficients of $z^{k}$. Namely, one has for $k\geq1$%
\begin{equation}
b_{k}=-a_{k}+\sum_{\substack{j_{1}+j_{2}=k\\j_{1},j_{2}\geq1}}a_{j_{1}%
}a_{j_{2}}-\sum_{\substack{j_{1}+j_{2}+j_{3}=k\\j_{1},j_{2},j_{3}\geq
1}}a_{j_{1}}a_{j_{2}}a_{j_{3}}+\cdots+\left(  -1\right)  ^{k}a_{1}^{k}\text{.}
\label{a7}%
\end{equation}
For $c$, $\alpha>0$ we are interested in a class of coefficients $a=\left\{
a_{k}\right\}  _{k\geq0}$:
\[
\mathcal{E}_{c,\alpha}=\left\{  a=\left\{  a_{k}\right\}  _{k\geq0}\text{;
\ }\left\vert a_{k}\right\vert \leq c^{k}\Gamma\left(  \alpha k+1\right)
\text{ for any }k\geq1\right\}  \text{.}%
\]
In the sequel an inequality on Gamma function:%
\begin{equation}
\Gamma\left(  x+1\right)  \Gamma\left(  y+1\right)  \leq\Gamma\left(
x+y+1\right)  \text{ \ for }x\text{, }y\geq0 \label{a8}%
\end{equation}
is effective to simplify arguments. This can be deduced inductively from the
identity $\Gamma\left(  x+1\right)  =x\Gamma\left(  x\right)  $ and an
inequality for Beta function obtained by Iv\'{a}dy \cite{i}:%
\[
B\left(  x,y\right)  \leq\dfrac{x+y}{xy}\dfrac{1}{1+xy}\text{ \ for }x\text{,
}y\in\left(  0,1\right)  \text{.}%
\]

\begin{lemma}
\label{a-2}(i) For $a\in\mathcal{E}_{c_{1},\alpha}$, $b\in\mathcal{E}%
_{c_{2},\alpha}$ it holds that%
\[
\left\vert \left(  a\ast b\right)  _{k}\right\vert \leq\left(  \left\vert
a_{0}\right\vert c_{2}^{k}+\left\vert b_{0}\right\vert c_{1}^{k}+c_{1}%
c_{2}\dfrac{c_{2}^{k-1}-c_{1}^{k-1}}{c_{2}-c_{1}}\right)  \Gamma\left(  \alpha
k+1\right)  \text{ \ for }k\geq1\text{.}%
\]
(ii) For $a\in\mathcal{E}_{c,\alpha}$ if $a_{0}=1$, then for $k\geq1$%
\[
\left\vert \left(  a^{-1}\right)  _{k}\right\vert \leq\dfrac{1}{2}\left(
2c\right)  ^{k}\Gamma\left(  \alpha k+1\right)  \text{.}%
\]

\end{lemma}

\begin{proof}
(i) obeys from (\ref{a6}) and (\ref{a8}):%
\begin{align*}
\left\vert \left(  a\ast b\right)  _{k}\right\vert  &  \leq\sum_{0\leq j\leq
k}\left\vert a_{j}b_{k-j}\right\vert \\
&  \leq\left(  \left\vert a_{0}\right\vert c_{2}^{k}+\left\vert b_{0}%
\right\vert c_{1}^{k}\right)  \Gamma\left(  \alpha k+1\right) \\
&  \text{ \ \ \ \ \ \ \ }+\sum_{1\leq j\leq k-1}c_{1}^{j}c_{2}^{k-j}%
\Gamma\left(  \alpha j+1\right)  \Gamma\left(  \alpha\left(  k-j\right)
+1\right) \\
&  \leq\left(  \left\vert a_{0}\right\vert c_{2}^{k}+\left\vert b_{0}%
\right\vert c_{1}^{k}+c_{1}c_{2}\dfrac{c_{2}^{k-1}-c_{1}^{k-1}}{c_{2}-c_{1}%
}\right)  \Gamma\left(  \alpha k+1\right)  \text{.}%
\end{align*}
For $k\geq1$ (\ref{a7}) yields%
\begin{align*}
\left\vert \left(  a^{-1}\right)  _{k}\right\vert  &  \leq\left\vert
a_{k}\right\vert +\sum_{j_{1}+j_{2}=k}\left\vert a_{j_{1}}a_{j_{2}}\right\vert
+\sum_{j_{1}+j_{2}+j_{3}=k}\left\vert a_{j_{1}}a_{j_{2}}a_{j_{3}}\right\vert
+\cdots+\left\vert a_{1}\right\vert ^{k}\\
&  \leq c^{k}\Gamma\left(  \alpha k+1\right)  +\sum_{\substack{j_{1}%
+j_{2}=k\\j_{1},j_{2}\geq1}}c^{j_{1}+j_{2}}\Gamma\left(  \alpha j_{1}%
+1\right)  \Gamma\left(  \alpha j_{2}+1\right)  +\cdots+c^{k}\Gamma\left(
\alpha+1\right)  ^{k}\\
&  \leq c^{k}\Gamma\left(  \alpha k+1\right)  \sum_{1\leq j\leq k}\dbinom
{k-1}{k-j}=\dfrac{1}{2}\left(  2c\right)  ^{k}\Gamma\left(  \alpha k+1\right)
\text{.}%
\end{align*}

\end{proof}

\begin{lemma}
\label{a-3}Let $\left\{  \mu_{k}\right\}  _{k\geq1}$, $\left\{  \sigma
_{k}\right\}  _{k\geq1}$ be 2 sequences of complex numbers connected with%
\[
\sum_{1\leq k\leq L-1}\mu_{k}\zeta^{-k}=\sum_{1\leq k\leq L-1}\sigma
_{k}\left(  \zeta+\zeta^{-1}\right)  ^{-k}+O\left(  \zeta^{-L}\right)  \text{
as }\zeta\rightarrow\infty
\]
for any $L\geq1$. Assume $\left\{  \mu_{k}\right\}  _{k\geq1}\in
\mathcal{E}_{c,\alpha}$. Then one has%
\begin{equation}
\left\vert \sigma_{k}\right\vert \leq2^{k}\left(  kc\Gamma\left(
\alpha+1\right)  +kc^{k}\Gamma\left(  \alpha k+1\right)  \right)  \text{.}
\label{a9}%
\end{equation}

\end{lemma}

\begin{proof}
Set $z=\left(  \zeta+\zeta^{-1}\right)  /2$. Then, we have%
\[
\zeta^{-1}=z^{-1}\varphi\left(  z^{-2}\right)
\]
\ with
\[
\varphi\left(  z\right)  =\dfrac{1-\sqrt{1-z}}{z}=\sum_{j\geq1}p_{j}%
z^{j-1}\text{, \ }p_{j}=\dfrac{2^{-2j}\left(  2j\right)  !}{\left(  j!\right)
^{2}\left(  2j-1\right)  }\text{.}%
\]
It is crucial that%
\[
0<p_{j}<1\text{, \ }\sum_{j\geq1}p_{j}=\varphi\left(  1\right)  =1
\]
holds. Let $\beta_{j}^{\left(  k\right)  }$ be the coefficients defined by%
\[
\varphi\left(  z\right)  ^{k}=\sum_{j\geq0}\beta_{j}^{\left(  k\right)  }%
z^{j}\text{.}%
\]
Then, substituting this expression of $\zeta^{-1}$ yields
\begin{align*}
\sum_{1\leq k\leq L-1}2^{-k}\sigma_{k}z^{k}  &  =\sum_{1\leq k\leq L-1}\mu
_{k}z^{k}\varphi\left(  z^{2}\right)  ^{k}+O\left(  z^{L}\right) \\
&  =\sum_{1\leq i\leq L-1}\mu_{i}z^{i}\sum_{0\leq2j\leq L-1-i}\beta
_{j}^{\left(  i\right)  }z^{2j}+O\left(  z^{L}\right)
\end{align*}
\ as \ $z\rightarrow0$ replacing $z^{-1}$ with $z$. Therefore, one has%
\[
2^{-k}\sigma_{k}=\sum_{0\leq2j\leq k-1}\mu_{k-2j}\beta_{j}^{\left(
k-2j\right)  }\text{,}%
\]
which leads us to%
\[
\left\vert \sigma_{k}\right\vert \leq2^{k}\sum_{0\leq2j\leq k-1}\left\vert
\mu_{k-2j}\right\vert \beta_{j}^{\left(  k-2j\right)  }\leq2^{k}\sum
_{0\leq2j\leq k-1}\left\vert \mu_{k-2j}\right\vert \leq2^{k}\sum_{1\leq j\leq
k}\left\vert \mu_{j}\right\vert \text{,}%
\]
where we have used the fact%
\[
0<\beta_{j}^{\left(  k\right)  }<1\text{ }\ \text{due to }\sum_{j\geq0}%
\beta_{j}^{\left(  k\right)  }=\varphi\left(  1\right)  ^{k}=1\text{.}%
\]
Then, (\ref{a9}) follows from%
\[
I\equiv\sum_{1\leq j\leq k}c^{j}\Gamma\left(  \alpha j+1\right)  =\int
_{0}^{\infty}e^{-t}\sum_{1\leq j\leq k}c^{j}t^{\alpha j}dt\text{.}%
\]
Separating the region of integration into 2 parts $\left\{  ct^{\alpha
}<1\right\}  $, $\left\{  ct^{\alpha}>1\right\}  $, one has%
\[
I\leq k\int_{ct^{\alpha}<1}e^{-t}ct^{\alpha}dt+k\int_{ct^{\alpha}>1}\left(
ct^{\alpha}\right)  ^{k}e^{-t}dt\leq kc\Gamma\left(  \alpha+1\right)
+kc^{k}\Gamma\left(  \alpha k+1\right)  \text{.}%
\]

\end{proof}

\subsection{Exponential integrability of spectral measures}

We know any spectral measure of a Jacobi operator has finite moments of any
degree. Here we characterize an exponential integrability by growing order of
the moments.

\begin{lemma}
\label{a-4}If the moment $\sigma_{k}$ of the spectral measure $\sigma_{+}$
satisfies%
\begin{equation}
\left\vert \sigma_{2k}\right\vert \leq c_{1}c^{-2k/\alpha}\Gamma\left(
2k/\alpha+1\right)  \text{ for any }k\geq1\text{,} \label{a10}%
\end{equation}
then, for any $c^{\prime}<c$ there exists a constant $c_{2}$ depending on
$c_{1}$, $c$, $c^{\prime}$, $\alpha$ such that%
\begin{equation}
\int_{-\infty}^{\infty}e^{c^{\prime}\left\vert \lambda\right\vert ^{\alpha}%
}\sigma_{+}\left(  d\lambda\right)  \leq c_{2}\text{.} \label{a11}%
\end{equation}
Conversely, the property (\ref{a11}) implies
\begin{equation}
\left\vert \sigma_{k}\right\vert \leq\int_{-\infty}^{\infty}\left\vert
\lambda\right\vert ^{k}\sigma_{+}\left(  d\lambda\right)  \leq c_{2}\left(
c^{\prime}\right)  ^{-k/\alpha}\Gamma\left(  k/\alpha+1\right)  \text{ for any
}k\in\mathbb{Z}_{+}\text{. } \label{a12}%
\end{equation}
In this case the associated Jacobi operator has the limit point type boundary
at $+\infty$ if $\alpha\geq1$. In this case the associated Jacobi operator has
the limit point type boundary at $+\infty$ if $\alpha\geq1$.
\end{lemma}

\begin{proof}
Observe%
\[
\int_{-\infty}^{\infty}e^{c^{\prime}\left\vert \lambda\right\vert ^{\alpha}%
}\sigma_{+}\left(  d\lambda\right)  =\sum_{k\geq0}\dfrac{\left(  c^{\prime
}\right)  ^{k}}{k!}\int_{-\infty}^{\infty}\left\vert \lambda\right\vert
^{\alpha k}\sigma_{+}\left(  d\lambda\right)  \text{.}%
\]
Define a non-negative integer $j$ by $j\leq k\alpha/2<j+1$ and set
$t=k\alpha/2-j$. Then,\ $k\alpha=2\left(  1-t\right)  j+2t\left(  j+1\right)
$ and H\"{o}lder's inequality imply%
\begin{align*}
\int_{-\infty}^{\infty}\left\vert \lambda\right\vert ^{\alpha k}\sigma
_{+}\left(  d\lambda\right)   &  \leq\left(  \int_{-\infty}^{\infty}\left\vert
\lambda\right\vert ^{2j}\sigma_{+}\left(  d\lambda\right)  \right)
^{1-t}\left(  \int_{-\infty}^{\infty}\left\vert \lambda\right\vert ^{2\left(
j+1\right)  }\sigma_{+}\left(  d\lambda\right)  \right)  ^{t}\\
&  \leq c_{1}c^{-k}\Gamma\left(  2j/\alpha+1\right)  ^{j+1-k\alpha}%
\Gamma\left(  2\left(  j+1\right)  /\alpha+1\right)  ^{k\alpha-j}\text{.}%
\end{align*}
Stirling's formula $\Gamma\left(  x+1\right)  \sim\sqrt{2\pi x}\left(
x/e\right)  ^{x}$ yields
\[
\Gamma\left(  j/\alpha+1\right)  ^{j+1-k\alpha}\Gamma\left(  \left(
j+1\right)  /\alpha+1\right)  ^{k\alpha-j}\leq c_{3}\Gamma\left(  k+1\right)
\]
for a constant $c_{3}$ depending only on $\alpha$. Therefore, one has%
\[
\int_{-\infty}^{\infty}\left\vert \lambda\right\vert ^{\alpha k}\sigma
_{+}\left(  d\lambda\right)  \leq c_{1}c_{3}c^{-k}k!\text{,}%
\]
which implies%
\[
\int_{-\infty}^{\infty}e^{c^{\prime}\left\vert \lambda\right\vert ^{\alpha}%
}\sigma_{+}\left(  d\lambda\right)  \leq c_{1}c_{3}\sum_{k\geq0}\dfrac{\left(
c^{\prime}\right)  ^{k}c^{-k}k!}{k!}=c_{1}c_{3}\sum_{k\geq0}\left(
\dfrac{c^{\prime}}{c}\right)  ^{k}<\infty\text{.}%
\]
Conversely, noting an inequality%
\[
e^{x}\geq\dfrac{x^{\beta}}{\Gamma\left(  \beta+1\right)  }\text{ \ for any
}\beta\text{, }x\geq0\text{,}%
\]
we have%
\[
\int_{-\infty}^{\infty}e^{c^{\prime}\left\vert \lambda\right\vert ^{\alpha}%
}\sigma_{+}\left(  d\lambda\right)  \geq\dfrac{\left(  c^{\prime}\right)
^{\beta}}{\Gamma\left(  \beta+1\right)  }\int_{-\infty}^{\infty}\left\vert
\lambda\right\vert ^{\alpha\beta}\sigma_{+}\left(  d\lambda\right)  \text{.}%
\]
Choosing $\beta=k\alpha^{-1}$ yields (\ref{a12}).

The last statement obeys from Carleman's criterion on the uniqueness of the
moment problem:%
\[
\sum_{k\geq1}\sigma_{2k}^{-\left(  2k\right)  ^{-1}}\geq\sum_{k\geq1}\left(
c_{1}\right)  ^{-\left(  2k\right)  ^{-1}}c^{1/\alpha}\Gamma\left(
2k/\alpha+1\right)  ^{-\left(  2k\right)  ^{-1}}\geq c_{3}\sum_{k\geq
1}k^{-1/\alpha}=\infty\text{.}%
\]

\end{proof}

\end{document}